\documentclass[a4paper, 11pt]{amsart}
\usepackage{enumerate}
\usepackage{amssymb, amsthm}
\usepackage{mathtools, stmaryrd}
\newenvironment{widthitemize}{\list{\textbullet}{\labelwidth=2ex\leftmargin=2.25ex}}{\endlist}
\allowdisplaybreaks
\def\supertiny{\fontsize{3}{4}\selectfont}
\usepackage{tikz}
\usetikzlibrary{patterns, decorations.pathmorphing, decorations.pathreplacing, decorations.markings}
\numberwithin{equation}{section}
\newtheorem{MainTheorem}{Theorem}

\newtheorem{theorem}{Theorem}[section]
\newtheorem{proposition}[theorem]{Proposition}
\newtheorem{lemma}[theorem]{Lemma}

\newtheoremstyle{customnumberedtheorem}{}{}{\em}{}{\bfseries}{.}{ }{\thmname{#1}\thmnumber{ #3}} 
\theoremstyle{customnumberedtheorem}
\newtheorem{CustomNumberedExample}{Example}
\newtheoremstyle{definitionwithbfnote}{}{}{}{}{\bfseries}{.}{ }{\thmname{#1}\thmnumber{ #2}\thmnote{ (#3)}}
\theoremstyle{definitionwithbfnote}
\newtheorem{definition}[theorem]{Definition}
\newtheorem{remark}[theorem]{Remark}

\newtheorem{example}[theorem]{Example}
\newcommand{\openthickbox}{\leavevmode\hbox to.77778em{%
  \hfil\vrule width.175em
  \vbox to.675em{\hrule width.26em height.175em\vfil\hrule height.175em}%
  \vrule width.175em\hfil}}
\newcommand{\qedple}{\leavevmode\unskip\penalty9999\hbox{}\nobreak\hfill\quad\hbox{\openthickbox}}
\makeatletter
\newcommand{\TeoremaAmbFinalMarcat}[1]{\expandafter\gdef\csname end#1\endcsname{\qedple\@endtheorem}}
\DeclareRobustCommand\vdotsdouble{\raisebox{-7.75\p@}{\vbox{\baselineskip4\p@\lineskiplimit\z@\kern6\p@\hbox{.}\hbox{.}\hbox{.}\hbox{.}\hbox{.}\hbox{.}}}}
\def\@map#1#2[#3]{\mbox{$#1 \colon #2 \longrightarrow #3$}}
\def\map#1#2{\@ifnextchar [{\@map{#1}{#2}}{\@map{#1}{#2}[#2]}}
\def\@chull#1[#2]{\ensuremath{\langle #1 \rangle\mkern-2mu_{_{\mbox{\supertiny$#2$}}}}}
\def\chull#1{\@ifnextchar [{\@chull{#1}}{\ensuremath{\langle #1 \rangle}}}
\newsavebox{\@cont@ent@}\newlength{\@h@t@}\newlength{\@d@t@}%
\newcommand{\spnorm}[2][\empty]{\begingroup\setbox\@cont@ent@\hbox{\ensuremath{#2}}%
  \setlength\@h@t@{\ht\@cont@ent@}\advance\@h@t@\dp\@cont@ent@\advance\@h@t@3.4pt
  \setlength\@d@t@{\dp\@cont@ent@}\advance\@d@t@1.8pt
  \mathopen{\rule[-\@d@t@]{0.9pt}{\@h@t@}\mkern2.8mu\rule[-\@d@t@]{0.45pt}{\@h@t@}}%
  \mkern2.2mu#2\mkern2.3mu%
  \mathclose{\rule[-\@d@t@]{0.45pt}{\@h@t@}\mkern2.8mu\rule[-\@d@t@]{0.9pt}{\@h@t@}}%
  \mathchoice{\ }{\mkern1.75mu}{\mkern1mu}{\mkern0.7mu}
  \ifx#1\empty\else\advance\@d@t@0.75pt\mathchoice{\mkern-4.25mu}{}{}{}\raisebox{-\@d@t@}{$\scriptscriptstyle#1$}\fi%
\endgroup}
\def\@ls#1#2{\@mathmeasure\z@\displaystyle{#2}\@mathmeasure\@ne\displaystyle{#1}\box\@ne\box\z@}
\def\gtso#1{\mathrel{\@ls{_{#1}}{>}}}
\def\geso#1{\mathrel{\@ls{_{#1}}{\ge}}}
\def\ltso#1{\mathrel{<_{#1}}}
\def\leso#1{\mathrel{\le_{#1}}}
\makeatother
\def\Sho{\mbox{\tiny\textup{Sh}}}

\TeoremaAmbFinalMarcat{definition}
\TeoremaAmbFinalMarcat{remark}
\TeoremaAmbFinalMarcat{notation}
\newenvironment{case}[1]{\trivlist \item[\hskip\labelsep{\bfseries #1}]\begin{em}}{\end{em}\endtrivlist}
\newenvironment{autocase}[2][Case]{\case{#1 #2.}}{\endcase}
\newcommand{\N}{\ensuremath{\mathbb{N}}}
\newcommand{\Z}{\ensuremath{\mathbb{Z}}}
\newcommand{\Q}{\ensuremath{\mathbb{Q}}}
\newcommand{\R}{\ensuremath{\mathbb{R}}}
\newcommand{\SI}{\ensuremath{\mathbb{S}^1}}
\newcommand{\C}{\ensuremath{\mathbb{C}}}
\DeclareMathOperator{\bc}{\textsf{\upshape BdCof\,}}
\DeclareMathOperator{\sbc}{\textsf{\upshape StrBdCof\,}}
\DeclareMathOperator{\len}{\textsf{\upshape len}}
\DeclareMathOperator{\Per}{\textsf{\upshape Per}}
\DeclareMathOperator{\Orb}{\textsf{\upshape Orb}}
\DeclareMathOperator{\Rot}{\textsf{\upshape Rot}}
\DeclareMathOperator{\Int}{\textsf{\upshape Int}}
\DeclareMathOperator{\Clos}{\textsf{\upshape Clos}}
\DeclareMathOperator{\Card}{\textsf{\upshape Card}}
\DeclareMathOperator{\Succ}{\textsf{\upshape Succ}}
\DeclareMathOperator{\sbcset}{\textsf{\upshape sBC}}
\DeclareMathOperator{\eexp}{\textsf{\upshape e}}
\newcommand{\dol}[1][1]{\ensuremath{\mathcal{L}_{#1}}}
\newcommand{\succs}[1]{\ensuremath{\Succ\left(#1\right)}}
\newcommand{\set}[2]{\ensuremath{\left\{#1 \,\colon #2\right\}}}
\newcommand{\floor}[1]{\ensuremath{\left\lfloor #1 \right\rfloor}}
\newcommand{\BIclass}[1]{\ensuremath{\llbracket #1\rrbracket}}
\newcommand{\bigBIclass}[1]{\ensuremath{\bigl\llbracket #1\bigr\rrbracket}}
\newcommand{\BIgraph}[1]{\ensuremath{\mathopen{\bigl\langle\mkern-7.8mu\bigl\langle\mspace{-6.5mu}\bigl\langle} #1 \mathclose{\bigr\rangle\mspace{-6.5mu}\bigr\rangle\mkern-7.8mu\bigr\rangle}}}
\newcommand{\emap}[1]{\ensuremath{\eexp(#1)}}
\newcommand{\bigemap}[1]{\ensuremath{\eexp\bigl(#1\bigr)}}
\newcommand{\bigeta}[1]{\ensuremath{\eta\bigl(#1\bigr)}}
\newcommand{\etaemap}[1]{\ensuremath{\eta(\emap{#1})}}
\def\calB{\mathcal{B}}
\newcommand{\SBI}[1][Q]{\ensuremath{\calB(#1)}}
\newcommand{\bigSBI}[1]{\ensuremath{\calB\bigl(#1\bigr)}}
\newcommand{\andq}[1][and]{\ensuremath{\quad\text{#1}\quad}}
\providecommand{\abs}[1]{\ensuremath{\left\lvert#1\right\rvert}}
\catcode`\_=13
\def_#1{\sb{\sb{#1}}}
\newcommand{\modulo}[2]{\{#1\}\sb{#2}}
\newcommand{\evalat}[1]{\bigr\rvert\sb{#1}}
\title[Entropy, periods and transitivity]{Topological entropy, sets of periods and transitivity for graph maps}
\author[Ll. Alsed\`a]{Llu\'{\i}s Alsed\`a}
\address{Departament de Matem\`atiques and Centre de Recerca Matem\`atica,
Edifici Cc,
Universitat Aut\`onoma de Bar\-ce\-lo\-na,
08913 Cerdanyola del Vall\`es,
Barcelona,
Spain}
\email{alseda@mat.uab.cat}
\email{llalseda@crm.cat}

\author[L. Bordignon]{Liane Bordignon}
\address{Departamento de Matem\'atica,
Universidade Federal de S\~{a}o Carlos,
S\~{a}o Carlos, S\~{a}o Paulo,
Brasil}
\email{liane@dm.ufscar.br}

\author[J. Groisman]{Jorge Groisman}
\address{IMERL,
Facultad de Ingenier\'{\i}a,
Universidad de la Rep\'ublica,
Montevideo,
Uruguay}
\email{jorgeg@fing.edu.uy}
\thanks{\textit{Acknowledgements:}\ The authors have been partially supported by MINECO grant numbers MTM2014-52209-C2-1-P and MTM2017-86795-C3-1-P.
        Llu\'{\i}s Alsed\`a acknowledges financial support from the Spanish Ministry of Economy and Competitiveness,
        through the ``Mar\'{\i}a de Maeztu'' Programme for Units of Excellence in R\&D (MDM-2014-0445).
        The second author has been partially supported by CNPq--Brasil.}
\date{June 21, 2018}
\begin{document}
\begin{abstract}
Transitivity, the existence of periodic points and positive
topological entropy can be used to characterize complexity in
dynamical systems. It is known that for graphs that are not trees, for
every $\varepsilon>0,$ there exist (complicate) totally transitive maps
(then with cofinite set of periods) such that the topological entropy
is smaller than $\varepsilon$ (simplicity). First we will show by
means of three examples that for any graph that is not a tree the
relatively simple maps (with small entropy) which are totally
transitive (and hence robustly complicate) can be constructed so that
the set of periods is also relatively simple. To numerically measure
the complexity of the set of periods we introduce a notion of a
\emph{boundary of cofiniteness\/}. Larger boundary of cofiniteness means
simpler set of periods. With the help of the notion of boundary of
cofiniteness we can state precisely what do we mean by extending the
entropy simplicity result to the set of periods: \emph{there exist
relatively simple maps such that the boundary of cofiniteness is
arbitrarily large (simplicity) which are totally transitive (and hence
robustly complicate)\/}. Moreover, we will show that, the above
statement holds for arbitrary continuous degree one circle maps.
\end{abstract}
\maketitle
\section{Introduction}

Transitivity, the existence of infinitely many periods and
positive topological entropy often characterize the complexity in
dynamical systems.
This paper aims at showing that totally transitive maps on graphs,
despite of being complicate in the above sense can have somewhat simple
sets of periods (simplicity with respect to topological entropy was
already known).
To be more precise and to state the main results of the paper we need
to introduce some basic notation.

Let $X$ be a topological space and let $\map{f}{X}$ be a map.
A point $x\in X$ is called a \emph{periodic point of $f$ of period $n$ \/}
if $f^n(x) = x$ and $n$ is the minimum positive integer with this property.
The set of all positive integers $n$ such that $f$ has a periodic point
of period $n$ is denoted by $\Per(f).$
A set of periods is called \emph{cofinite\/} if its complement
(on $\N$) is finite or, equivalently, it contains all positive
integers larger than a given period.

Let $X$ be a compact space and let $\map{f}{X}$ be a continuous map.
The \emph{topological entropy\/} of $f$ is defined as in \cite{AKM} and
denoted by $h(f).$

\begin{definition}
Given a topological space $X,$ a continuous map {\map{f}{X}}
is called \emph{transitive\/} if for every pair of open subsets of $X,$
$U$ and $V,$ there is a positive integer $n$ such that
$f^n(U)\cap V\ne \emptyset.$

It is well known that when $X$ has no isolated point, transitivity is
equivalent to the existence of a dense orbit (see for instance
\cite{KS}).

A map $f$ is called \emph{totally transitive\/} if all
iterates of $f$ are transitive.
\end{definition}

We are interested in totally transitive maps on graphs.
A \emph{(topological) graph\/} is a connected compact Hausdorff space
for which there exists a finite non-empty subset whose complement is
the disjoint union of a finite number subsets, each of them
homeomorphic to an open interval of the real line.
A \emph{tree\/} is a graph without circuits or, equivalently,
a uniquely arcwise connected graph.
A continuous map from a graph to itself will be called a \emph{graph map\/}.

A transitive graph map has positive topological entropy and dense
set of periodic points \cite{Bkh83, Bkh87, ARR2003-survey, adrr2},
with the only exception of an irrational rotation on the circle.
Thus, in view of \cite{BBCDS} every transitive map on a graph is chaotic
in the sense of Devaney (except, again, for an irrational rotation on
the circle).
Moreover, from \cite[Main~Theorem]{adrr} we have

\begin{theorem}\label{theoremtotallytransitivefromadrr}
Let $G$ be a graph and let $f$ be a continuous transitive map from $G$
to itself which has periodic points.
Then the following statements are equivalent:
\begin{enumerate}[(a)]
\item $f$ is totally transitive.
\item $\Per(f)$ is cofinite in $\N.$
\end{enumerate}
\end{theorem}

Thus, \emph{totally transitive maps on graphs are complicate since they
have positive topological entropy, are chaotic in the sense of Devaney
and have cofinite set of periods.}

However, from \cite{arr} we know that for every graph that is not a
tree and for every $\varepsilon > 0,$ there exists a totally transitive
map such that its topological entropy is positive but smaller than
$\varepsilon.$
Thus, \emph{transitive maps on graphs may be relatively simple
because they may have arbitrarily small positive topological
entropy.}

\begin{remark}
This result is valid only for graphs that are not trees since from
\cite{ablm} we know that for any tree $T$ and any transitive map
$\map{f}{T},$
\[
h(f) \ge \frac{\log 2}{\textsf{\upshape E}(T)},
\]
where $\textsf{\upshape E}(T)$ denotes the number of endpoints of $T.$
\end{remark}

The aim of this paper is to study whether the simplicity
phenomenon described above, that happens for the topological entropy,
can be extended to the set of periods. More precisely,
\emph{is it true that when a totally transitive graph map
has small positive topological entropy it also has
simple set of periods (and in particular small ``cofinite part''
of the set of periods)?}

To measure the size of the set of periods and, in particular,
of its ``cofinite part'' we introduce the notion of
\emph{boundary of cofiniteness\/}.
For clarity we will also introduce three auxiliary notions:
\begin{itemize}
\item For every $L \in \N,$ we define
\emph{the set of successors of $L$}, denoted by $\succs{L},$
as the set $\set{k\in \N}{k \ge L}.$

\item Let $f$ be a graph map whose set of periods is cofinite.
The \emph{strict boundary of cofiniteness of $f$},
denoted by $\sbc(f),$ is defined as the smallest positive integer
$n$ such that $\Per(f) \supset \succs{n}.$
Accordingly, the set $\succs{\sbc(f)}$ will be called the
\emph{cofinite part of $\Per(f)$}.

\item Given a graph map $f$ whose set of periods is cofinite
we define the set
\begin{multline*}
   \sbcset(f) := \Bigl\{L \in \Per(f)  \,\colon
        L > 2,\ L-1 \notin \Per (f)\text{ and}\\
        \Card\bigl(\{1,\ldots, L-2\}\cap \Per(f)\bigr) \le 2 \log_2(L-2)
   \Bigr\}.
\end{multline*}
Observe that every $L \in \Per(f),$ $L > 2,$ such that $L-1 \notin \Per(f)$
must satisfy $L \le \sbc(f).$
Hence, $\sbcset(f) \subset \left\{1,2,\dots, \sbc(f)\right\}$ is finite.
\end{itemize}

\begin{definition}[boundary of cofiniteness]\label{SCB}
Let $f$ be a graph map whose set of periods is cofinite.
When $\sbcset(f) \ne \emptyset$ we define the
\emph{boundary of cofiniteness of $f$ \/} as
the number $\bc(f) := \max \sbcset(f).$
Observe that
\begin{widthitemize}
\item $\bc(f)$ is not defined if and only if the set $\sbcset(f)$ is empty.
\item Whenever $\bc(f)$ is defined we have $\bc(f) \le \sbc(f).$
      Moreover $\bc(f) < \sbc(f)$ if and only if $\sbc(f) \notin \sbcset(f).$
\end{widthitemize}
\end{definition}

Now let us see that \emph{larger boundary of cofiniteness implies
simpler set of periods}:
On the one hand, the facts that $\bc(f) \le \sbc(f)$ and
$\succs{\sbc(f)}$ is the  cofinite part of $\Per(f)$
imply that $\bc(f)$ is a lower bound of the beginning of the
cofinite part of $\Per(f)$
(in particular, the larger it is $\bc(f)$ the smaller it is the
cofinite part of $\Per(f)$).
On the other hand,
\[
 \Card\bigl(\{1,\ldots, \bc(f)-2\}\cap \Per(f)\bigr) \le 2 \log_2(\bc(f)-2)
\]
is equivalent to
\begin{eqnarray*}
 \textsf{DensLowPer}_f(\bc(f))
    & :=  & \frac{\Card\bigl(\{1,\ldots, \bc(f)-2\}\cap \Per(f)\bigr)}{\bc(f)-2}\\
    & \le & 2\frac{\log_2(\bc(f)-2)}{\bc(f)-2},
\end{eqnarray*}
where for every $L \in \sbcset(f),$
$\textsf{DensLowPer}_f(L)$ denotes the \emph{density of the
$L-$low periods of $f$} which, by definition, are the
periods of $f$ which are smaller than $L.$
Consequently, again, the larger it is $\bc(f)$ the smaller it is
the density of the $\bc(f)-$low periods of $f$ because
\[
      \frac{\log_2(x)}{x}\text{ is decreasing for $x \ge 2$}
      \andq
      \lim_{x\to\infty} \frac{\log_2(x)}{x} = 0.
\]

\begin{remark}
The reason for defining $\bc(f):=\max \sbcset(f)$ is that,
among all elements of the set $\sbcset(f),$ $\max \sbcset(f)$ is the
one that gives
the largest possible lower bound of the beginning of the cofinite part
of the set of periods
and \emph{simultaneously}
the smallest possible density of the set of $\bc(f)-$low periods of $f.$
\end{remark}

\begin{remark}\label{NotSBC}
The boundary of cofiniteness cannot be replaced by the
much simpler notion of strict boundary of cofiniteness.
Indeed, since the condition
\[
 \Card\bigl(\{1,\ldots, \bc(f)-2\}\cap \Per(f)\bigr) \le 2 \log_2(\bc(f)-2)
\]
need not be verified by $\sbc(f)$ (when $\sbc(f) \notin \sbcset(f)$),
the density of the $\sbc(f)-$low periods of $f$
could be large (in fact, even, arbitrarily close to one),
contradicting the simplicity of the set of periods.
\end{remark}

Now we can state precisely what do we mean by extending the entropy
simplicity phenomenon described above to the set of periods:
\emph{is it true that there exist totally transitive (and hence
dynamically complicate) graph maps with arbitrarily large boundary of
cofiniteness?}

We start by illustrating the above statement with three examples
for arbitrary graphs which are not trees.

The rotation interval of a circle map of degree one is a fundamental
tool to determine its set of periods.
We will define this object in Section~\ref{DTR}, where will describe
in detail its relation with the set of periods of the map under
consideration.
In what follows we will denote the set of all liftings of all
continuous circle maps of degree  one by $\dol,$ and
the rotation interval of $F\in \dol$ by $\Rot(F).$
These notions are well known and,
to improve the readability of this rather long introduction,
their definitions are written in detail in Subsection~\ref{RotTheor}.

\begin{example}[the dream example]\label{exampleBCNintroduction}
For every positive integer $n \ge 3$ there exists $f_n,$
a totally transitive continuous circle map of degree one having a
lifting $F_n \in \dol$ such that
$\Rot(F_n) = \left[\tfrac{1}{2n-1}, \tfrac{2}{2n-1}\right],$
$\Per(f_n) = \succs{n}$ and
$\lim_{n\to\infty} h(f_n) = 0.$
Hence, $\bc(f_n) = \sbc(f_n) = n$ and, consequently,
$\lim_{n\to\infty} \bc(f_n) = \infty$.
\smallskip

Furthermore, given any graph $G$ with a circuit, the sequence of maps
$\{f_n\}_{n\ge 5}$ can be extended to a sequence of
continuous totally transitive self maps of $G,$ $\{g_n\}_{n\ge 5},$
such that $\Per(g_n) = \Per(f_n)$ and $\lim_{n\to\infty} h(g_n) = 0.$
\end{example}

\begin{remark}\label{rem:exampleBCNintroduction}
In this example there are no $\bc(g_n)-$low periods.
Hence, $\sbc(g_n)$ is enough to control the complexity of the set of periods.
\end{remark}

\begin{example}[with persistent fixed low periods]\label{examplefirstexamplebcnintroduction}
For every positive integer $n \in \set{4k+1, 4k-1}{k\in \N}$
there exists $f_n,$ a totally transitive continuous circle map
of degree one having a lifting $F_n \in \dol$ such that
$\Rot(F_n) = \left[\tfrac{1}{2}, \tfrac{n+2}{2n}\right],$
$\lim_{n\to\infty} h(f_n) = 0,$
\[
  \Per (f_n) = \{ 2\} \cup \set{q\ \text{odd}}{2k+1 \le q \le n-2} \cup  \succs{n}
\]
and $\bc(f_n)$ exists and verifies $2k+1 \le \bc(f_n) \le n$
(and, hence, $\lim_{n\to\infty} \bc(f_n) = \infty$).\smallskip

Furthermore, given any graph $G$ with a circuit, the sequence of maps
$\{f_n\}_{n\ge 7, n\text{ odd}}$ can be extended to a sequence
of continuous totally transitive self maps of $G$,
$\{g_n\}_{n\ge 7, n\text{ odd}},$
such that
$\Per(g_n) = \Per(f_n)$ and, additionally, $\lim_{n\to\infty} h(g_n) = 0.$
\end{example}

\begin{remark}\label{rem:examplefirstexamplebcnintroduction}
The above example is different from the previous one
since for every $n$ there exist $\bc(f_n)-$low periods and, moreover,
every map $f_n$ has a constant $\bc(f_n)-$low period 2.

On the other hand, $\sbc(f_n) = n \ne \bc(f_n).$
In this example this is due to the fact that the set of periods which
are smaller than $\sbc(f_n)$ is very large relative to the value of
$\sbc(f_n).$ More concretely,
\[
   \textsf{DensLowPer}_{f_n}\left(\sbc(f_n)\right) =
   \begin{cases}
       \tfrac{k+1}{4k-1} & \text{if $n = 4k+1,$}\\
       \tfrac{k}{4k-3} & \text{if $n = 4k-1.$}
   \end{cases}
\]
So, for large $n,$
\begin{multline*}
   \textsf{DensLowPer}_{f_n}\left(\sbc(f_n)\right)  \approx \tfrac{1}{4} >\\
   \frac{2 \log_2(n-2)}{n-2} = \frac{2 \log_2(\sbc(f_n)-2)}{\sbc(f_n)-2}
\end{multline*}
and, hence, the strict boundary of cofiniteness does not belong to
$\sbcset(f_n).$
Furthermore, the differences $\sbc(f_n) - \bc(f_n)$ are unbounded
because
\[
  \textsf{DensLowPer}_{f_n}\left(\bc(f_n)\right) \le
        \frac{2 \log_2(\bc(f_n)-2)}{\bc(f_n)-2}
\]
converges to zero as $n$ goes to infinity (so, if
the differences $\sbc(f_n) - \bc(f_n)$ are bounded, then
$\textsf{DensLowPer}_{f_n}\left(\sbc(f_n)\right)$ also
converges to zero as $n$ goes to infinity; which contradicts the
previous estimate).
This is a concrete new motivation of our definition of boundary of
cofiniteness.
\end{remark}

\begin{example}[with non-constant low periods]\label{examplemontevideuexampleintroduction}
For every $n \in \N,\ n\ge 3$ there exists $f_n,$  a totally transitive
continuous circle map of degree one having a lifting $F_n \in \dol$
such that
\[
    \Rot(F_n) = \left[\tfrac{2n-1}{2n^2}, \tfrac{2n+1}{2n^2}\right] =
        \left[\tfrac{1}{n}-\tfrac{1}{2n^2},
        \tfrac{1}{n}+\tfrac{1}{2n^2}\right],
\]
$\lim_{n\to\infty} h(f_n) = 0$ and
\begin{multline*}
 \Per(f_n) =
   \{n\}\ \cup\\
   \set{t n + k}{t \in \{2,3,\dots,\nu-1\} \text{ and }
                    -\tfrac{t}{2} < k \le \tfrac{t}{2},\ k \in \Z} \cup \\
   \succs{n\nu+1-\tfrac{\nu}{2}}
\end{multline*}
with
\[
\nu = \begin{cases}
        n & \text{if $n$ is even, and}\\
        n-1 & \text{if $n$ is odd.}
     \end{cases}
\]
Moreover, $\sbc(f_n) = n\nu+1-\tfrac{\nu}{2}$
and $\bc(f_n)$ exists and verifies $n\le \bc(f_n) \le n\nu - 1 -\tfrac{\nu}{2}$
(and hence, $\lim_{n\to\infty} \bc(f_n) = \infty$).
\smallskip

Furthermore, given any graph $G$ with a circuit, the sequence of maps
$\{f_n\}_{n=4}^\infty$ can be extended to a sequence of continuous
totally transitive self maps of $G,$ $\{g_n\}_{n=4}^\infty,$
such that $\Per(g_n) = \Per(f_n)$ and $\lim_{n\to\infty} h(g_n) = 0.$
\end{example}

\begin{remark}\label{rem:examplemontevideuexampleintroduction}
This example is different from the previous two examples
since for every $n$ there exist $\bc(f_n)-$low periods but there is
no constant $\bc(f_n)-$low period.

Moreover, as in the previous example, $\sbc(f_n) \ne \bc(f_n),$
\[
\textsf{DensLowPer}_{f_n}\left(\sbc(f_n)\right)  \approx \tfrac{1}{2}
\]
and the differences $\sbc(f_n) - \bc(f_n)$ are unbounded.
\end{remark}

\begin{remark}
In all three examples we still have $\lim_{n\to\infty} h(g_n) = 0$
despite of the fact that $h(g_n)$ is slightly larger than $h(f_n)$
for every $n.$
\end{remark}

Finally, we state the main theorem of the paper that shows that the
above examples are not exceptional among circle maps of degree one.
On the contrary, a sequence of totally transitive circle maps
that unfolds an entropy simplification process also must unfold a
set of periods simplification process:

\begin{MainTheorem}\label{MTS1}
Let $\{f_n\}_{n\in \N}$ be a sequence of totally transitive circle maps
of degree one with periodic points such that
$\lim_{n\to\infty} h(f_n) = 0.$
For every $n$ let $F_n \in \dol$ be a lifting of $f_n.$
Then,
\begin{widthitemize}
\item $\lim_{n\to\infty} \len\left(\Rot(F_n)\right) = 0,$
\item there exists $N \in \N$ such that $\bc(f_n)$ exists for every $n \ge N,$ and
\item $\lim_{n\to\infty} \bc(f_n) = \infty.$
\end{widthitemize}
\end{MainTheorem}

Despite of this very general theorem for totally transitive circle
maps of degree one, we emphasize that the examples are strongly
relevant and motivated by the following two reasons:
from one side they are valid for any general graph that is not a tree
(not just the circle) thus showing that the phenomenon described in
this paper seems to be much more general than the result obtained in
Theorem~\ref{MTS1}. From another side, the three examples show that all
situations about the density of the \emph{$\sbc(f)-$low periods of $f$}
are possible (see Remarks~\ref{rem:exampleBCNintroduction},
\ref{rem:examplefirstexamplebcnintroduction} and
\ref{rem:examplemontevideuexampleintroduction}) and, in particular,
they show that the boundary of cofiniteness cannot be replaced by the
notion of strict boundary of cofiniteness (see Remark~\ref{NotSBC}).

This paper is organized as follows.
In Section~\ref{DTR} we introduce the definitions and preliminary
results for the rest of the paper;
the construction of the
Examples~\ref{exampleBCNintroduction},
         \ref{examplefirstexamplebcnintroduction} and
         \ref{examplemontevideuexampleintroduction}
is done in the very long Section~\ref{sectionexamples}.
Finally, Theorem~\ref{MTS1} is proved in Section~\ref{sectiontheorem}.


\section{Definitions and preliminary results}\label{DTR}

In this section we essentially follow the notation and strategy of \cite{alm}.

\subsection{Basic definitions}

In the rest of the paper we denote the unit circle $\set{z\in \C}{\abs{x} = 1}$ by $\SI.$

A \emph{topological graph\/} (or simply a \emph{graph\/})
is a connected compact Hausdorff space $X$ for which there exists a
finite (or empty) set $V(X),$ called the \emph{set of vertices of $X$ \/},
such that either
$X= \SI$ and $V(X) = \emptyset$ or
$X\setminus V(X)$ is the disjoint union of finitely many
open subsets of $X$ each of them homeomorphic to an open
interval of the real line, called \emph{edges of $X$}, with the property
that the boundary of every edge consists of at most two points which are in $V(X).$
A point $z \in V(X)$ is an \emph{endpoint of $X$} if there exists
an open (in $X$) neighbourhood $U$ of $z$ such that
$X \cap U$ is homeomorphic to the interval $[0,1)$ being $z$ the preimage of $0$.
A \emph{circuit\/} of a graph $X$ is any subset of $X$ homeomorphic to a circle.
A \emph{tree\/} is a graph without circuits.

Let $X$ be a topological space and let $\map{f}{X}$ be a continuous map.
When iterating the map $f$ we will use the following notation:
$f^0$ will denote the identity map (in $X$),
and $f^n := f\circ f^{n-1}$ for every $n\in \N,$ $n \ge 1.$
For a point $x\in X,$ we define the \emph{$f$-orbit of $x$ \/}
(or simply the \emph{orbit of $x$ \/}),
denoted by $\Orb_f(x),$ as the set $\set{f^n(x)}{n \ge 0}.$
A point $x\in X$ is called a \emph{periodic point of $f$ \/} if $f^n(x) = x.$
In such case $\Orb_f(x)$ is called a \emph{periodic orbit of $f$ \/}
and $\Card\left(\Orb_f(x)\right)$ is called the \emph{period of $x$ \/}
(or \emph{$f$-period of $x$ \/} if we need to specify the map).
Observe that if $x$ is a periodic point of $f$ of period $n,$
then $f^k(x) \ne x$ for every $1 \le k < n$ and
if $P$ is a periodic orbit of $f$, then $P = \Orb_f(x)$ for every $x\in P.$

The set of all positive integers $n$ such that $f$ has a periodic point
of period $n$ is denoted by $\Per(f).$

\subsection{Rotation theory and sets of periods for circle maps of degree one}\label{RotTheor}
We start by introducing the key notion to work with circle maps: the liftings.
Let $\map{\eexp}{\R}[\SI]$ be the natural projection which is defined
by $\emap{x} := \exp(2\pi ix).$
Given a continuous map $\map{f}{\SI},$ we say that a
continuous map $\map{F}{\R}$ is a \emph{lifting of $f$ \/}
if $\emap{F(x)} = f(\emap{x})$ for every $x \in \R.$
For such $F,$ there exists $d\in \Z$ such that
$F(x +1)= F(x) +d$ for all $x\in\R,$
and this integer is called both the \emph{degree of $f$ \/} and the \emph{degree of $F$}.
If $G$ and $F$ are two  liftings of $f$ then $G = F + k$ for some integer $k$
and so $F$ and $G$ have the same degree.
We denote by $\dol[d]$  the set of all liftings  of circle maps of degree $d.$

Next we introduce the important notion of \emph{rotation interval} for maps from $\dol$.
Let $F\in \dol$ and let $x \in \R$. The number
\[
   \rho_F (x) := \limsup\limits_{n\to \infty}\frac{F^n(x) -x}{n}
\]
will be called the \emph{rotation number of $x$}. Moreover, the set
\[
  \Rot(F) := \set{\rho_F(x)}{x\in \R} = \set{\rho_F(x)}{x\in [0,1]}
\]
will be called the \emph{rotation interval of $F$}.
It is well known that it is a closed interval of the real line \cite{Ito}.

If $F\in \dol$ is a non-decreasing map, then
\[
   \rho_F (x) = \lim\limits_{n\to \infty}\frac{F^n(x) -x}{n}
\]
for every $x \in \R$ and, moreover, it is independent on $x$
(see for instance \cite{alm}).
Then this number ($\rho_F (x)$ for any $x\in \R$), will be called the
\emph{rotation number of $F$}.
For every $F \in \dol$ we define the \emph{lower map \map{F_l}{\R}} by
(see Figure~\ref{upper-lowermaps} for a graphical example)
\begin{align*}
F_l(x) &= \inf\set{F(y)}{y \ge x}\\
\intertext{and the \emph{upper map \map{F_u}{\R}} by}
F_u(x) &= \sup\set{F(y)}{y \le x}.
\end{align*}
It is easy to see (see e.g. \cite{alm}) that
$F_l,F_u$ are non-decreasing maps from $\dol.$
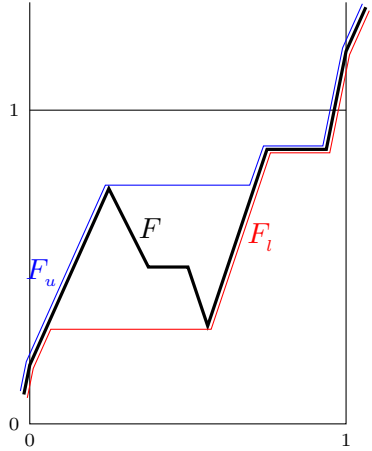
\begin{figure}
\begin{tikzpicture}[scale=0.26]
\draw (0,21.5) -- (0,0) -- (16,0) -- (16,21.5);
\draw (0,16) -- (16,16);
\node[below] at (0,0) {\tiny $0$}; \node[left] at (0,0) {\tiny $0$};
\node[below] at (16,0) {\tiny $1$}; \node[left] at (0,16) {\tiny $1$};

\begin{scope}[shift={(0,3)}]
  \draw[very thick] (-0.3,-1.5) -- (0,0) -- (4,9) -- (6,5) -- (8,5) -- (9,2) -- (12,11) -- (15,11) -- (16,16) -- (17, 18.25);
  \node[right] at (5,7) {$F$};
  \begin{scope}[shift={(0.175,-0.175)}]\draw[color=red] (-0.3,-1.5) -- (0,0) -- (0.888,2) -- (9,2) -- (12,11) -- (15,11) -- (16,16) -- (17, 18.25);\end{scope}
  \node[left, color=blue] at (2,4.7) {$F_u$};
  \begin{scope}[shift={(-0.175,0.175)}]\draw[color=blue] (-0.3,-1.5) -- (0,0) -- (4,9) -- (11.3,9) -- (12,11) -- (15,11) -- (16,16) -- (17, 18.25);\end{scope}
  \node[right, color=red] at (10.5,6.5) {$F_l$};
\end{scope}
\end{tikzpicture}
\caption{An example of a map $F \in \dol$ with its
lower map \textcolor{red}{$F_l$} in \textcolor{red}{red} and its
upper map \textcolor{blue}{$F_u$} in \textcolor{blue}{blue}.}\label{upper-lowermaps}
\end{figure}

The next theorem gives an effective way to compute the rotation
interval from the rotation numbers of the upper and lower maps.

\begin{theorem}[{\cite[Theorem~3.7.20]{alm}}]\label{theoremrotationintwathermap}
For every $F\in \dol$ it follows that $\Rot(F) =[\rho(F_l), \rho(F_u)].$
\end{theorem}

It is well known that the rotation interval of a lifting $F \in \dol$
can be used to obtain information about the set of periods of the
corresponding circle map.
To clarify this point we will introduce the notion of \emph{lifted orbit\/}.

Let $f$ be a continuous circle map of degree $d$ and let $F\in \dol[d]$
be a lifting of $f$. A set $P\subset \R$ will be called a
\emph{lifted orbit of $F$ \/} if there exists $z \in \SI$ such that
$P = \eexp^{-1}\bigl(\Orb_f(z)\bigr)$ and
$f(\emap{x})=\emap{F(x)}$ for every $x \in P.$
Whenever $z$ is a periodic point of $f$ of period $n,$
$P$ will be called a \emph{lifted periodic orbit of $F$ of period $n$}.
We will denote by $\Per(F)$ the set of periods of all
lifted periodic orbits of $F.$ Observe that then, $\Per(F) = \Per(f).$

\begin{remark}
Let $F\in \dol$ and let $P$ be a lifted periodic orbit of $F$ of
period period $n.$ Set
\[
  P = \{\dots, x_{-2}, x_{-1},x_0, x_1, x_2,\dots\}
\]
with $x_i < x_j$ if and only if $i < j.$
The fact that $P = \eexp^{-1}\bigl(\Orb_f(z)\bigr)$, in this case, gives
\[
 \Card\left(P \cap [r,r+1)\right) = n
\]
for every $r \in \R$ and, hence,
\[
  x_{kn + i}=x_i + k
\]
for every $i, k \in \Z.$

Moreover, there exists $m \in \Z$ such that
$F^n(x_i) = x_i + m = x_{mn + i}$ for every $x_i \in P.$
Consequently,
\[
 \rho_F (x_i) = \tfrac{m}{n}
\]
for every $x_i \in P.$
\end{remark}

From the above remark it follows that if $P$ is a lifted periodic orbit
of $F \in \dol,$ then all the points of $P$ have the same rotation number.
This number will be called the \emph{rotation number of $P$ \/}
(or \emph{$F$-rotation number of $P$ \/} if we need to specify the lifting).

A lifted periodic orbit $P$ of $F \in \dol$
such that $F\evalat{P}$ is increasing
will be called \emph{twist\/}.

\begin{remark}\label{rem:liftedperiodicorbitandtwist}
Let
\[
  P = \{\dots, x_{-2}, x_{-1},x_0, x_1, x_2,\dots\}
\]
be a twist lifted periodic orbit of $F \in \dol$
of period $n$ and rotation number $m/n$ labelled so that
$x_i < x_j$ if and only if $i < j.$
By \cite[Lemma~3.7.4 and Corollary~3.7.6]{alm} we have that
$m$ and $n$ are coprime and
\[
  F(x_i) = x_{i + m},
\]
for all $i\in \Z.$
\end{remark}

The next theorem due to Misiurewicz (see \cite{m, alm})
already makes the connection between $\Rot(F)$  and  $\Per(F).$
To state it, we still need to recall the \emph{Sharkovski\u{\i} Ordering}.

The \emph{Sharkovski\u{\i} Ordering} $\geso{\Sho}$ (the symbols
$\gtso{\Sho},$ $\ltso{\Sho}$  and $\leso{\Sho}$  will also be
used in the natural way) is a linear ordering of
$\N_{\Sho} := \N \cup \{2^\infty\}$ (we have to
include the symbol $\{2^\infty\}$ in order to ensure the existence of
supremum of every subset with respect to the ordering $\geso{\Sho}$)
defined as follows:
\begin{align*}
 & 3 \gtso{\Sho} 5 \gtso{\Sho} 7 \gtso{\Sho} 9 \gtso{\Sho} \dots \gtso{\Sho} \\
 & 2 \cdot 3 \gtso{\Sho} 2 \cdot 5 \gtso{\Sho} 2 \cdot 7 \gtso{\Sho} 2 \cdot 9 \gtso{\Sho} \dots \gtso{\Sho} \\
 & 4 \cdot 3 \gtso{\Sho} 4 \cdot 5 \gtso{\Sho} 4 \cdot 7 \gtso{\Sho} 4 \cdot 9 \gtso{\Sho} \dots \gtso{\Sho}\\
 & \hspace*{7em} \vdots \\
 & 2^n \cdot 3  \gtso{\Sho}  2^n \cdot 5 \gtso{\Sho} 2^n \cdot 7 \gtso{\Sho} 2^n \cdot 9 \gtso{\Sho} \dots \gtso{\Sho} \\
 & \hspace*{7em} \vdots \\
 & 2^\infty\gtso{\Sho}\cdots \gtso{\Sho} 2^n \gtso{\Sho} \dots \gtso{\Sho}
 16 \gtso{\Sho} 8 \gtso{\Sho} 4 \gtso{\Sho} 2 \gtso{\Sho} 1.
\end{align*}

We introduce the following notation.
Given $c,d \in \R,$ $c \le d$ we set
\[
  M(c,d) := \set{n \in \N}{c < k/n < d \text{ for some integer } k}.
\]
Let $F\in\dol $  let $c$ be an endpoint of $\Rot(F).$
We define the set
\[
  Q_F(c) := \begin{cases}
             \emptyset                                       & \text{if $c \notin \Q$}\\
             \set{sk}{k\in\N \text{ and } k \leso{\Sho} s_c} & \text{if $c = r/s$ with $r, s$ coprime}
        \end{cases}
\]
and $s_c \in \N_{\Sho}$ is defined by
the Sharkovski\u{\i} Theorem on the real line.
Indeed, since $c = r/s$ and $r$ and $s$ are coprime,
the map $F^s - r$ is a continuous map on the real line with periodic
points.
Hence, by the Sharkovski\u{\i} Theorem
there exists an $s_c \in \N_{\Sho}$ such that the set of periods
(not lifted periods) of $F^s - r$ is precisely
$\set{s\in\N}{s\leso{\Sho} s_c}.$

\begin{theorem}\label{theoremMisiurewicz}
Let $f$ be a continuous circle map of degree one having a lifting $F \in \dol.$
Assume that $\Rot(F) = [c,d].$
Then \[ \Per(f) = Q_F(c) \cup M(c,d) \cup Q_F(d). \]
\end{theorem}

\subsection{Markov graphs, Markov maps and sets of periods}
Take a finite set $V=\{v_1, v_2, \dots, v_n\}$.
The pair  $\mathcal{G}=(V, U)$ where $U \subset V \times V$
is called a \emph{combinatorial oriented  graph\/}.
The elements of $V$ are called the \emph{vertices of $\mathcal{G}$ \/}
and each element $(v_i, v_j) \in U$ is called an \emph{arrow from $v_i$ to $v_j$}.
An arrow $(v_i, v_j)$ will also be denoted  by $v_i \longrightarrow v_j,$
which allows us to give  a graphical representation of an oriented graph.
A \emph{path of length $k$} is a sequence of $k+1$ vertices
$v_0,v_1,\dots,v_k$  with the property that there is an arrow from every vertex
to the next one. A path is denoted as
$v_0\longrightarrow v_1 \longrightarrow v_2  \longrightarrow \cdots \longrightarrow v_k.$
A \emph{loop of length $k$} is a path of length $k$ where the first
and last vertex coincide:
$v_0\longrightarrow v_1 \longrightarrow v_2  \longrightarrow \cdots \longrightarrow v_{k-1}  \longrightarrow v_0.$

Let $X$ be a topological graph.
Every subset of $X$ homeomorphic to the interval $[0,1]$
will in turn be called an \emph{interval of $X$}.
The preimages of $0$ and $1$ by the homeomorphism will
be called \emph{the endpoints of $I$} and the set of (both) endpoints
of $I$ will be denoted by $\partial{}I.$
Note that if $I \cap V(X) = \emptyset,$ then $I\setminus\partial{}I = \Int(I).$

Let $X$ be a topological graph and let $\map{f}{X}$ be a continuous map.
A set $Q\subset X$ will be called \emph{$f$-invariant\/} if $f(Q)\subset Q.$
A \emph{Markov invariant set\/} is defined to be a
finite $f$-invariant set $Q \supset V(X)$ such that
the closure of each connected component of $X \setminus Q,$
called a \emph{$Q$-basic interval\/}, is an interval of $X.$
Observe that two different $Q$-basic intervals have disjoint interiors
(here \emph{interior} means the topological interior in $X$ which is
the whole $Q$-basic interval minus those of its endpoints which
are not endpoints of the graph because $Q \supset V(X)$).

The set of all $Q$-basic intervals will be denoted by \SBI.

\begin{definition}[Of monotonicity over an interval]
Let $X$ be a topological graph, let $\map{f}{X}$ be a continuous map
and let $I$ be an interval of $X.$
The map $f$ will be said to be \emph{monotone at $I$ \/} if the set
$\bigl(f\evalat{I}\bigr)^{-1}(y) = \set{x\in I}{f(x) = y}$
is connected for every $y \in f(I).$
Clearly, since $I$ is an interval,
$\bigl(f\evalat{I}\bigr)^{-1}(y)$ is either a point or an interval
for every $y \in f(I).$
Moreover, a simple exercise shows that the subgraph $f(I),$ in turn,
must be either a point or an interval.
Finally, let $\map{g}{X}$ be another continuous map,
and let $J$ be another interval of $X$ such that
$f(I) \cap \Int(J) \ne \emptyset$ and
$g$ is monotone at $J.$
A first year calculus exercise shows that
$\bigl(f\evalat{I}\bigr)^{-1}(J) = \set{x\in I}{f(x) \in J}$ 
is an interval and
$g\evalat{J} \circ f\evalat{I}$
is monotone at $\bigl(f\evalat{I}\bigr)^{-1}(J).$
\end{definition}

\begin{remark}
The above definition of monotonicity over an interval is equivalent
to the usual one: $f(I)$ is a point or an interval and,
in the second case, the map
$\map{\zeta \circ f\evalat{I} \circ \xi}{[0,1]}$
is monotone as interval map, where
$\map{\xi}{[0,1]}[I]$ and $\map{\zeta}{f(I)}[{[0,1]}]$
are homeomorphisms (which exist because $I$ and $f(I)$ are intervals).
We prefer the more intrinsic definition given above because it is
independent on the choice of the auxiliary homeomorphisms and
on whether they are increasing or decreasing. This will be specially helpful
when studying compositions of monotone maps over intervals like in
the rest of this subsection and Subsection~\ref{TransitivityMarkovmatrices}.
\end{remark}

Let $X$ be a topological graph, let $\map{f}{X}$ be a continuous map
and let $Q \subset X$ be a Markov invariant set.
We say that $f$ is \emph{$Q$-monotone\/} if $f$ is monotone on each $Q$-basic interval.
In such a case, $Q$ is called a \emph{Markov partition of $X$ with respect to $f$ \/}
and $f$ is called a \emph{Markov map with respect to $Q$}.

Next we introduce the very important notion of $f$-covering that
allows us to get a combinatorial oriented graph from a Markov partition
of a Markov map.

Let $X$ be a topological graph, let $\map{f}{X}$ be a continuous map and
let $Q$ be a Markov partition of $X$ with respect to $f.$
Given $I, J \in \SBI,$
we say  that  $I$ \emph{$f$-covers\/} $J$ if $f(I)\supset J.$
The \emph{Markov graph of $f$ with respect to $Q$ \/} (or \emph{$f$-graph\/})
is a combinatorial oriented graph whose vertices are all the $Q$-basic
intervals and there is an arrow  $I \longrightarrow J$
from the vertex ($Q$-basic interval) $I$ to the vertex ($Q$-basic interval) $J$
if and only if $I$ $f$-covers $J.$
The \emph{Markov  matrix of $f$ with respect to $Q$ \/} is another
combinatorial object that  describe the dynamical behaviour
of a Markov map $f$ and is associated in a natural way to the
Markov graph of $f$ with respect to $Q$.
It is a  $\Card(\SBI) \times \Card(\SBI)$ matrix
$M = (m_{I,J})_{I,J \in \SBI}$ such that
\[
  m_{I,J}=\begin{cases}
      1 & \text{if $I$ $f$-covers $J$}\\
      0 & \text{otherwise}
\end{cases}.
\]

The next lemma shows the relation between loops of Markov graphs and
periodic points. Essentially it is \cite[Corollary~1.2.8]{alm}
extended to graph maps. 

Given a tree $T$ (which is uniquely arcwise connected) and a set
$A \subset T,$ we denote by $\chull{A}[T]$ the
\emph{convex hull of $A$ in $T$},
that is, the smallest closed connected set of $T$
that contains $A.$ Also, given $q \in \N$,
the congruence classes modulo $q$ will be $\{0,1,\dots,q-1\}.$

\begin{lemma}\label{lemmaInterv}
Let $X$ be a topological graph, let $\map{f}{X}$ be a
Markov map with respect to a Markov partition $Q$ of $X$
and let
$\alpha= I_0 \longrightarrow I_1 \longrightarrow \cdots \longrightarrow I_{n-1} \longrightarrow I_0$
be a loop in the Markov graph of $f$ with respect to $Q.$
Then, there exists a fixed point $x \in I_0$ of $f^n$
such that $f^{i}(x)\in I_i$ for $i=1,2,\dots,n-1.$
\end{lemma}

The next result compiles \cite[Lemma~1.2.12 and Theorem~2.6.4]{alm} extended to graph maps.
To state it we need to introduce some more definitions.

Given two paths
$\alpha = v_0 \longrightarrow v_1 \longrightarrow v_2  \longrightarrow \cdots \longrightarrow v_k$
and $\beta = w_0 \longrightarrow w_1 \longrightarrow w_2  \longrightarrow \cdots \longrightarrow w_m$
in a combinatorial graph such that the last vertex of the first path
is the first vertex of the second one (i.e., $v_k = w_0$),
the \emph{concatenation\/} of $\alpha$ and $\beta$
is denoted by $\alpha\beta$ and is the path
\[
  v_0 \longrightarrow v_1 \longrightarrow v_2  \longrightarrow \cdots \longrightarrow
  v_{k-1} \longrightarrow w_0 \longrightarrow w_1 \longrightarrow w_2  \longrightarrow
  \cdots \longrightarrow w_m.
\]
Clearly, the length of the concatenated path is the sum of the lengths of the original paths.
A loop is an \emph{$n$-repetition\/} of a (shorter) loop $\alpha$
if $n \ge 2$ and it is a concatenation of $\alpha$ with itself $n$ times.
Such a loop will be called \emph{repetitive\/}.
A loop which is not repetitive will also be called \emph{simple}.

The \emph{shift of a loop\/}
$\alpha = v_0 \longrightarrow v_1 \longrightarrow v_2 \longrightarrow \cdots \longrightarrow v_{k-1} \longrightarrow v_0$
is defined to be the loop
\[
 S(\alpha) := v_1 \longrightarrow v_2 \longrightarrow \cdots \longrightarrow v_{k-1} \longrightarrow v_0 \longrightarrow v_1.
\]
Iteratively, we set $S^0(\alpha) := \alpha$ and, for every $m \in \N,$
we define the \emph{$m$-shift of $\alpha$,} denoted by $S^m(\alpha),$
as the loop
\[
  v_{m\mkern-12mu\pmod{k}} \longrightarrow v_{m+1\mkern-12mu\pmod{k}} \longrightarrow \cdots \longrightarrow v_{m+k-1\mkern-12mu\pmod{k}} \longrightarrow v_{m\mkern-12mu\pmod{k}}.
\]

Let $X$ be a topological graph, let $\map{f}{X}$ be a
Markov map with respect to a Markov partition,
let
$\alpha = I_0 \longrightarrow I_1 \longrightarrow \cdots \longrightarrow I_{m-1} \longrightarrow I_0$
be a loop in the Markov graph of $f$
and let $P$ be a periodic orbit of period $m$ of $f$.
We say that $\alpha$ and $P$ are \emph{associated\/} to each other
if there exists an $x \in P$ such that $f^k(x)\in I_k$ for $k=0,1,\dots,m-1.$
Observe that if $\alpha$ is associated to a periodic orbit,
then so is $S^k(\alpha)$ for every $k \in \N$
(by replacing $x$ by $f^k(x)$ in the above definition).

Next we will define the notion of \emph{negative\/} and \emph{positive\/}
loop in a Markov graph of $f$ with respect to a Markov partition.
Given a loop
$\alpha = I_0 \longrightarrow I_1 \longrightarrow \cdots \longrightarrow I_{m-1} \longrightarrow I_0$
we inductively define a sequence of intervals $\{U_k\}_{k=0}^m$ as follows.
We start by setting $U_m = I_0.$
Then, for every positive integer $k = m-1, m-2, \dots,1,0,$ we take $U_k$
to be the unique interval contained in $I_k$ such that
$f(\Int(U_k)) = \Int(U_{k+1})$
(such intervals exist due to the monotonicity of $f$ at $I_k$
and because $I_k$ $f$-covers $I_{k+1\mkern-6mu\pmod{m}} \supset U_{k+1}$).
The continuity of $f$ implies that $f$ sends $\partial{}U_k$
bijectively to $\partial{}U_{k+1}$ for every $k=0,1,\dots,m-1.$
Denote the endpoints of $U_0$ as $u^0_-$ and $u^0_+$
(this is equivalent to define a linear ordering in $I_0;$
namely the orientation that gives $u^0_- < u^0_+$).
Now we label the endpoints of $I_0$ as $r^0_-$ and $r^0_+$
in such a way that
$\chull{r^0_-, u^0_-}[I_0] \cap U_0 = \{u^0_-\}$ and
$\chull{r^0_+, u^0_+}[I_0] \cap U_0 = \{u^0_+\}.$
Putting all together,
$f^m$ is monotone at $U_0,$
$f^m(\Int(U_0)) = \Int(I_0)$ and
$f^m$ sends $\{u^0_-, u^0_+\}$ bijectively to $\{r^0_-, r^0_+\}.$
There are two cases:
\begin{itemize}
 \item $f^m(u^0_-) = r^0_-$ and $f^m(u^0_+) = r^0_+,$ and we call the loop $\alpha$ \emph{positive,} and
 \item $f^m(u^0_-) = r^0_+$ and $f^m(u^0_+) = r^0_-$ in which case we call the loop $\alpha$ \emph{negative.}
\end{itemize}
Note that in the positive case $f^m$ sends $U_0$ to $I_0$ preserving
the linear ordering compatible with $u^0_- < u^0_+,$ while
in the negative case $f^m$ sends $U_0$ to $I_0$ by reversing
this linear ordering.

\begin{proposition}[{\cite[Lemma~1.2.12 and Theorem~2.6.4]{alm}}]\label{propPerLoops}
Let $X$ be a topological graph and let $\map{f}{X}$ be a
Markov map with respect to a Markov partition $Q$ of $X.$
Then, the following statements hold:
\begin{enumerate}[(a)]
\item Assume that $P$ is a periodic orbit of $f$ disjoint from $Q.$
Then,
there exists a loop $\alpha$ in the Markov graph of $f$
with respect to $Q$ which is associated to $P,$
and any other loop in the $f$-graph of $Q$ associated to $P$ is of the
form $S^k(\alpha)$ with $k \in \N.$

\item Let $\alpha$ be a loop from the Markov graph of $f$ with respect to $Q$
which has a periodic orbit associated to it.
Then $\alpha$ is either a simple loop or a $2$-repetition of a simple negative loop.
\end{enumerate}
\end{proposition}

\subsection{Markov graphs modulo 1 for circle maps of degree one}
As we have seen, when the topological graph is the circle $\SI$ and
we are dealing with maps of degree one it is better to work with
liftings instead of with the original maps, specially to avoid problems
with ordering (the circle does not have a linear ordering).
In what follows we adapt the definitions related to Markov graphs
to this approach.

Let $f$ be a continuous circle map and let $\widetilde{Q} \subset \SI$
be a finite $f$-invariant set with at least two elements
(in fact, since $V(\SI) = \emptyset$, $\widetilde{Q}$ is a Markov invariant set).
Let $F\in \dol$ be a lifting of $f.$
Then the set $Q = \eexp^{-1}\bigl(\widetilde{Q}\bigr)$ is $F$-invariant,
is a partition of $\R$ and each interval produced by this
partition will be called a \emph{$Q$-basic interval\/}.
Again the set of all $Q$-basic intervals will be denoted by \SBI.
If the restriction of $F$ to each $Q$-basic interval is monotone
(as a map from the real line to itself),
we say that $F$ is $Q$-monotone,
$Q$ is a \emph{Markov partition with respect to $F$},
and $F$ is a \emph{Markov map}.
Given $I,J \in \SBI,$  we say $I$
is equivalent to $J$ and denote it by $I \sim J$
if $I = J + k$ for some $k \in \Z.$
The equivalence class of $I,$ $\set{I + k}{k \in\Z},$ is denoted by $\BIclass{I}.$

Now we define the \emph{Markov graph modulo 1 of $F$ with respect to $Q$}.
It is a combinatorial oriented graph whose vertices are
all the equivalence classes of $Q$-basic intervals and
there is an arrow $\BIclass{I} \longrightarrow \BIclass{J}$ from $\BIclass{I}$ to $\BIclass{J}$
if and only if there is a representative $J+k$ of $\BIclass{J}$ such that
$F(I) \supset J+k.$ Recall that, since $F \in \dol,$
$F(I + \ell) = F(I) + \ell$ for every $\ell \in \Z.$
Therefore, the Markov graph modulo 1 of $F$ with respect to $Q$ is well defined.
Moreover, two different liftings of $f$ have the same
Markov graphs modulo 1 with respect to $Q$.

\begin{remark}[On the projection of a Markov graph modulo 1 to the circle]\label{projectiontoSI}
Since the kernel of $\eexp$ is $\Z$
(that is, $\emap{x + k} = \emap{x}$ for every $x \in \R$ and $k \in \Z$),
it follows that $\emap{I} = \emap{K}$  for every $K \in \BIclass{I}.$
So, $\bigemap{\BIclass{I}} := \emap{I}$ is well defined.
Moreover, since $\widetilde{Q}$ has at least two elements, it follows
that every $Q$-basic interval has length strictly smaller than 1.
Hence, if $I \in \SBI,$ $\bigemap{\BIclass{I}} \in \SBI[\emap{Q}]$
(that is, $\bigemap{\BIclass{I}}$ is a $\emap{Q}$-basic interval in $\SI$
--- recall that $\emap{Q} = \widetilde{Q}$ is a Markov invariant set
of $\SI$ with respect to $f$).
Additionally, since $f(\emap{I}) = \emap{F(I)},$
$\bigemap{\BIclass{I}}$ $f$-covers $\bigemap{\BIclass{J}}$
if and only if there is an arrow $\BIclass{I} \longrightarrow \BIclass{J}$
in the Markov graph modulo 1 of $F.$

However, the $Q$-monotonicity of $F$ implies the
$\emap{Q}$-monotonicity of $f$
\emph{provided that the length of the $F$-image of every
$Q$-basic interval is smaller than 1\/}
(otherwise there exists $I \in \SBI$ such that
$f(\emap{I}) = \emap{F(I)} = \SI$ is not an interval).
In other words,
$\widetilde{Q}$ is a Markov partition of $\SI$ with respect to $f$
(and $f$ is a Markov map with respect to $\widetilde{Q}$)
whenever $Q$ is a Markov partition with respect to $F$ and
the length of the $F$-image of every $Q$-basic interval is smaller than 1.
\end{remark}

This remark motivates the following definition:
Let $F\in \dol$ be a lifting of $f$ and let $Q$ be a Markov partition with respect to $F.$
A $Q$-basic interval will be called \emph{$F$-short\/} if the length of the
interval $F(I)$ is strictly smaller than 1.
Then, $Q$ will be called a \emph{short Markov partition with respect to $F$}
whenever every $Q$-basic interval is $F$-short.

With this definition Remark~\ref{projectiontoSI} immediately gives
the following result that relates Markov partitions in the circle with
short Markov partitions with respect to liftings from $\dol$.

\begin{proposition}\label{Markov partitionprojectiontoSI}
Let $f$ be a continuous circle map and let $F\in \dol$ be a lifting of $f.$
Let $Q \subset \R$ be a Markov partition with respect to $F.$
Then, the following statements hold:
\begin{enumerate}[(a)]
\item $\emap{Q}$ is a Markov partition with respect to $f$ if and only if $Q$ is short.
\item When $Q$ is short,
the Markov graph of $f$ with respect to $\emap{Q}$ and
the Markov graph modulo 1 of $F$ with respect to $Q$ coincide,
provided that we identify $\BIclass{I}$ with $\bigemap{\BIclass{I}}$
for every $I \in \SBI$.
\end{enumerate}
\end{proposition}

\subsection{Markov graphs and entropy}
The Markov graph of a map is very useful to obtain information
about the dynamics of graph maps. The next result,
due to Block, Guckenheimer, Misiurewicz and Young \cite{bgmy}
(see also \cite[Theorem~4.4.5]{alm}),
relates the topological entropy of a Markov map with the
spectral radius of its associated Markov matrix.
We recall that the \emph{spectral radius\/} of a matrix $M$ is
the maximum of the moduli of all the eigenvalues of $M$ and
it will be denoted here by $\sigma(M).$

\begin{proposition}\label{propositionmonotoneTopEnt}
Let $X$ be a topological graph, let $\map{\phi}{X}$  be a
Markov map with respect to a Markov partition $Q$ of $X$ and let
$M$ be a Markov  matrix of $\phi$ with respect to $Q.$ Then,
\[
    h(\phi) = \log\max\{\sigma(M),1\}.
\]
\end{proposition}

To  compute the spectral radius of a Markov matrix  we will use the
\emph{rome method\/} proposed in \cite{bgmy} (see also \cite{alm}).
To this purpose we will introduce some notation.

Let $M = (m_{ij})$ be a $k \times k$ Markov matrix.
Given a sequence $p = (p_j)_{j=0}^{\ell(p)}$ of elements of
$\{1,2,\dots,k\}$ we define the \emph{width of $p,$} denoted by
$w(p),$  as the number $\prod_{j=1}^{\ell(p)} m_{p_{j-1}p_j}.$
If $w(p)\ne 0$ then $p$ is called a \emph{path of length $\ell(p)$}.
A \emph{loop\/} is a path such that $p_{\ell(p)} = p_0$ i.e.\ that
begins and ends at the same index.
The words ``path" and ``loop"  in this setting are inherited from
the analogous notions in the Markov graph.

A subset $R$ of $\{1,2,\dots,k\}$ is called a \emph{rome\/}
if there is no loop outside $R,$ i.e.\ there is no path
$(p_j)_{j=0}^{\ell}$ such that $p_{\ell} = p_0$ and
$\set{p_j}{0 \le j \le \ell}$ is disjoint from $R.$
For a rome $R,$ a path $(p_j)_{j=0}^{\ell}$ is called \emph{simple\/}
if $p_i\in R$ for $i=0,\ell$ and $p_i\notin R$ for $i=1,2,\dots,\ell-1.$
Of course, we can define a rome using the vertices in the Markov graph
associated with the matrix instead of the matrix itself.

If $R=\{r_1,r_2,\dots,r_{m}\}$ is a rome of a matrix $M$ then we
define an $m \times m$ matrix $M_R(x)$ whose entries are real functions
by setting
$M_R(x) := (a_{ij}(x)),$ where
$a_{ij}(x) := \sum_p w(p) \cdot x^{-\ell(p)},$
where the sum is taken over all simple paths
starting at $r_i$ and ending at $r_j.$

\begin{theorem}[{\cite[Theorem~1.7]{bgmy}}]\label{theoremrome}
Let $\mathbf{I}_m$ be the identity matrix of size $m\times m.$
If $R$ is a rome of cardinality $m$ of a $k\times k$ matrix $M$ then
the characteristic polynomial of $M$ is equal to
\[
    (-1)^{k-m} x^k \det(M_R(x) - \mathbf{I}_{m}).
\]
\end{theorem}

\subsection{Transitivity and Markov matrices}\label{TransitivityMarkovmatrices}

The aim of this subsection is to establish and prove the following result:

\begin{theorem}\label{theoremTransitivityexpansivemap}
Let $X$ be a topological graph and let $\map{f}{X}$  be a
$Q$-expansive Markov map with respect to a Markov partition $Q$ of $X.$
Then $f$ is transitive if and only if the Markov  matrix of $f$
with respect to $Q$ is irreducible but not a permutation matrix.
\end{theorem}

This result is well known when $X$ is a closed interval
of the real line and the map is piecewise affine
(see \cite[Theorem~3.1]{bc}) but we aim at extending it to the general
setting of graphs.
Its proof in this more general case goes along the lines of the one
from \cite{bc} for the interval but we will sketch it here
for completeness.

In any case we need to recall the definition of irreducibility
and establish what we understand by a piecewise expansive graph map.

A $k\times k$ matrix $M$ is called \emph{reducible\/} if there exists
a permutation matrix $P$ such that
\[
  P^T M P = \left( \begin{array}{c|c}
       M_{11} & \mathbf{0} \\ \hline
       \rule{0pt}{11pt} M_{21} & M_{22}
  \end{array}\right)
\]
where $M_{11}, M_{21}$ and $M_{22}$ are block matrices, and
$\mathbf{0}$ is a block matrix whose entries are all 0.
A matrix $P=(p_{ij})_{i,j=1}^k$  is a permutation matrix whenever
$p_{ij} \in \{0,1\}$ for all $0 \le i,j \le k$ and in each row and in each
column there is exactly one non-zero element. Observe that if $P$ is a
permutation matrix, then $P^{-1} = P^T.$

The matrix $M$ is called \emph{irreducible\/} if it is not reducible
or, equivalently (see \cite{g}),
if for every $0 \le i, j \le k$ there is a
natural number $n = n(i,j)$ such that the $i,j$-entry of $M^n$ is
strictly positive.
In the case of a Markov matrix of a Markov partition of $X,$
if we set $M^n= (m^{(n)}_{ij})_{i,j=1}^k,$
then $m^{(n)}_{ij}$ is the number of paths of length $n$
in the Markov graph starting  at the vertex $v_i$ and ending
at the vertex $v_j.$ In this context, $M$ is irreducible
if and only if there exists a path
from the vertex $v_i$ to the vertex $v_j$
for every $0 \le i,j \le k.$
In particular $f(I)$ is a (non-degenerate) interval
for every basic interval $I$.

To define the notion of $Q$-expansive graph map we need to define a
distance on basic intervals of graphs.

Let $X$ be a topological graph and let $Q \subset X$ be a finite set
such that $Q \supset V(X)$ and the closure of every connected component
of $X \setminus Q,$ called a \emph{$Q$-basic interval\/},
is an interval of $X$
(formally the notion of $Q$-basic interval is only defined for
Markov partitions of graph maps but we use it here by
abusing of notation for simplicity).
Every $Q$-basic interval $I$ can be endowed in many ways with a
distance $d_I$ verifying that the \emph{length of $I$}, defined as
$\max_{x,y\in I} d_I(x,y),$ is 1.
For instance, we can fix a homeomorphism
$\map{\mu_I}{I}[{[0,1]}]$ and set
$
d_I(x,y) := \abs{\mu_I(x) -\mu_I(y)}
$
for every $x,y \in I.$
Given a connected set $W \subset I$ we  define the \emph{length of $W$\/} by
\[
   \spnorm[I]{W} := \max\set{d_I(x,y)}{x,y \in W}.
\]
From above it follows that $\spnorm[I]{I} = 1$ and $\spnorm[I]{W} \le 1.$

Now we are ready to define the notion of:

\begin{definition}[piecewise expansiveness]\label{expansiveMM}
Let $X$ be a (topological) graph and let $\map{f}{X}$ be a
Markov map with respect to a Markov invariant set $Q.$

We say that $f$ is \emph{expansive on $I$} if
$f(I)$ is not a point and
\begin{widthitemize}
\item \textbf{when $\boldsymbol{f(I) \in \SBI}$:} $f$ verifies
\[
 d_{f(I)}(f(x), f(y)) = \lambda_I d_I(x,y) = d_I(x,y)
\]
with $\lambda_I = 1$ for every $x,y \in I;$
\item \textbf{when $\boldsymbol{f(I)}$ contains more than one $\boldsymbol{Q}$-basic interval:}
there exists $\lambda_I > 1$ such that
\[
 d_{J}(f(x), f(y)) \ge \lambda_I d_I(x,y)
\]
for every $x,y \in I$ such that $\chull{f(x),f(y)}[f(I)] \subset J \in \SBI.$
\end{widthitemize}
Observe that when $f$ is expansive on $I$ then $f\evalat{I}$ is  one-to-one.

We say that $f$ is \emph{$Q$-expansive\/} if it is
expansive on every $Q$-basic interval.
\end{definition}

\begin{proof}[Proof of Theorem~\ref{theoremTransitivityexpansivemap}]
First we will perform the simple exercise of proving that
if $f$ is transitive then
the Markov matrix of $f$ with respect to $Q$ is irreducible but not a permutation matrix.

First we assume that the Markov matrix of $f$ with respect to $Q$ is a permutation matrix.
This is equivalent to say that
we can label the set of all $Q$-basic intervals as
$I_0, I_1, \dots, I_{m-1}$ so that $f(I_i) = I_{i+1 \pmod{m}},$
and $f\evalat{I_i}$ is monotone for every $i = 0,1,\dots, m-1$.
In these conditions $f$ cannot have a dense orbit and thus it cannot be transitive.

On the other hand, by transitivity,
the image of every $Q$-basic interval is different from a point
(otherwise, again, we get a contradiction with the existence of a dense $f$-orbit).
Consequently, since $f$ is Markov with respect to a
Markov partition $Q$ (in particular $Q$ is $f$-invariant),
it follows that for every $I \in \SBI$,
$f(I)$ is an interval which is a union of $Q$-basic intervals and $f\evalat{I}$ is monotone.
It follows inductively that $f^k(I)$ is a union of $Q$-basic intervals for every $k \ge 1$.

Now we choose two arbitrary intervals $I, J \in \SBI$.
Since $f$ is transitive there exists a positive integer $n$ such that
\[ f^n(I) \cap \Int(J) \supset f^n(\Int(I)) \cap \Int(J) \ne \emptyset. \]
Since $f^n(I)$ is a union of $Q$-basic intervals and two different $Q$-basic intervals
have disjoint interiors, $f^n(I) \supset J.$ This means that there exists a
$Q$-basic interval $J_{n-1} \subset f^{n-1}(I)$ such that $J_{n-1}$ $f$-covers $J$.
Analogously, there exists a $Q$-basic interval $J_{n-2} \subset f^{n-2}(I)$ such that $J_{n-2}$ $f$-covers $J_{n-1}$.
Iterating this argument we obtain a path
$I \longrightarrow J_1 \longrightarrow J_2 \longrightarrow \cdots \longrightarrow J_{n-1} \longrightarrow J$
from $I$ to $J$ in the Markov graph of $f$ with respect to $Q$.
Consequently, the Markov matrix of $f$ with respect to $Q$ is irreducible.
This ends the proof of the ``only if part'' of the theorem.

Now we prove the ``if part'' following the ideas of \cite{bc}.
So, we assume that the Markov  matrix of $f$ with respect to $Q$
is irreducible but not a permutation matrix.
The fact that the Markov  matrix of $f$ with respect to $Q$
is not a permutation matrix tells us that
\[
\set{I}{I\in \SBI\text{ and $I$ $f$-covers at least two basic intervals}} \ne \emptyset.
\]
Hence,
\[
\lambda_f := \min \set{\lambda_I}{I\in \SBI\text{ and $I$ $f$-covers at least two basic intervals}} > 1.
\]

We claim that $U \cap \left(\cup_{k \ge 0}f^{-k}(Q) \right) \ne \emptyset$
for every connected non-empty open set $U \subset X.$
To prove the claim assume by way of contradiction that
$f^k(U) \cap Q = \emptyset$ for every $k \ge 0.$
Then, there exists a sequence of $Q$-basic intervals
$\{J_k\}_{k=0}^\infty$ such that $U \subset \Int(J_0)$ and
$f^k(U) \subset \Int(J_k) \subset f(J_{k-1})$ for every $k \ge 1.$
Moreover, from Definition~\ref{expansiveMM} we have
\begin{equation}\label{lengthofiteratedimages}
\begin{split}
\spnorm[J_{k+1}]{f^{k+1}(U)}
   & \ge \lambda_{J_k} \spnorm[J_k]{f^{k}(U)}
     \ge \lambda_{J_k} \lambda_{J_{k-1}} \spnorm[J_{k-1}]{f^{k-1}(U)}
     \ge \cdots\\
   & \ge \spnorm[J_0]{U} \prod_{i=0}^k\ \lambda_{J_i}
\end{split}
\end{equation}
for every $k \ge 0.$

Assume that $\lambda_{J_i} \ge \lambda_f > 1$ for infinitely many indices $i.$
Then, since $\lambda_{I} \ge 1$ for every $I \in \SBI,$ the sequence
\[
  \left\{\prod_{i=0}^k\ \lambda_{J_i}\right\}_{k=0}^\infty
\]
is non-decreasing, and hence
\[
  \lim_{k\to\infty} \spnorm[J_{k+1}]{f^{k+1}(U)} \ge
  \spnorm[J_0]{U} \lim_{k\to\infty}\prod_{i=0}^k\ \lambda_{J_i} = \infty.
\]
This is a contradiction because, for every $k \ge 0,$
$J_{k+1}$ is a $Q$-basic interval and
$f^{k+1}(U) \subset J_{k+1} \subset f(J_k);$ which implies
$\spnorm[J_{k+1}]{f^{k+1}(U)} \le \spnorm[J_{k+1}]{J_{k+1}} = 1.$

From the part of the claim already proven, there exists $m \in \N$
such that $\lambda_{J_k} = 1$ (that is, $f(J_k) \in \SBI$)
for every $k \ge m.$
Thus, $f(J_k) = J_{k+1}$ for every $k \ge m$
because $J_{k+1} \subset f(J_{k}).$
Since the number of $Q$-basic intervals is finite,
there exist $\ell \ge m$ and $t \ge 1$ such that $J_\ell = J_{\ell+t}.$

We already know that there exists a basic interval $I\in\SBI$
that $f$-covers at least two basic intervals. So,
$I \notin \{J_\ell, J_{\ell+1},\dots,J_{\ell+t-1}\},$
and in the Markov graph of $f$ with respect to $Q$ there does not exist
any path starting in a $Q$-basic interval from
$\{J_\ell, J_{\ell+1},\dots,J_{\ell+t-1}\}$ and ending at $I.$
This contradicts the irreducibility of the
Markov matrix of $f$ with respect to $Q$ and ends the proof of the claim.

Since for every non-empty open set $V$ there exists
$I \in \SBI$ such that $V \cap \Int(I) \ne \emptyset,$
to prove that $f$ is transitive it is enough to show that
for every non-empty open set $U \subset X$ and every $I \in \SBI$
there exists a positive integer $n$ such that $f^n(U) \supset \Int(I).$

Let $J \in \SBI$ be such that $U \cap \Int(J) \ne \emptyset.$
By the above claim with $U$ replaced by a connected component of $U \cap\, \Int(J),$
it follows that there exists
$x \in \left(U \cap \Int(J) \right) \cap \left(\cup_{k \ge 0}f^{-k}(Q) \right).$ So,
again by the claim for a connected component of
$\left(U \cap \Int(J) \right) \setminus \{x\}$ instead of $U,$
we obtain that
\[ \Card\left(\left(U \cap \Int(J) \right) \cap \left(\cup_{k \ge 0}f^{-k}(Q) \right)\right) \ge 2. \]
Therefore, there exist $x,y \in U \cap \Int(J)$ with $x \ne y$
and $k_x, k_y \in \N$ such that
$\chull{x,y}[J] \subset U \cap \Int(J),$
$1 \le k_x \le k_y,$
$f^{k_x}(x),  f^{k_y}(y) \in Q,$
and, concerning the preimages of $Q,$
$\left(U \cap \Int(J) \right) \cap f^{-k}(Q) = \emptyset$ for $k = 0,1,\dots, k_x-1$
and
$\bigl(\left(U \cap \Int(J) \right) \setminus \{x\} \bigr) \cap f^{-k}(Q) = \emptyset$ for $k = k_x,k_x+1,\dots, k_y-1.$

Consequently, as in the proof of the above claim and using
the fact that $f$ is $Q$-monotone,
it follows inductively that there exist
$Q$-basic intervals $J_0 = J, J_1,\dots,J_{k_y-1}$ such that
\begin{align*}
& \chull{x,y}[J_0] \subset U \cap \Int(J_0),\\
& f^k\Bigl(\chull{x,y}[J_0]\Bigr) = \chull{f^k(x),f^k(y)}[J_k] \subset \Int(J_k)\text{ for $k = 1,2,\dots,k_x-1$ and}\\
& f^k\Bigl(\chull{x,y}[J_0]\mkern-6mu\setminus\{x\}\Bigr) = \chull{f^k(x),f^k(y)}[J_k]\mkern-6mu\setminus\{f^k(x)\} \subset \Int(J_k)\\
& \hspace*{18.05em}\text{for $k = k_x,k_x+1,\dots, k_y-1$}
\end{align*}
(recall that $f(Q) \subset Q$ and, hence, $f^k(x) \in Q$ for every $k \ge k_x$).
Moreover,
\[
 f^{k_y}\Bigl(\chull{x,y}[J_0]\Bigr) = \chull{f^{k_y}(x), f^{k_y}(y)}[f\Bigl(J_{k_y-1}\Bigr)] \subset f\Bigl(J_{k_y-1}\Bigr)
\]
with $f^{k_y}(x), f^{k_y}(y) \in Q.$
On the other hand, from above it follows that $f^k(x), f^k(y) \in J_k$
for $k = 0,1,\dots, k_y-1,$ and from Definition~\ref{expansiveMM},
$f\evalat{J_k}$ is one-to-one.
Hence, $f^k(x) \ne f^k(y)$ for $k = 0,1,\dots, k_y,$
and consequently there exists $J_{k_y} \in \SBI$ such that
\[
  J_{k_y} \subset \chull{f^{k_y}(x), f^{k_y}(y)}[f\Bigl(J_{k_y-1}\Bigr)]
  \andq
  f^{k_y}\Bigl(\chull{x,y}[J_0]\Bigr) \supset J_{k_y}.
\]

Since the Markov  matrix of $f$ with respect to $Q$ is irreducible
there exists a path of length $r \ge 0$
from $J_{k_y}$ to $I$ in the Markov graph of $f$ with respect to $Q$.
Then, from the definition of path and $f$-covering it follows that
$f^r\bigl(J_{k_y}\bigr) \supset I.$
Consequently,
\[
f^{r+k_y}(U) \supset f^{r+k_y}\Bigl(\chull{x,y}[J_0]\Bigr) \supset f^r\bigl(J_{k_y}\bigr) \supset I.
\]
This ends the proof of the proposition.
\end{proof}


\section{Examples}\label{sectionexamples}
This section is devoted to construct the
Examples~\ref{exampleBCNintroduction},
         \ref{examplefirstexamplebcnintroduction} and
         \ref{examplemontevideuexampleintroduction}.
For clarity, each example will be dealt in a subsection.
Moreover, we will start with to additional subsections: the first one
being introductory to explain the philosophy of the constructions that
we make, while the second one is devoted to introduce a couple of
specific technical auxiliary results.

\subsection{Philosophy and introduction:
on how to explicitly specify the families of circle maps in the
examples}\label{Examples:PhilandIntro-HowTo}
We have to construct examples of totally transitive graph maps
with arbitrarily small topological entropy and an arbitrarily
large boundary of cofiniteness.

To do this we start with circle maps that verify these properties
and, with the help of the result from the next subsection,
we extend these circle examples to any graph that is not a tree while
keeping the stated properties.

A natural idea to define these circle maps is:
\emph{start by fixing a family of
non-degenerate intervals, and for each of them take the circle map of
degree one with the prescribed interval as rotation interval and
having minimum entropy among all maps with these properties.}
Indeed, \emph{the appropriate choice of the rotation interval}
and the fact that we take minimum entropy maps
should guarantee that that the boundary of cofiniteness goes to infinity
(because the rotation interval for these maps completely determines the set of periods),
that the entropies converge to zero and, finally, the total
transitivity should be inherited from the non-degeneracy of the
rotation interval.

Les us see with more detail how this can be achieved.
In particular this will explain the ``mysterious'' assumptions in the
examples.

In this discussion and survey on minimum entropy maps
depending on the rotation interval we will follow the approach from
\cite{almm, alm}.

For $c,d \in \R,$ $c < d$ and $z > 1$ we define
\[
R_{c,d}(z) := \sum_{q \in M(c,d)} z^{-q}.
\]
Then, one can show that $R_{c,d}(z) = \tfrac{1}{2}$ has a unique solution
$\beta_{c,d}$ and $\beta_{c,d} > 1.$

The following result is what makes possible the strategy proposed above.

\begin{theorem}\label{LBE}
Let $f$ be a circle map of degree $1$ with rotation interval $[c,d]$
with $c<d.$ Then $h(f) \ge \log \beta_{c,d}.$
Moreover, for every $c,d \in \R,$ $c < d$
there exists a circle map of degree 1, $f_{c,d},$
having rotation interval $[c,d]$ and entropy
$h(f_{c,d}) = \log \beta_{c,d}.$
\end{theorem}

From the proof of this theorem it follows that
the circle map $f_{c,d}$
has as a lifting (see Figure~\ref{figuregraphF5firstexamplebcn}
for an example with $c = \tfrac{1}{2}$ and $d = \tfrac{7}{10}$)
\[
  G_{c,d}(x) := \begin{cases}
                  \beta_{c,d} x + b & \text{if $0 \le x \le u,$} \\
                  \beta_{c,d} (1-x) + b + 1 & \text{if $u \le x \le 1,$} \\
                  G_{c,d}(x - \floor{x}) + \floor{x} & \text{if $x \notin [0,1],$}
                \end{cases}
\]
where $\floor{\cdot}$ denotes the integer part function,
\[
   b := (\beta_{c,d}-1)^2 \sum_{n=1}^{\infty} \floor{nc} \beta_{c,d}^{-n-1},
\]
and the continuity of $G_{c,d}$ gives $u := \tfrac{\beta_{c,d} + 1}{2\beta_{c,d}}.$

\begin{remark}
For $c\in \R$, $c > 0,$ and $z > 1$ we define
\[
T_c(z) := \sum_{n=0}^\infty z^{-\floor{\tfrac{n}{c}}},
\]
and, for definiteness, we set $T_0(z) \equiv 0.$
Then, for $c,d \in \R,$ $c < d,$ $c \in [0,1),$ and $z > 1$ we define
\[
  Q_{c,d}(z) := z + 1 + 2\left(\frac{z}{z-1} - T_{1-c}(z) - T_{d}(z)\right)
\]
(observe that if $[c,d]$ is a rotation interval then, by replacing the
lifting $F$ used to compute the rotation interval by
$F - \floor{c},$ we get the new rotation interval
$\bigl[c-\floor{c},d-\floor{c}\bigr]$
with $c-\floor{c} \in [0,1);$
this is the reason that the assumption $c \in [0,1)$ above is not
restrictive).

Concerning the map $Q_{c,d}$ one can show that
\[
   Q_{c,d}(z) = (z -1) \left(1 - 2R_{c,d}(z)\right).
\]
Hence, $\beta_{c,d}$ is the largest root of the equation
$Q_{c,d}(z) = 0.$
This observation gives a much easier way of calculating the numbers
$\beta_{c,d}$.
\end{remark}

\begin{remark}[On when \boldmath$\lim\beta_{c,d} = 1$]\label{PropsGoingZero}
The numbers $\beta_{c,d}$ have the following important properties:
\begin{itemize}
 \item If $c \le a < b \le d$ and $\{a,b\} \ne \{c,d\},$ then
       $\beta_{c,d} > \beta_{a,b}.$
       This implies that if we have a decreasing sequence of intervals
       $\{[c_n,d_n]\}_{n=0}^\infty$ whose lengths do not converge to 0,
       then the sequence $\{\beta_{c_n,d_n}\mkern-2mu\}_{n=0}^\infty$
       is bounded away from 1.
 \item Assume that $c < \tfrac{p}{q} < d$ with $p$ and $q$ coprime.
       Then, $\beta_{c,d} > 3^{\tfrac{1}{q}}.$
       This implies that given a sequence of intervals
       $\{[c_n,d_n]\}_{n=0}^\infty$ for which there exist $M$ such that
       $\min M(c_n, d_n) \le M$ for every $n$
       (for instance
        $
          \left[\tfrac{n-1}{2n}, \tfrac{n+1}{2n}\right] =
          \left[\tfrac{1}{2}-\tfrac{1}{n}, \tfrac{1}{2}+\tfrac{1}{n}\right]
        $),
       then $\beta_{c_n,d_n}\mkern-4mu > 3^{\tfrac{1}{M}}$ and, hence,
       the sequence $\{\beta_{c_n,d_n}\mkern-2mu\}_{n=0}^\infty$
       is bounded away from 1.
\end{itemize}
Summarizing, if we want to achieve
$\lim_{n\to\infty} \beta_{c_n,d_n} = 1$
for a given sequence of intervals $\{[c_n,d_n]\}_{n=0}^\infty,$
then we need (as necessary but not sufficient conditions)
that the lengths $d_n-c_n$ of the intervals converge to zero,
and $\{\min M(c_n, d_n)\}_{n=0}^\infty$ is unbounded.
\end{remark}

Taking all the above comments and results into account,
Theorem~\ref{LBE} gives the following procedure to build our examples:
\begin{enumerate}[{\footnotesize$\blacktriangleright$}]
 \item Choose, a sequence of closed intervals of the real
       line $\{[c_n,d_n]\}_{n=0}^\infty$
       with rational endpoints such that:
       \begin{itemize}
         \item $\lim_{n\to\infty} \beta_{c_n,d_n} = 1$ (see Remark~\ref{PropsGoingZero}), and
         \item the boundary of cofiniteness of $M(c_n,d_n)$
               (defined as the boundary of cofiniteness of a map $f$
               but replacing $\Per(f)$ by $M(c_n,d_n)$)
               exists and goes to infinity with $n.$
       \end{itemize}
  \item Compute the numbers $\beta_{c_n,d_n},$ $b$ and $u$ defined above
        to get the map $f_{c_n,d_n}$ determined.
        Observe that we automatically have
        \[ \lim_{n\to\infty} h\bigl(f_{c_n,d_n}\bigr) = \lim_{n\to\infty} \log \beta_{c_n,d_n} = \log\lim_{n\to\infty} \beta_{c_n,d_n} = 0.\]
  \item Check that $\Per(f_{c_n,d_n}) = M(c_n,d_n)$ and compute this set to get
        \[ \lim_{n\to\infty} \bc\bigl(f_{c_n,d_n}\bigr) = \lim_{n\to\infty} \bc\bigl(M(c_n,d_n)\bigr) = \infty.\]
  \item Show that the map $f_{c_n,d_n}$ is totally transitive
        for every $n$, if that is the case.
\end{enumerate}

This method, while giving an effective procedure to construct the
sequences of maps that we are looking for, has two serious drawbacks:
First, in this setting it is very difficult to show that the map
$f_{c_n,d_n}$ is totally transitive and
to compute $\Per\bigl(f_{c_n,d_n}\bigr)$ and $\beta_{c_n,d_n}$
for every $n$ (essentially, we only can do it numerically).
Second, we cannot extend these models to graphs in
such a way that we can also extend the study of the transitivity,
$\Per\bigl(f_{c_n,d_n}\bigr)$ and $h\bigl(f_{c_n,d_n}\bigr)$
from the circle (indeed, for graphs there is no analogous to the theorem
stating that the topological entropy of a piecewise linear circle map
such that the absolute value of its slopes is constant,
is precisely the logarithm of this number).

To solve all these problems it is much better
to find a Markov partition for every map $f_{c_n,d_n}$ and
use it to prove that it is totally transitive and
to compute $\Per\bigl(f_{c_n,d_n}\bigr)$ and $\beta_{c_n,d_n}.$
Moreover, this also gives a good tool to extend the circle models
to graphs while keeping the transitivity,
the sets of periods, and the fact that
$\left\{h\bigl(f_{c_n,d_n}\bigr)\right\}_{n=0}^\infty$
converges to zero
(despite of the fact these entropies increase  a little bit).

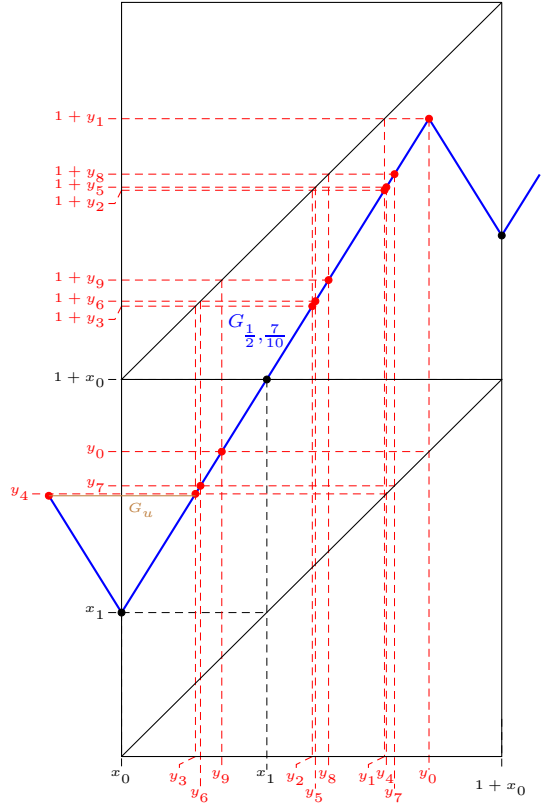
\begin{figure}
\def\Gmap{G\mkern-2mu_{\mbox{\tiny$\tfrac{1}{2},\mkern-4mu\tfrac{7}{10}$}}}
\begin{tikzpicture}[scale=5]
\newdimen\unpunt \unpunt=1pt
\foreach \x / \y / \dx / \dy / \lab in {0/0.381736793335882413710052829071/0/0/0, 0.3821736793335882413710052829071/1/0/0/1}{ 
           \ifnum\lab=0\def\labseg{1}\else\def\labseg{0}\fi
           \node[below] at ([yshift=-0.3]\x,\dy) {\supertiny$x_{\lab}$};
           \ifnum\lab=0 \node[left] at ([xshift=-0.3]\dx,\y) {\supertiny$x_{\labseg}$};
           \else        \node[left] at ([xshift=-0.3]\dx,\y) {\supertiny$1+x_{\labseg}$}; \fi
           \draw[densely dashed, ultra thin] ([xshift=-1]\dx,\y) -- (\y, \y);
           \draw[densely dashed, ultra thin] (\x, \y) -- ([yshift=-1]\x, \dy);
} \draw[densely dashed, ultra thin] (1,0.1) -- ([yshift=-1]1,0); \node[below] at ([yshift=-1]1,0) {\supertiny$1+x_0$};
\foreach \x / \y / \dx / \dy / \lab in {
                            0.194388387422006/0.696569394866251/-0.2/1pt/3, 0.207581634338377/0.717937256930439/0/-0.05/6, 0.263632178542455/0.808717055471868/0/0/9,
                            0.501758814501575/1.194388387422006/1pt/1pt/2,    0.509904775181846/1.207581634338377/0/-0.05/5, 0.544512293110403/1.263632178542455/0/0/8,
                            0.691539800875845/1.501758814501575/1pt/1pt/1,    0.696569394866251/1.509904775181846/0/0/4, 0.717937256930439/1.544512293110403/0/-0.05/7,
                            0.808717055471868/1.691539800875845/0/0/0}{
           \ifnum\lab=9\def\labseg{0}\else\pgfmathtruncatemacro{\labseg}{1 + \lab};\fi
           \ifnum\lab=0\node[below, red] at ([yshift=-0.3]\x,\dy) {\supertiny$y_{\lab}$};\else\ifnum\lab>3\node[below, red] at ([yshift=-0.3]\x,\dy) {\supertiny$y_{\lab}$};\fi\fi
           \ifdim\y\unpunt<\unpunt
                    \node[left, red] at ([xshift=-0.3]\dx,\y) {\supertiny$y_{\labseg}$};
                    \draw[densely dashed, ultra thin, red] ([xshift=-1]\dx,\y) -- (\y, \y);
                    \draw[densely dashed, ultra thin, red] (\x, {1+\x}) -- ([yshift=-1]\x, \dy);
           \else    \ifnum\labseg=2\else\ifnum\labseg=3\else\node[left, red] at ([xshift=-0.3]\dx,\y) {\supertiny$1+y_{\labseg}$};\fi\fi
                    \draw[densely dashed, ultra thin, red] ([xshift=-1]\dx,\y) -- (\x, \y);
                    \ifdim\y\unpunt<1.505pt\draw[densely dashed, ultra thin, red] (\x, {1+\x}) -- ([yshift=-1]\x, \dy);
                    \else \draw[densely dashed, ultra thin, red] (\x, \y) -- ([yshift=-1]\x, \dy);\fi
           \fi
}
\foreach \y / \labseg in {1.194388387422006/3, 1.501758814501575/2}{
   \draw[densely dashed, thin, red] (0,\y) -- ([xshift=-0.5, yshift=-1]0,\y) -- ([xshift=-1, yshift=-1]0,\y);
   \node[left, red] at ([xshift=-0.3, yshift=-1]0,\y) {\supertiny$1+y_{\labseg}$};
}
\foreach \x / \lab in {0.194388387422006/3, 0.501758814501575/2, 0.691539800875845/1}{
   \draw[densely dashed, thin, red] (\x,0) -- ([xshift=-1.2, yshift=-0.6]\x,0) -- ([xshift=-1.2, yshift=-1.2]\x,0);
   \node[below, red] at ([xshift=-1.2, yshift=-0.5]\x,0) {\supertiny$y_{\lab}$};
}
\draw (1,2) -- (0,2) -- (0,0) -- (1,0) -- (1,2);
\draw (1,2) -- (0,1) -- (1,1) -- (0,0);
\node[color=blue, left] at (0.47,1.13) {\tiny$\Gmap$};
\draw[thick, color=blue] (-0.1912945, 0.691539800875845) -- (0,0.381736793335882413710052829071) -- (0.808717055471868,1.691539800875845) --
                         (1, 1.381736793335882413710052829071) -- (1.1, 1.543697395);
\foreach \x / \y / \c in { -0.191282945/0.691539800875845/red,         1/1.381736793335882413710052829071/black, 0/0.381736793335882413710052829071/black,
                            0.3821736793335882413710052829071/1/black, 0.194388387422006/0.696569394866251/red,        0.207581634338377/0.717937256930439/red,
                            0.263632178542455/0.808717055471868/red,   0.501758814501575/1.194388387422006/red,        0.509904775181846/1.207581634338377/red,
                            0.544512293110403/1.263632178542455/red,   0.691539800875845/1.501758814501575/red,        0.696569394866251/1.509904775181846/red,
                            0.717937256930439/1.544512293110403/red,   0.808717055471868/1.691539800875845/red}  { \filldraw[\c] (\x,\y) circle (0.009); }
\node[color=brown, below right] at (-0.01,0.7) {\supertiny$G_u$};
\draw[brown] (-0.1912945, 0.691539800875845) -- (0.194388387422006, 0.691539800875845);
\end{tikzpicture}\vspace*{-2ex}
\caption{The graph of the map $\Gmap$ used in
Example~\ref{examplefirstexamplebcnintroduction}
as a model for the map $F_5$.
To better show the dynamics of the orbits, in this picture
they are labelled so that
$\Gmap(x_0) = x_1,\ \Gmap(x_1) = x_0 + 1,$
$\Gmap(y_i) = 1+y_{i+1}$ for $i \in \{0,1,2,4,5,7,8\},$
$\Gmap(y_i) = y_{i+1}$ for $i \in \{3,6\}$ and
$\Gmap(y_9) = y_{0}$
(observe that, for the points $y_i,$ this labelling is different than
the linearly ordered one introduced below).}\label{figuregraphF5firstexamplebcn}
\vspace*{-2ex}
\end{figure}

To find these Markov partitions we note that the maps
$G_{c_n,d_n}\mkern-4mu\evalat{[0,1]}$ are bimodal,
and monotone on the intervals $[0,u]$ and $[u, 1]$.
So, every Markov partition must include $\{0,u\}.$
Moreover, the existence of such a partition depends on the finiteness
of the orbits of $\emap{0}$ and $\emap{u}$ (on the circle).
It follows that if $c_n = \tfrac{p_n}{q_n}$ with $p_n$ and $q_n$
relatively prime (respectively, $d_n = \tfrac{r_n}{s_n}$
with $r_n$ and $s_n$ relatively prime),
then $\emap{0}$ (respectively $\emap{u}$) is a periodic point of
$f_{c_n,d_n}$ of period $q_n$ (respectively $s_n$)
and $Q_n = \eexp^{-1}\Bigl(\Orb_{f_{c_n,d_n}}\mkern-8mu(\emap{0})\Bigr)$
(respectively $S_n = \eexp^{-1}\Bigl(\Orb_{f_{c_n,d_n}}\mkern-8mu(\emap{u})\Bigr)$)
is a twist lifted periodic orbit with rotation number $\tfrac{p_n}{q_n}$
(respectively $\tfrac{r_n}{s_n}$).
Hence, $Q_n \cup S_n \supset \{0,u\}$ is a Markov partition with respect
$G_{c_n,d_n}$ (see Figure~\ref{figuregraphF5firstexamplebcn}).\enlargethispage{5mm}
Notice also that
\begin{align*}
& \Card\left(Q_n \cap [0,1]\right) = \Card\left(Q_n \cap [0,u]\right) = q_n,\\
& \Card\left(S_n \cap [0,1]\right) = \Card\left(S_n \cap [0,u]\right) = s_n
\end{align*}
and $G_{c_n,d_n}\mkern-4mu\evalat{Q_n}$
(respectively $G_{c_n,d_n}\mkern-4mu\evalat{S_n}$)
is determined by Remark~\ref{rem:liftedperiodicorbitandtwist}
and the numbers $p_n$ and $q_n$ (respectively $r_n$ and $s_n$).
So, to determine the Markov partition $Q_n \cup S_n$ and
$G_{c_n,d_n}\mkern-4mu\evalat{Q_n \cup S_n}$
it is enough to determine the relative positions of the points from
$Q_n \cap [0,u]$ and $S_n \cap [0,u].$ It turns out that if the
endpoints of the rotation interval are appropriately chosen, then
there is a formula for the relative positions of the points from
$Q_n \cap [0,u]$ and $S_n \cap [0,u]$ that depends solely on $n$.
For instance (to fix ideas),
in Example~\ref{examplefirstexamplebcnintroduction}
we consider the family of liftings with rotation interval
$\left[\tfrac{1}{2}, \tfrac{n+2}{2n}\right]$ and minimum entropy
relative to the rotation interval.
Then, $\Card\left(Q_n \cap [0,u]\right) = 2$
and $Q_n$ is a twist lifted periodic orbit with rotation number
$\tfrac{1}{2},$
$\Card\left(S_n \cap [0,u]\right) = 2n$
and $S_n$ is a twist lifted periodic orbit with rotation number
$\tfrac{n+2}{2n}.$ Without loss of generality we can write
$Q_n \cap [0,u] = \{x_0, x_1\}$ with $x_0 = 0 < x_1$ and
$S_n \cap [0,u] = \{y_0,y_1,\dots,y_{2n-1}\}$ with $y_0 < y_1 < \dots < y_{2n-1} = u.$
Then, by explicitly computing some particular examples of this family
(as it is done in Figure~\ref{figuregraphF5firstexamplebcn} for $n=5$),
one can see that that the $2n+2$ points of $(Q_n \cup S_n) \cap [0,u]$
are organized in the following way:
\[
 x_0 = 0 < y_0 < y_1 < \dots < y_{n-3} < x_1 < y_{n-2} < y_{n-1} < \dots < y_{2n-1} = u
\]
(notice that the number $n-3$ in the above formula is, indeed,
the number $s_n -r_n - 1 = 2n-(n+2)-1$).

Summarizing, the numbers $p_n,\ q_n,\ r_n$ and $s_n$ and the relative
positions of the points of the  lifted periodic orbits with rotation
numbers $\tfrac{p_n}{q_n}$ and $\tfrac{r_n}{s_n}$ in $[0,u]$
specify in a totally explicit way the Markov partitions of the functions
$f_{c_n,d_n}$ (and hence, by linearity, the maps $f_{c_n,d_n}$
themselves) in a way that easily give the totally transitive of
these maps, $\Per\bigl(f_{c_n,d_n}\bigr)$ and $\beta_{c_n,d_n}.$
Moreover, we can easily extend the circle models to graphs while keeping the
transitivity, the sets of periods and the fact that
$h\bigl(f_{c_n,d_n}\bigr)$ still converges to zero.

\subsection{Two technical auxiliary results}
The following lemma is analogous to \cite[Lemma~3.6]{arr}
with the additional assumption that the number of elements
of the partition must be even.
It will be our main tool to translate the examples from $\SI$
to any graph that is not a tree (see Figure~\ref{mapsfromarr}).
Due to the additional assumption about the parity of the number
of elements of the partition we will include the proof for
completeness.
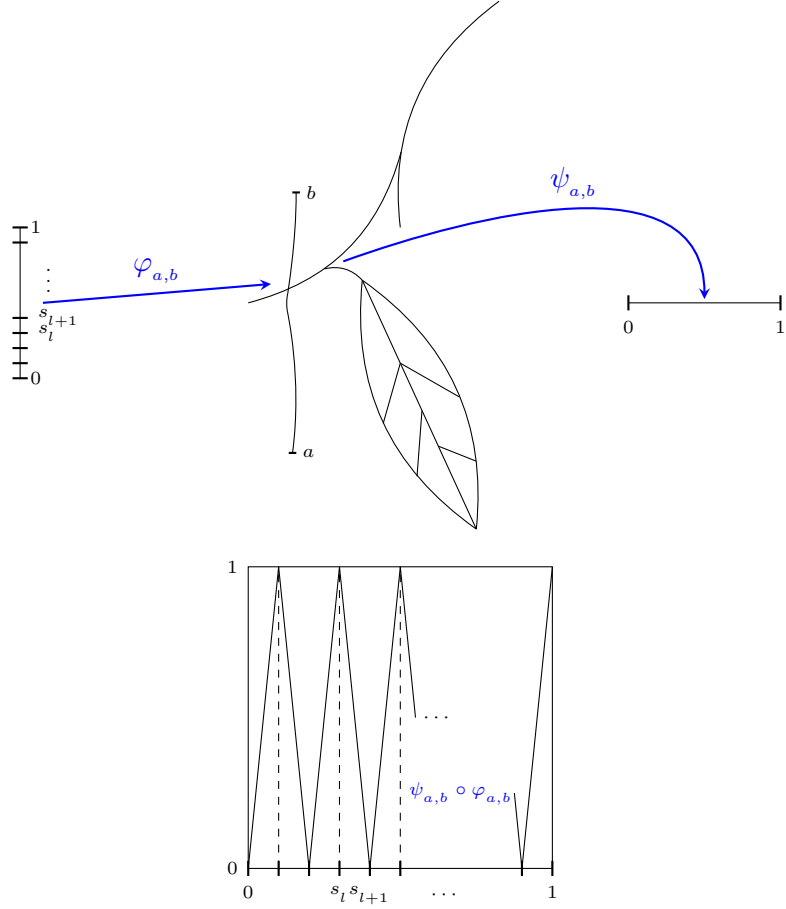
\begin{figure}
\begin{tikzpicture}
\draw[rounded corners] (6.63,1.464) .. controls (6.635,0.8) and (6.5125,0.1) .. (6.485,0)  .. controls (6.5105,-0.1) and (6.714,-1) .. (6.5825,-1.99);
\draw[thick] (6.63,1.464)+(-0.05,0) -- +(0.05,0); \node[right] at (6.63,1.464) {\tiny $b$};
\draw[thick] (6.5825,-1.99)+(-0.05,0) -- +(0.05,0); \node[right] at (6.5825,-1.99) {\tiny $a$};
\path (6,0)       edge [bend right] (8.01,2)
      (8,1)       edge [bend left]  (9.3,4)
      (7,0.45)    edge [bend left]  (7.5, 0.3)
      (7.5,0.3)   edge [bend left]  (9,-3)
      (7.5,0.3)   edge [bend right] (9,-3)
      (7.5,0.3)   edge (9,-3)
      (8,-0.8)    edge (8.78,-1.25)
      (8.285,-1.43) edge (8.22, -2.3)
      (8.5,-1.9)  edge (9, -2.1)
      (8,-0.8)    edge (7.775,-1.6);
\draw  (3,-1) -- (3,1);
\foreach \s in {-1,-0.8,-0.6,-0.4,-0.2,0.8,1}{ \draw[thick] (2.9,\s)   -- (3.1,\s); }
\node[right] at (3,-1) {\tiny 0}; \node[right] at (3,1) {\tiny 1};
\node[right] at (3.1,-0.4)     {\tiny $s_{l}$};
\node[right] at (3.1,-0.2)      {\tiny $s_{l+1}$};
\node[right] at (3.2,0.4)      {\tiny $\vdots$};
\path[->, >=stealth, blue, thick] (3.3,0) edge node[above] {$\varphi_{a,b}$} (6.3,0.25);

\draw  (11,0) -- (13,0);
\draw[thick] (11,-0.1) -- (11,0.1); \draw[thick] (13,-0.1) -- (13,0.1);
\node[below] at (11,-0.1) {\tiny 0}; \node[below] at (13,-0.1) {\tiny 1};
\path[->, >=stealth, blue, thick] (7.25,0.55) edge [out=20, in=90] node[above]  {$\psi_{a,b}$} (12,0.05);

\begin{scope}[shift={(1.5,1)}]
\draw  (4.5, -8.5) rectangle (8.5,-4.5);
\foreach \s in {4.9,5.7,6.5}{ \draw[thick] (\s,-8.6) -- (\s,-8.4);  \draw[dashed] (\s,-8.5) -- (\s,-4.5); }
\foreach \s in {4.5,5.3,6.1,8.1,8.5}{ \draw[thick] (\s,-8.6) -- (\s,-8.4); }

\node[left]  at  (4.5,-8.5){\tiny 0}; \node[left] at (4.5,-4.5) {\tiny 1};
\node[below] at (4.5,-8.6) {\tiny 0}; \node[below] at (8.5,-8.6)     {\tiny 1};
\node[below] at (5.7,-8.6) {\tiny $s_{l}$};
\node[below] at (6.1,-8.6) {\tiny $s_{l+1}$};
\node[below] at (7.1,-8.65){\tiny $\cdots$};

\draw (4.5,-8.5) -- (4.9,-4.5) -- (5.3, -8.5) -- (5.7, -4.5) -- (6.1,-8.5) -- (6.5, -4.5) -- (6.7,-6.5);
\draw (8,-7.5) -- (8.1,-8.5) -- (8.5,-4.5);
\node at (7,-6.5) {\tiny $\cdots$};
\node[blue] at (7.3,-7.5) {\tiny $\psi_{a, b}\circ\varphi_{a, b}$};
\end{scope}
\end{tikzpicture}
\caption{A sketch of a topological graph $X$ and the maps from Lemma~\ref{lemmaPhiPsi} (top picture).
The map $\psi_{a, b}\circ\varphi_{a, b}$ is shown in the bottom picture.}\label{mapsfromarr}
\end{figure}

\begin{lemma}\label{lemmaPhiPsi}
Let $X$ be a topological graph which is not an interval
and let $a, b \in V(X)$ be two different endpoints of $X.$
Then, there exist a partition of the interval $[0,1],$
$0=s_0<s_1<\cdots <s_m=1,$ with $m=m(X, a, b) \ge 5$ odd,
and two continuous surjective maps
$\map{\varphi_{a, b}}{[0,1]}[X]$ and
$\map{\psi_{a, b}}{X}[{[0,1]}]$
such that the following statements hold:
\begin{enumerate}[(a)]
\item $\varphi_{a, b}^{-1}(W) =\set{s_i}{i\in\{0,1,\dots,m\}},$ where
      \[ W := \varphi_{a, b}\bigl(\set{s_i}{i\in\{0,1,\dots,m\}}\bigr) \supset V(X), \]
      and $\varphi_{a, b}(0)=a$ and $\varphi_{a, b}(1)=b.$
\item For every $i=0, 1,\dots, m-1,$
      $\varphi_{a, b}\evalat{[s_i, s_{i+1}]}$ is injective and
      $\varphi_{a, b}\bigl([s_i, s_{i+1}]\bigr)$
      is an interval which is the closure of a connected component of
      the punctured graph $X \setminus W.$
\item If $\varphi_{a, b}(s_i) = \varphi_{a, b}(s_j)$ then $i\equiv j \pmod{2}.$
\item $\psi_{a, b}(\varphi_{a, b}(s_i))=0$ if $i$ is even and
$\psi_{a, b}(\varphi_{a, b}(s_i))=1$ if $i$ is odd
(in particular, $\psi_{a, b}(a)=0$ and $\psi_{a, b}(b)=1$).
\item
The map
$\psi_{a, b}\rule[-9pt]{0.5pt}{18pt}\mkern3mu\raisebox{-6pt}{$\scriptstyle \varphi_{a, b}([s_i, s_{i+1}])$}$
is injective and
$\psi_{a, b}\bigl(\varphi_{a, b}([s_i, s_{i+1}])\bigr) = [0,1]$
for every $i=0, 1,\dots, m-1.$
In particular, the map
$\bigl(\psi_{a, b}\circ\varphi_{a, b}\bigr)\evalat{[s_i, s_{i+1}]}$
is strictly monotone.
\end{enumerate}
\end{lemma}

In what follows the closure of a set $A \subset X$ will be denoted by
$\Clos(A).$

\begin{proof}
The existence of a surjective map $\psi_{a,b}$ which
satisfies (d--e) follows easily from the existence of the partition
$0=s_0<s_1<\cdots <s_m=1$ and a map $\varphi_{a, b}$
which satisfy statements (a--c) of the lemma.
In particular,
(d) can be guaranteed by using (c),
and (e) by (b) and (d).
So, we only have to show that there exist a partition
$0=s_0<s_1<\cdots <s_m=1$ with $m \ge 5$ odd, and
a continuous surjective map $\varphi_{a, b}$ such that (a--c) hold.

Since $X$ is not an interval and $a$ and $b$ are endpoints of $X,$
there exist $v \in V(X) \setminus\{a,b\}$ and an interval
$J \subset X$ with endpoints $v$ and $b$ such that
$J \cap V(X) = \{v,b\}.$

Let $C$ be an edge of $X$
(i.e. a connected component of $X \setminus V(X)$).
Clearly, $\Clos(C)$ is either an interval or a circuit
which contains a unique vertex of $X$.
For every $C$ such that $\Clos(C)$ is a circuit we choose a point
$v_C \in C$ (that will play the role of an artificial vertex).
Then we set
\[
 \widetilde{V}(X) := V(X) \cup
 \set{v_C}{\text{$C$ is an edge of $X$ such that $\Clos(C)$ is a circuit}}.
\]
Observe that the closure of every connected component of
$X \setminus \widetilde{V}(X)$ is an interval and
$J \cap \widetilde{V}(X) = J \cap V(X) = \{v,b\}.$
Moreover, since $\{a,b,v\} \subset V(X) \subset \widetilde{V}(X)$
and $a$ and $b$ are endpoints of $X,$
$X \setminus \widetilde{V}(X)$ has at least three connected components
and 4 vertices. So, $\Card\bigl(\widetilde{V}(X)\bigr) \ge 4.$

Since a topological graph is path connected there exists
$\map{\varphi_{a, b}}{[0,1]}[X],$
a path from $a$ to $b,$
which is continuous and onto (i.e., visits each point from $X$
going several times trough the same edge, if necessary)
and a partition of the interval $[0,1],$
$0 = s^*_0 < s^*_1 < \dots < s^*_n = 1,$
such that the following statements hold:
\begin{enumerate}[(i)]
\item $\set{s^*_j}{j\in\{0,1,\dots,n\}} := \varphi_{a, b}^{-1}\bigl(\widetilde{V}(X)\bigr)$
      with
      $\varphi_{a, b}(s^*_0) = \varphi_{a, b}(0) = a,$
      $\varphi_{a, b}\bigl(s^*_{n-1}\bigr) = v$ and
      $\varphi_{a, b}(s^*_n) = \varphi_{a, b}(1) = b,$
\item for every $j=0, 1,\dots, n-1,$
      $\varphi_{a, b}\evalat{\bigl[s^*_j,s^*_{j+1}\bigr]}$
      is injective and, hence,
      $\varphi_{a, b}\bigl(\bigl[s^*_j,s^*_{j+1}\bigr]\bigr)$
      is an interval which is the closure of a connected component of
      $X \setminus \widetilde{V}(X),$
\item $\varphi_{a, b}^{-1}(b) = s^*_{n}$ and
      $\varphi_{a, b}^{-1}\bigl(J \setminus \{v\} \bigr) = \bigl(s^*_{n-1},s^*_{n}\bigr].$
\end{enumerate}

Let $\mathcal{E}$ be the set of all connected components
of $X \setminus \widetilde{V}(X)$ which are different from
$J\setminus \{b,v\}.$
Then, for every $C \in \mathcal{E}$ we choose an arbitrary but fixed point
$\alpha_C \in C.$
By (i) and (iii),
\[
 \varphi_{a, b}^{-1}\bigl(\set{\alpha_C}{C \in \mathcal{E}}\bigr) \subset
  \bigl(s^*_0, s^*_{n-1}\bigr) \setminus \set{s^*_j}{j\in\{1,\dots,n-2\}}.
\]
We claim that for every $j=0,1,\dots, n-2,$
\[ \Card\Bigl(\varphi_{a, b}^{-1}\bigl(\set{\alpha_C}{C \in \mathcal{E}}\bigr) \cap \bigl(s^*_j, s^*_{j+1}\bigr) \Bigr) = 1. \]
To prove the claim note that, by (i--iii),
\begin{equation}\label{TheSetE}
 \mathcal{E} = \set{\varphi_{a, b}\bigl(\bigl(s^*_j,s^*_{j+1}\bigr)\bigr)}{j\in\{0,1,\dots, n-2\}}.
\end{equation}
Hence, for every $j=0,1,\dots, n-2,$
$
\alpha\sb{\varphi_{a, b}\bigl(\bigl(s^*_j,s^*_{j+1}\bigr)\bigr)} \in \varphi_{a, b}\bigl(\bigl(s^*_j,s^*_{j+1}\bigr)\bigr),
$
implies
\[
\emptyset \ne \varphi_{a, b}^{-1}\left(\alpha\sb{\varphi_{a, b}\bigl(\bigl(s^*_j,s^*_{j+1}\bigr)\bigr)}\right)
   \cap (s^*_j,s^*_{j+1}) \subset \varphi_{a, b}^{-1}\bigl(\set{\alpha_C}{C \in \mathcal{E}}\bigr) \cap \bigl(s^*_j, s^*_{j+1}\bigr) .
\]
Assume that
$
 s_i^1, s_i^2 \in \varphi_{a, b}^{-1}\bigl(\set{\alpha_C}{C \in \mathcal{E}}\bigr) \cap \bigl(s^*_j, s^*_{j+1}\bigr).
$
By \eqref{TheSetE} and the definition of the points $\alpha_C,$
\[
\varphi_{a, b}\bigl(s_i^1\bigr), \varphi_{a, b}\bigl(s_i^2\bigr) \in
 \varphi_{a, b}\bigl((s^*_j,s^*_{j+1})\bigr) \cap \set{\alpha_C}{C \in \mathcal{E}} =
 \Bigl\{\alpha\sb{\varphi_{a, b}\bigl(\bigl(s^*_j,s^*_{j+1}\bigr)\bigr)}\Bigr\}.
\]
Consequently, $s_i^1 = s_i^2$ by (ii). This proves the claim.

Now we set $m = m(X, a, b) := 2n-1$ and by the above claim we define the partition
\[
  s_0 = s^*_0 = 0 < s_1 < s_ 2 \cdots < s_{m-2} < s_{m-1} = s_ {2(n-1)} = s^*_{n-1} < s_m = s^*_n = 1
\]
of the interval $[0,1]$ by:
\[\begin{split}
& s_{2j} := s^*_j,\text{ and} \\
& s_{2j+1} \text{ is the unique point of } \varphi_{a, b}^{-1}\bigl(\set{\alpha_C}{C \in \mathcal{E}}\bigr) \cap \bigl(s^*_j, s^*_{j+1}\bigr)
\end{split}\]
for every $j=0,1,2,\dots, n-2.$
With these definitions and (i), the set
$\{s_0,s_1,\dots,s_m\}$ is the union of two disjoint sets:
\begin{equation}\label{construction-partition-setversion}
\begin{split}
\{s_0,s_2,\dots,s_{m-1},s_m\} &= \set{s^*_j}{j\in\{0,1,\dots,n\}} = \varphi_{a, b}^{-1}\bigl(\widetilde{V}(X)\bigr)\text{ and} \\
\{s_1,s_3,\dots,s_{m-2}\}     &= \varphi_{a, b}^{-1}\bigl(\set{\alpha_C}{C \in \mathcal{E}}\bigr).
\end{split}
\end{equation}

By definition $m = m(X, a, b, M)$ is odd,
$\map{\varphi_{a, b}}{[0,1]}[X]$ is continuous and surjective
and, by (i),
$\varphi_{a, b}(0) = a$ and $\varphi_{a, b}(1) = b.$
Moreover, since the map $\varphi_{a, b}$ is onto,
\begin{multline*}
 n+1 = \Card\left(\set{s^*_j}{j\in\{0,1,\dots,n\}}\right) =
   \Card\left(\varphi_{a, b}^{-1}\bigl(\widetilde{V}(X)\bigr)\right) \ge\\
   \Card\bigl(\widetilde{V}(X)\bigr) \ge 4,
\end{multline*}
and hence, $m = 2n - 1 \ge 5.$

On the other hand, by \eqref{construction-partition-setversion},
\begin{align*}
W &= \varphi_{a, b}\bigl(\set{s_i}{i\in\{0,1,\dots,m\}}\bigr) \\
  &= \varphi_{a, b}\left(\varphi_{a, b}^{-1}\bigl(\widetilde{V}(X)\bigr)\right) \cup \varphi_{a, b}\left(\varphi_{a, b}^{-1}\bigl(\set{\alpha_C}{C \in \mathcal{E}}\bigr)\right)\\
  &= \widetilde{V}(X) \cup \set{\alpha_C}{C \in \mathcal{E}} \supset V(X),
\end{align*}
and
\[
 \varphi_{a, b}^{-1}(W) =
 \varphi_{a, b}^{-1}\bigl(\widetilde{V}(X)\bigr) \cup \varphi_{a, b}^{-1}\bigl(\set{\alpha_C}{C \in \mathcal{E}}\bigr)
 = \set{s_i}{i\in\{0,1,\dots,m\}}.
\]
Thus, (a) holds.

Statement~(b) follows from (ii), Statement~(a) and the fact
that every interval $[s_i,s_{i+1}]$ is contained in an interval
$\bigl[s^*_j,s^*_{j+1}\bigr].$

To end the proof of the lemma it remains to prove (c).
Assume that $\varphi_{a, b}(s_i) = \varphi_{a, b}(s_j)$
(or, equivalently, that there exists $\alpha \in W$ such that
$s_i, s_j \in \varphi_{a, b}^{-1}(\alpha) \subset \varphi_{a, b}^{-1}(W)$).
Since
\[
\varphi_{a, b}^{-1}(W) =
  \varphi_{a, b}^{-1}\bigl(\widetilde{V}(X)\bigr) \cup
  \varphi_{a, b}^{-1}\bigl(\set{\alpha_C}{C \in \mathcal{E}}\bigr)
\text{ and }
\widetilde{V}(X) \cap \set{\alpha_C}{C \in \mathcal{E}} = \emptyset,
\]
by \eqref{construction-partition-setversion}, it follows that either
\begin{align*}
s_i, s_j &\in \varphi_{a, b}^{-1}\bigl(\widetilde{V}(X)\bigr) = \{s_0,s_2,\dots,s_{m-1},s_m\}\text{ or}\\
s_i, s_j &\in \varphi_{a, b}^{-1}\bigl(\set{\alpha_C}{C \in \mathcal{E}}\bigr) = \{s_1,s_3,\dots,s_{m-2}\}.
\end{align*}
On the other hand, by (iii),
$s_m = s^*_n = \varphi_{a, b}^{-1}(b) \notin \{s_i,s_j\}.$
Consequently, either
$i,j \in \{0,2,4,\dots,{m-1}\}$ or
$i,j \in \{1,3,5,\dots,{m-2}\},$
and (c) holds.
\end{proof}

The next lemma will be useful in dealing with piecewise expansiveness
and in making possible to use Theorem~\ref{theoremTransitivityexpansivemap}
to obtain transitive graph maps.

\begin{lemma}\label{convertingtoexpansiveMM}
Let $X$ be a topological graph and let $\map{f}{X}$ be a
Markov map with respect to a Markov invariant set $Q$
such that $f(I)$ is a (non-degenerate) interval for every
$I \in \SBI.$
Then there exists a $Q$-expansive (Markov) map $\map{g}{X}$
such that $g\evalat{Q} = f\evalat{Q}$ and $g(I) = f(I)$ for every
$I \in \SBI.$ In particular, the Markov graphs of $f$ and $g$
with respect to $Q$ coincide.
\end{lemma}

\begin{proof}
The requirement that $g\evalat{Q} = f\evalat{Q}$ implies that
it is enough to define $g\evalat{I}$ for every $I \in \SBI$ so that
$g\evalat{I}$ is expansive,
$g\evalat{\partial I} = f\evalat{\partial I}$ and
$g\evalat{I}(I) = f(I).$

Let $I \in \SBI.$ The monotonicity of $f$ on $I$ implies that
$f(I)$ is an interval that it is the union of $n \ge 1$
$Q$-basic intervals.
Thus there exists a partition $t_0, t_1,\dots, t_n$ of $I$ with
\[
\chull{t_i,t_{i+1}}[I] \cap \{t_0,t_1,\dots,t_n\} = \{t_i,t_{i+1}\}
\andq[for] i=0,1,\dots,n-1
\]
(in particular, $\{t_0, t_n\} = \partial I$)
and such that
\begin{multline*}
f\bigl(\chull{t_0,t_1}[I]\bigr) = \chull{f(t_0), f(t_1)}[f(I)],
  f\bigl(\chull{t_1,t_2}[I]\bigr)= \chull{f(t_1), f(t_2)}[f(I)], \dots,\\
  f\bigl(\chull{t_{n-1},t_n}[I]\bigr)= \chull{f(t_{n-1}), f(t_n)}[f(I)]
\end{multline*}
are pairwise different basic intervals.
Clearly, for every $i\in\{0,1,\dots,n-1\},$
\begin{align*}
& \set{\frac{d_I(x,t_i)}{d_I(t_i,t_{i+1})}}{x \in \chull{t_i,t_{i+1}}[I]} = \left[0, \frac{d_I(t_{i+1},t_i)}{d_I(t_i,t_{i+1})}\right] = [0, 1], \text{ and}\\
& \set{d_{f\bigl(\chull{t_i,t_{i+1}}[I]\bigr)}(z, f(t_i))}{z \in f\bigl(\chull{t_i,t_{i+1}}[I]\bigr)} = [0, 1]
\end{align*}
because $f\bigl(\chull{t_i,t_{i+1}}[I]\bigr) \in \SBI,$
and hence
\[
  d_{f\bigl(\chull{t_i,t_{i+1}}[I]\bigr)}(f(t_{i+1}), f(t_i)) =
  \spnorm[{f\bigl(\chull{t_i,t_{i+1}}[I]\bigr)}]{f\bigl(\chull{t_i,t_{i+1}}[I]\bigr)} =
  1.
\]

Then, for $i\in\{0,1,\dots,n-1\}$ and $x \in \chull{t_i,t_{i+1}}[I]$
we define $g\evalat{\chull{t_i,t_{i+1}}[I]}(x)$ to be the unique point from
$\chull{f(t_i), f(t_{i+1})}[f(I)]$ that verifies
\[
 d_{f\bigl(\chull{t_i,t_{i+1}}[I]\bigr)}(g\evalat{\chull{t_i,t_{i+1}}[I]}(x), f(t_i)) = \frac{1}{d_I(t_i,t_{i+1})} d_I(x,t_i).
\]
Observe that this formula defines
\[
   g\evalat{\chull{t_i,t_{i+1}}[I]}(t_i) = f(t_i) \andq
   g\evalat{\chull{t_i,t_{i+1}}[I]}(t_{i+1}) = f(t_{i+1}).
\]
Hence $g\evalat{I}$ is well defined and continuous and
$g\evalat{\{t_0, t_1,\dots, t_n\}} = f\evalat{\{t_0, t_1,\dots, t_n\}}.$
In particular, $g\evalat{\partial I} = f\evalat{\partial I}.$
Moreover, for every $i\in\{0,1,\dots,n-1\},$
$g\evalat{\chull{t_i,t_{i+1}}[I]}$ is one-to-one (and hence monotone),
and
\[
  g\evalat{I}\bigl(\chull{t_i,t_{i+1}}[I]\bigr) =
  \chull{f(t_i), f(t_{i+1})}[f(I)] =
  f\bigl(\chull{t_i,t_{i+1}}[I]\bigr).
\]
Consequently $g\evalat{I}(I) = f(I).$

To end the proof of the lemma only it remains to show that
$g\evalat{I}$ is expansive.
By using appropriately the triangle inequality and the monotonicity of
$g\evalat{\chull{t_i,t_{i+1}}[I]}$
is not difficult to see that
\[
 d_{f\bigl(\chull{t_i,t_{i+1}}[I]\bigr)}\bigl(g\evalat{I}(x), g\evalat{I}(y)\bigr) = \frac{1}{d_I(t_i,t_{i+1})} d_I(x,y)
\]
for every $x,y \in \chull{t_i,t_{i+1}}[I]$ with $i\in\{0,1,\dots,n-1\}.$
Thus, in the special case when $n=1$
(that is, when $I = \chull{t_0,t_1}[I] \in \SBI$ and $f(I) \in \SBI$),
we have
\[
 d_{f(I)}\bigl(g\evalat{I}(x), g\evalat{I}(y)\bigr) = d_I(x,y)
\]
for every $x,y \in I,$ because $d_I(t_i,t_{i+1}) = \spnorm[I]{I} = 1.$
Hence $g\evalat{I}$ is expansive on $I.$

When $n > 1$ then $g\evalat{I}$ also is expansive on $I$ by setting
\[
 \lambda_I = \min \set{\tfrac{1}{d_I(t_i,t_{i+1})}}{i\in\{0,1,\dots,n-1\}} > 1.
\]
\end{proof}
\subsection{Example with persistent fixed low periods}

This subsection is devoted to construct and prove

\begin{CustomNumberedExample}[\ref{examplefirstexamplebcnintroduction}]
For every positive integer $n \in \set{4k+1, 4k-1}{k\in \N}$
there exists $f_n,$ a totally transitive continuous circle map
of degree one having a lifting $F_n \in \dol$ such that
$\Rot(F_n) = \left[\tfrac{1}{2}, \tfrac{n+2}{2n}\right],$
$\lim_{n\to\infty} h(f_n) = 0,$
\[
  \Per (f_n) = \{ 2\} \cup \set{p\ \text{odd}}{2k+1 \le p\le n-2} \cup  \succs{n}
\]
and $\bc(f_n)$ exists and verifies $2k+1 \le \bc(f_n) \le n$
(and, hence, $\lim_{n\to\infty} \bc(f_n) = \infty$).\smallskip

Furthermore, given any graph $G$ with a circuit, the sequence of maps
$\{f_n\}_{n\ge 7, n\text{ odd}}$ can be extended to a sequence
of continuous totally transitive self maps of $G$,
$\{g_n\}_{n\ge 7, n\text{ odd}},$
such that
$\Per(g_n) = \Per(f_n)$ and $\lim_{n\to\infty} h(g_n) = 0.$
\end{CustomNumberedExample}

Example~\ref{examplefirstexamplebcnintroduction} will be split
into Theorem~\ref{theoremfirstexamplebcncircle}
which shows the existence of the circle maps $f_n$ by constructing them
along the lines of Subsection~\ref{Examples:PhilandIntro-HowTo},
and Theorem~\ref{theoremfirstexamplebcngraph} that
extends these maps to a generic graph that is not a tree.
The proof of Theorem~\ref{theoremfirstexamplebcncircle}, in turn, will
use a proposition that computes the Markov graph modulo 1 of the
liftings $F_n.$

The auxiliary Figure~\ref{figuretheoremfirstexamplebcncircle}
illustrates the construction of the orbits $P_n$, $Q_n$ and
the map $F_n$ from the next theorem in a particular case.

\begin{theorem}\label{theoremfirstexamplebcncircle}
Let $n \in \set{4k+1, 4k-1}{k\in \N}$ and let
\begin{align*}
Q_n &= \{\dots, x_{-1}, x_0, x_1, x_2,\dots\} \subset \R,\andq \\
P_n &= \{\dots y_{-1}, y_0, y_1, y_2, \dots, y_{2n-1}, y_{2n}, y_{2n+1}, \dots\} \subset \R
\end{align*}
be infinite sets such that the points of $P_n$ and $Q_n$ are intertwined so that
\[
   0 = x_0 < y_0 < y_1 < \cdots y_{n-3} < x_1 < y_{n-2} < \cdots < y_{2n-1} < x_2 = 1,
\]
and
$x_{i + 2\ell} = x_i + \ell$ and $y_{i + 2n\ell} = y_i + \ell$ for every $i,\ell \in \Z.$

We define a lifting $F_n \in \dol$ such that, for every $i \in \Z,$
$F_n(x_i) = x_{i+1}$ and $F_n(y_i) = y_{i+n+2},$
and $F_n$ is affine between consecutive points of $P_n \cup Q_n.$
Then, $Q_n$ (respectively $P_n$)
is a twist lifted periodic orbit of $F_n$ of period 2 (respectively $2n$)
with rotation number $\tfrac{1}{2}$ (respectively $\tfrac{n+2}{2n}$).
Moreover, the map $F_n$ has
$\Rot(F_n) = \left[\tfrac{1}{2}, \tfrac{n+2}{2n}\right]$
as rotation interval.

Let $\map{f_n}{\SI}$ be the continuous map which has $F_n$ as a lifting.
Then, $f_n$ is totally transitive, $\lim_{n\to\infty} h(f_n) = 0,$
\[
  \Per (f_n) = \Per (F_n) = \{ 2\} \cup \set{q\ \text{odd}}{2k+1 \le q \le n-2} \cup  \succs{n}
\]
and $\bc(f_n)$ exists and verifies $2k+1 \le \bc(f_n) \le n.$
\end{theorem}
\begin{figure}[ht]
\def\equisp{0.416666667}
\begin{tikzpicture}
\draw (0,9) -- (0,0) -- (5,0) -- (5,9);
\draw (4,9) -- (0,5) -- (5,5) -- (0,0);
\def\xtest{x}
\foreach \x / \l / \p / \f / \c [ remember=\p as \lp (initially -1),
                                  remember=\f as \lf (initially 9) ] in {
                                  x/0/0/4/black, y/0/1/8/darkgray, y/1/2/8/darkgray, y/2/3/8/darkgray,
                                  x/1/4/8/black, y/3/5/8/darkgray, y/4/6/8/darkgray, y/5/7/8/darkgray,
                                  y/6/8/9/darkgray, y/7/9/9/darkgray, y/8/10/9/darkgray, y/9/11/9/darkgray,
                                  x/2/12/4/black} { \pgfmathsetmacro{\px}{\p*\equisp}; \pgfmathsetmacro{\pxf}{(\p+\f)*\equisp};
                                                    \pgfmathsetmacro{\lx}{\lp*\equisp}; \pgfmathsetmacro{\lxf}{(\lp+\lf)*\equisp};
                                    \node[below, \c] at (\px,0) {\tiny $\x_{\l}$};
                                    \draw[dashed, \c] (\px, -0.1) -- (\px, \pxf);
                                    \ifnum\p>3
                                         \ifnum\p<12 \node[left, \c] at (-0.6,\px) {\tiny $\x_{\l}$}; \draw[dashed, \c] (-0.7,\px) -- (\px, \px); \fi
                                         \draw[dashed, \c] (\px, \pxf) -- (-0.1,\pxf);
                                         \ifx\x\xtest \pgfmathtruncatemacro{\lsub}{\l - 1};\pgfmathtruncatemacro{\lsubtop}{\lsub + 2};
                                         \else        \pgfmathtruncatemacro{\lsub}{\l - 3}; \pgfmathtruncatemacro{\lsubtop}{\lsub + 10}; \fi
                                         \node[left, \c] at (0,\pxf) {\tiny $\x_{\lsubtop} = \x_{\lsub} + 1$};
                                    \fi
                                    \draw[thick] (\lx, \lxf) -- (\px, \pxf); \filldraw (\px, \pxf) circle (0.04);
                                    \ifnum\p>1\ifnum\p<12\begin{scope}[shift={(-0.04,0.04)}] \draw[color=blue] (\lx, \lxf) -- (\px, \pxf); \end{scope}\fi\fi
                                    \ifnum\p>0\ifnum\p<8\begin{scope}[shift={(0.04,-0.04)}] \draw[color=red] (\lx, \lxf) -- (\px, \pxf); \end{scope}\fi\fi
}
\begin{scope}[shift={(-0.04,0.04)}]
   \draw[color=blue] (-\equisp,{8*\equisp}) -- ({4*\equisp/5}, {8*\equisp}) -- (\equisp, {9*\equisp});
   \draw[color=blue, shift={(5,5)}] (-\equisp,{8*\equisp}) -- ({4*\equisp/5}, {8*\equisp}) -- (\equisp, {9*\equisp});
\end{scope} \filldraw[red] (-\equisp,{8*\equisp}) circle (0.04);
\draw[thick] (5, {5 + 4*\equisp}) -- +(\equisp,{5*\equisp}); \filldraw[red] ({5+\equisp},{5 + 9*\equisp}) circle (0.04);
\begin{scope}[shift={(0.04,-0.04)}]
    \draw[color=red] (-\equisp,{4*\equisp}) -- +(\equisp, 0);
    \draw[color=red] ({7*\equisp}, {5 + 3*\equisp}) -- ({7.5*\equisp}, {5 + 4*\equisp}) -- (5, {5 + 4*\equisp}) -- +(\equisp,{5*\equisp});
\end{scope}
\node[above left] at (2.55, 5.725) {$F_{5}$};
\end{tikzpicture}
\caption{A possible choice of the points of
$(P_5 \cup Q_5) \cap [0,1]$ from
Theorem~\ref{theoremfirstexamplebcncircle} and
Proposition~\ref{propositionmarkovgraphfirstexamplebcn},
and the graph of the corresponding map $F_5$.
The lower map \textcolor{red}{$(F_5)_l$} is drawn in \textcolor{red}{red} and
the upper map \textcolor{blue}{$(F_5)_u$} in \textcolor{blue}{blue}.}\label{figuretheoremfirstexamplebcncircle}
\end{figure}
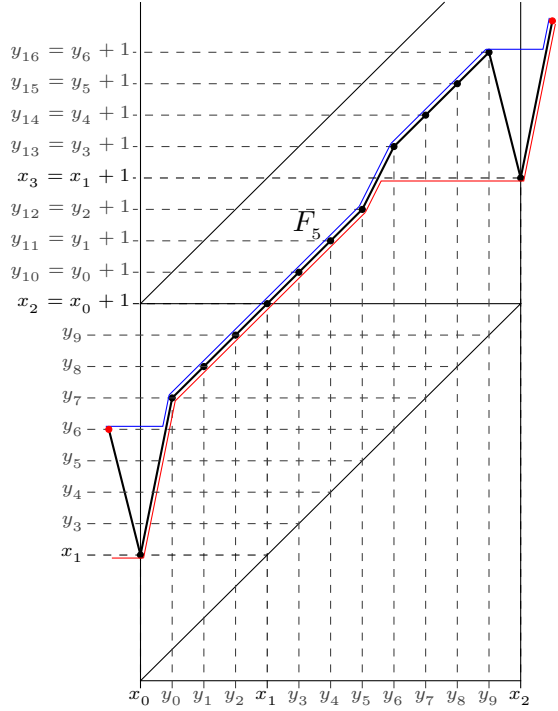

\begin{theorem}\label{theoremfirstexamplebcngraph}
Let $G$ be a graph with a circuit. Then, the sequence of maps
$\{f_n\}_{n\ge 7, \ n\text{ odd}}$ from
Theorem~\ref{theoremfirstexamplebcncircle}
can be extended to a sequence
of continuous totally transitive self maps of $G$,
$\{g_n\}_{n\ge 7, \ n\text{ odd}},$
such that
$\Per(g_n) = \Per(f_n)$ and $\lim_{n\to\infty} h(g_n) = 0.$
\end{theorem}

Before proving Theorem~\ref{theoremfirstexamplebcncircle}, in the next
proposition, we study the Markov graph  modulo 1 of the liftings $F_n$
(see the auxiliary Figure~\ref{figuretheoremfirstexamplebcncircle}).
Given $m \in \Z$ and $q \in \N,$ to simplify the notation,
we will denote ${m \pmod{q}}$ by $\modulo{m}{q}.$

\long\def\subfigurefngraphModOne{%
\node[place] (j2)        at (9.5,0)    {$\bigBIclass{J_2}$};
\node[place, double, pattern color=black!30, pattern=north east lines] (j0) at (5, 0) {$\bigBIclass{J_0}$};
\node[place] (i4d1)      at (11.5,1.5) {$\bigBIclass{I_{4d-1}}$};
\node[rotate=90] (dots)  at (11.5,3)   {$\cdots$};
\node[place] (i3d1)      at (11.5,4.5) {$\bigBIclass{I_{3d+1}}$};
\node[place] (i3d)       at (11.5,6)   {$\bigBIclass{I_{3d}}$};
\node (dots2)            at (9.6,6)    {$\cdots $};
\node[place] (i2d1)      at (7.7,6)    {$\bigBIclass{I_{2d+1}}$};
\node[place] (i2d2)      at (5.8,6)    {$\bigBIclass{I_{2d}}$};
\node (dots3)            at (3.9,6)    {$\cdots $};
\node[place] (id1)       at (1.9,6)    {$\bigBIclass{I_{d+1}}$};
\node[place] (id)        at (0,6)      {$\bigBIclass{I_{d}}$};
\node[place] (j1)        at (9.5,1.5)  {$\bigBIclass{J_1}$};
\node[place, double, pattern color=black!30, pattern=north east lines] (j3) at (7.5,1.5) {$\bigBIclass{J_3}$};

\path[post]
     (i4d1.west)+(-1pt,0)  edge (j1.east)
     (j1.west)+(-1pt,0)    edge (j3)
     (j2.west)+(-1pt,-4pt) edge ([yshift=-4pt]j0.east)
     (j0.east)+(2pt,4pt)   edge (j2.158)
     (id.east)+(1pt,0)     edge (id1)
     (id1.east)+(1pt,0)    edge (dots3)
     (dots3.east)+(1pt,0)  edge (i2d2)
     (i2d2.east)+(1pt,0)   edge (i2d1)
     (i2d1.east)+(1pt,0)   edge (dots2)
     (dots2.east)+(1pt,0)  edge (i3d)
     (i3d.south)+(0,-1pt)  edge (i3d1)
     (i3d1.south)+(0,-1pt) edge (dots)
     (i4d1.west)+(-1pt,0)  edge (j2.north)
     (dots.west)+(0,-1pt)  edge (i4d1)
     (j3.south)+(0,-2pt)   edge (j2.north)
     (j3.north)+(0,1pt)    edge (id.300)
     (j3.north)+(0,1pt)    edge (i2d2)
     (j3.north)+(0,1pt)    edge (i3d.south)
     (j0.north)+(1pt,1pt)  edge (i2d2)
     (j0.north)+(1pt,1pt)  edge (i3d);
}

\begin{proposition}[$\calB(P_n \cup Q_n)$ and the $F_n$-Markov graph modulo 1]\label{propositionmarkovgraphfirstexamplebcn}
In the assumptions of Theorem~\ref{theoremfirstexamplebcncircle} we set
\[
  J_0 := [ x_0, y_0],\ J_1 := [ y_{n-3}, x_1],\ J_2 := [ x_1, y_{n-2}], \andq J_3 := [ y_{2n-1}, x_2],
\]
and
\[
  I_i := \Bigl[y_{\modulo{n+1+i(n+2)}{2n}}, y_{\modulo{n+2+i(n+2)}{2n}}\Bigr] \andq[for] i \in \{0,1,\dots,2n-3\}.
\]
Then the following statements hold:
\begin{enumerate}[(a)]
\item We have,
\begin{multline*}
 \bigSBI{P_n \cup Q_n} = \set{I_i + \ell}{i \in \{0,1,\dots,2n-3\},\ \ell \in \Z} \cup \\
                      \set{J_i + \ell}{i \in \{0,1,2,3\},\ \ell \in \Z}.
\end{multline*}
and, in particular,
\[
   \left(\bigcup\nolimits_{i=0}^{2n-3} I_i  \right)\cup \left(\bigcup\nolimits_{j=0}^{3} J_j \right) = [0,1].
\]
\item When $n = 4k-1$ for some $k \in \N,$
the Markov graph modulo 1 of $F_n$ is:
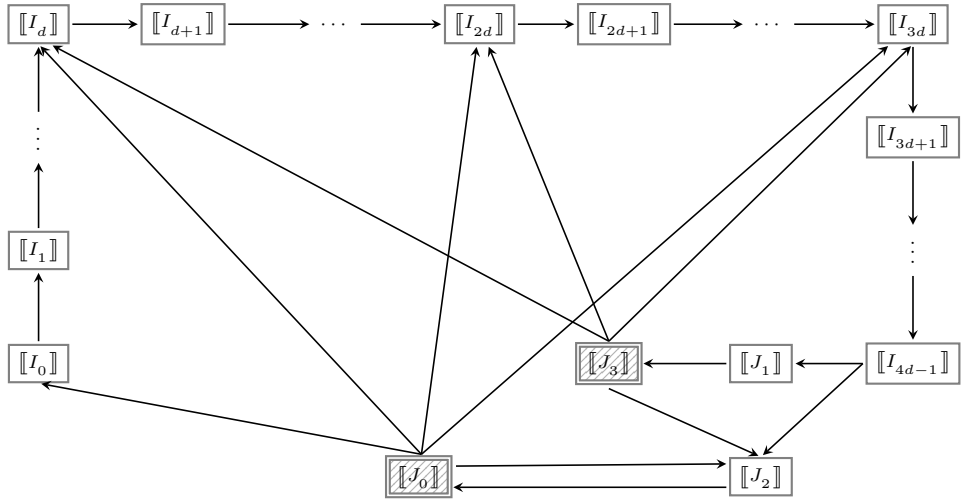
\begin{center}\tiny
        \tikzstyle{place}=[rectangle,draw=black!50,fill=black!0,thick]
        \tikzstyle{post}=[->,shorten >=1pt,>=stealth,semithick]
\hspace*{-1.8em}\begin{tikzpicture}
\subfigurefngraphModOne
\node[rotate=90] (dots4) at (0,4.5)    {$\cdots$};
\node[place] (i1)        at (0,3)      {$\bigBIclass{I_1}$};
\node[place] (i0)        at (0,1.5)    {$\bigBIclass{I_0}$};
\path[post]
     (j0.north)+(1pt,1pt)  edge (i0.south)
     (j0.north)+(1pt,1pt)  edge (id.south)
     (i0.north)+(0,1pt)    edge (i1)
     (i1.north)+(0,1pt)    edge (dots4)
     (dots4.east)+(0,1pt)  edge (id);
\end{tikzpicture}\medskip
\end{center}
where $d := 2k-1.$
Otherwise, when $n = 4k+1$ for some $k \in \N,$
the Markov graph modulo 1 of $F_n$ is:\vspace*{-1ex}
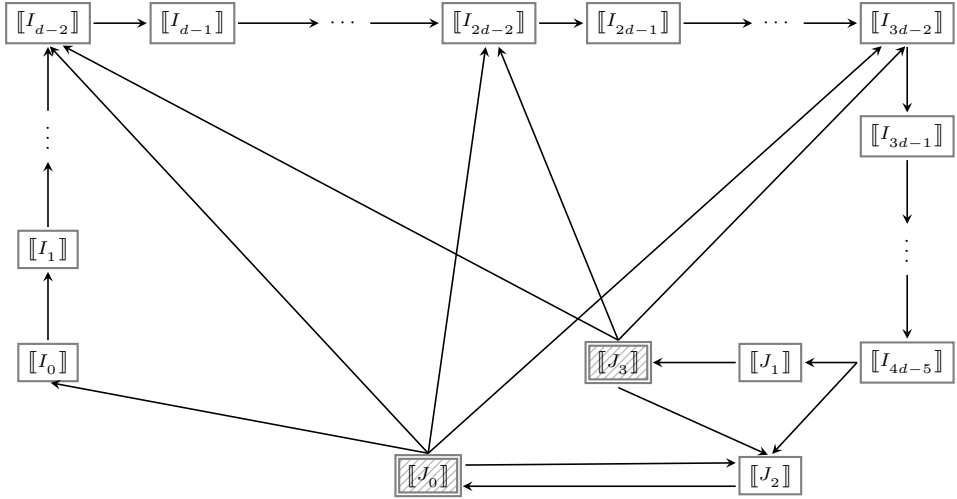
\begin{center}\tiny
        \tikzstyle{place}=[rectangle,draw=black!50,fill=black!0,thick]
        \tikzstyle{post}=[->,shorten >=1pt,>=stealth,semithick]
\smallskip
\hspace*{-1.6em}\begin{tikzpicture}
\node[place](j2)        at (9.5,0)    {$\bigBIclass{J_2}$};
\node[place, double, pattern color=black!30, pattern=north east lines] (j0) at (5, 0) {$\bigBIclass{J_0}$};
\node[place](i4d5)      at (11.3,1.5) {$\bigBIclass{I_{4d-5}}$};
\node[rotate=90](dots)  at (11.3,3)   {$\cdots$};
\node[place](i3d1)      at (11.3,4.5) {$\bigBIclass{I_{3d-1}}$};
\node[place](i3d2)      at (11.3,6)   {$\bigBIclass{I_{3d-2}}$};
\node(dots2)            at (9.6,6)    {$\cdots $};
\node[place](i2d1)      at (7.7,6)    {$\bigBIclass{I_{2d-1}}$};
\node[place](i2d2)      at (5.8,6)    {$\bigBIclass{I_{2d-2}}$};
\node(dots3)            at (3.9,6)    {$\cdots $};
\node[place](id1)       at (1.9,6)    {$\bigBIclass{I_{d-1}}$};
\node[place](id2)       at (0,6)      {$\bigBIclass{I_{d-2}}$};
\node[rotate=90](dots4) at (0,4.5)    {$\cdots$};
\node[place](i1)        at (0,3)      {$\bigBIclass{I_1}$};
\node[place](i0)        at (0,1.5)    {$\bigBIclass{I_0}$};
\node[place](j1)        at (9.5,1.5)  {$\bigBIclass{J_1}$};
\node[place, double, pattern color=black!30, pattern=north east lines] (j3) at (7.5,1.5) {$\bigBIclass{J_3}$};

\path[post]
     (i4d5.west)+(-1pt,0)  edge (j1.east)
     (j1.west)+(-1pt,0)    edge (j3)
     (j2.west)+(-1pt,-4pt) edge ([yshift=-4pt]j0.east)
     (j0.east)+(2pt,4pt)   edge (j2.158)
     (j0.north)+(0,1pt)    edge (i0.south)
     (i0.north)+(0,1pt)    edge (i1)
     (i1.north)+(0,1pt)    edge (dots4)
     (dots4.east)+(0,1pt)  edge (id2)
     (id2.east)+(1pt,0)    edge (id1)
     (id1.east)+(1pt,0)    edge (dots3)
     (dots3.east)+(1pt,0)  edge (i2d2)
     (i2d2.east)+(1pt,0)   edge (i2d1)
     (i2d1.east)+(1pt,0)   edge (dots2)
     (dots2.east)+(1pt,0)  edge (i3d2)
     (i3d2.south)+(0,-1pt) edge (i3d1)
     (i3d1.south)+(0,-1pt) edge (dots)
     (i4d5.west)+(-1pt,0)  edge (j2.north)
     (dots.west)+(0,-1pt)  edge (i4d5)
     (j3.south)+(0,-2pt)   edge (j2.north)
     (j3.north)+(0,1pt)    edge (id2.300)
     (j3.north)+(0,1pt)    edge (i2d2)
     (j3.north)+(0,1pt)    edge (i3d2.south)
     (j0.north)+(0,1pt)    edge (id2.south)
     (j0.north)+(0,1pt)    edge (i2d2)
     (j0.north)+(0,1pt)    edge (i3d2);
\end{tikzpicture}\medskip
\end{center}
where  $d := 2k+1.$

\item $h(f_n) = \log \rho_n,$ where $\rho_n > 1$ is the largest root
of the polynomial
\[
    T_n(x)  = x^{2n}(x^2 - 1) -2x^{\tfrac{3n+1}{2}} -2x^{n+1}-2x^{\tfrac{n+3}{2}} -x^2 -1.
\]
\end{enumerate}
\end{proposition}

\begin{proof}
It is obvious from the assumptions that
$J_0,J_1,J_2,J_3 \in \bigSBI{P_n \cup Q_n}$
and $\Card\bigl((P_n \cup Q_n) \cap [0,1]\bigr) = 2n + 3.$
Hence, there are $2n-2$ $P_n \cup Q_n$-basic intervals contained
in the interval $[0,1]$ different from $J_0,J_1,J_2$ and $J_3.$
On the other hand, $n+2$ and $2n$ are coprime and, hence,
there exist $2n-2$ pairwise different intervals $I_i.$
Thus, to prove (a) it is enough to show
that all the intervals $I_i$ are $P_n \cup Q_n$-basic.
This amounts showing that
\begin{multline*}
 \set{I_i = [y_{\modulo{n+1+i(n+2)}{2n}}, y_{\modulo{n+2+i(n+2)}{2n}}]}{i \in \{0,1,\dots,2n-3\}} \subset \\
    \bigl\{ [y_0,y_1], [y_1,y_2], \dots,
            [y_{n-4},y_{n-3}], [y_{n-2}, y_{n-1}], [y_{n-1},y_{n}],\dots,
            [y_{2n-2},y_{2n-1}]\bigr\}.
\end{multline*}
More concretely, we have to see that
\[
\bigl\{\modulo{n+1+i(n+2)}{2n} \,\colon i \in \{0,1,\dots,2n-3\}\bigr\} \cap \{n-3, 2n-1\} = \emptyset
\]
because $\modulo{n+2+i(n+2)}{2n} = \modulo{n+1+i(n+2)}{2n} + 1$
provided that $\modulo{n+1+i(n+2)}{2n} \ne 2n-1.$

Assume by way of contradiction that there exist
$i \in \{0,1,\dots,2n-3\},$ $\ell \in \Z$
and  $a \in \{n-3, 2n-1\}$ such that
\[
  n+1+i(n+2)  = a + \ell 2n
  \quad\Longleftrightarrow\quad
   i(n+2)  = (2\ell - 1)n + (a-1).
\]
Then, since $n \in \set{4k+1, 4k-1}{k\in \N}$ is odd,
it follows that $i$ has the same parity as $a.$
So, there exists $t \in \{0,1,\dots, n-2\}$ such that
$i = 2t$ when $a = n-3,$ and $i = 2t+1$ when $a = 2n-1.$
In any case,
\[
   i(n+2)  = (2\ell - 1)n + (a-1)
   \quad\Longleftrightarrow\quad
   4t  = 2n(\ell - t - 1) + (2n-4).
\]
Since $0 \le 4t \le 4n-8$ it follows that
$4t  = 2n-4$ (that is, $\ell - t - 1$ must be 0)
which implies $\tfrac{n}{2}-1 = t \in \Z;$
a contradiction. So, the intervals $I_i$ are
$P_n \cup Q_n$-basic and, hence, (a) holds.

Now we will compute the Markov graph modulo 1 of $F_n$ to prove (b)
Recall that, by definition,
\[
  y_{i + 2n\ell} = y_i + \ell = y_{\modulo{i + 2n\ell}{2n}} + \ell
  \andq
  F_n(y_i) = y_{i+n+2}
\]
for every $i,\ell \in \Z.$
For convenience we set
$\widetilde{I}_i = \bigl[y_{n+1+i(n+2)}, y_{n+2+i(n+2)}\bigr]$
for every $i \in \{0,1,\dots,2n-2\}.$
Hence, by the part of the lemma already proven,
for $i \in \{0,1,\dots,2n-3\}$ we have
$\BIclass{I_i} =  \left\llbracket\widetilde{I}_i\right\rrbracket,$
$\widetilde{I}_i \in \bigSBI{P_n \cup Q_n},$ and
\begin{equation}\label{images}
 F_n\left(\widetilde{I}_i\right) =
 \bigl[F_n(y_{n+1+i(n+2)}), F_n(y_{n+2+i(n+2)})\bigr] =
 \widetilde{I}_{i+1}
\end{equation}
because $F_n\evalat{P_n}$ is increasing and
$F_n$ is affine on $P_n \cup Q_n$-basic intervals.
Moreover,
\begin{multline*}
 F_n\left(\widetilde{I}_{2n-3}\right) = \widetilde{I}_{2n-2} =
   \bigl[y_{n-3 + 2n(n+1)}, y_{n-2 + 2n(n+1)}\bigr] = \\
   (J_1 + (n+1)) \cup (J_2 + (n+1)).
\end{multline*}
Consequently, the Markov graph modulo 1 of $F_n$ has the following
subgraph:
\begin{equation}\label{fundamentaloop} \arraycolsep=2pt
 \BIclass{I_0} \longrightarrow \BIclass{I_1} \longrightarrow \cdots
 \longrightarrow \BIclass{I_{2n-3}}
 \begin{array}{ll} \nearrow & \raisebox{6pt}{\BIclass{J_1}}\\ \searrow & \raisebox{-6pt}{\BIclass{J_2}} \end{array}.
\end{equation}

To completely determine the Markov graph modulo 1 of $F_n$
we still need to compute the images of the intervals
$J_0,\ J_1,\ J_2$ and $J_3.$
We have $x_{i + 2\ell} = x_i + \ell$ and $F_n(x_i) = x_{i+1}$
for every $i,\ell \in \Z$ and, hence,
\begin{enumerate}[(i)]
\item $F_n(J_0) = [x_1, y_{n+2}] = J_2 \cup [y_{n-2}, y_{n+2}]$;
\item $F_n(J_1) = [y_{2n-1}, x_2] = J_3$;
\item $F_n(J_2) = [x_2, y_{2n}] = J_0 + 1;$ and
\item $F_n(J_3) = [x_3,y_{3n+1}] = [x_1,y_{n+1}] + 1 = (J_2 + 1) \cup ([y_{n-2}, y_{n+1}] + 1).$
\end{enumerate}

We will end the proof in the case $n = 4k-1$ and
$d = 2k-1 = \tfrac{n-1}{2}.$
The proof in the other case follows analogously.

In this case we have $2n-3 = 4d-1.$
Moreover, for $\ell \ge 1$ we have
\[
  \ell d(n+2) = dn\ell + 2d\ell = 2kn\ell - n\ell + (4k-1)\ell - \ell = k\ell2n - \ell.
\]
Hence,
\begin{align*}
  & \modulo{n+1+\ell d(n+2)}{2n} = \modulo{n+1-\ell + k\ell2n}{2n} = n+1-\ell,\text{ and} \\
  & \modulo{n+2+\ell d(n+2)}{2n} = \modulo{n+2-\ell + k\ell2n}{2n} = n+2-\ell.
\end{align*}
So, $I_{\ell d} = [y_{n+1-\ell}, y_{n+2-\ell}]$ for $\ell = 0,1,2,3$
(observe that the interval $I_{4d}$ is not defined because $4d > 2n-3$
and, on the other hand, $[y_{n+1-\ell}, y_{n+2-\ell}]$ with $\ell = 4$
is not $P_n \cup Q_n$-basic).
Consequently, from \eqref{fundamentaloop} and (i--iv) above we get
that the Markov graph modulo 1 of $F_n$ is the union of the following
two subgraphs:
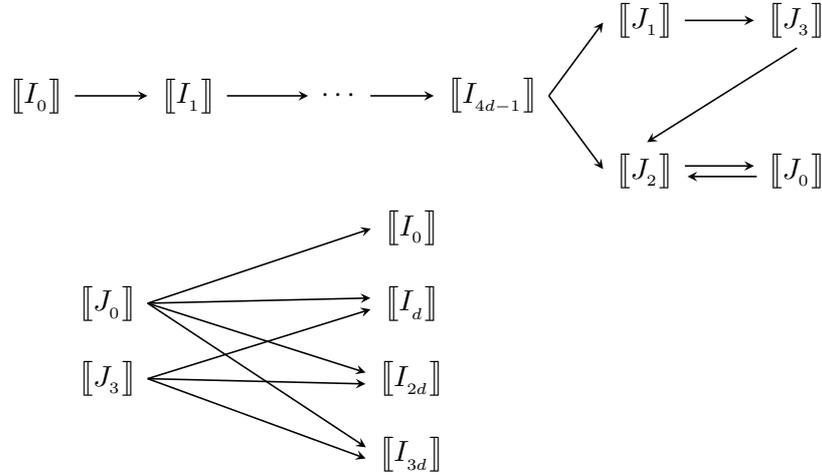
\begin{center}
\begin{tikzpicture}
  \node(I0) at (0,0)    {$\bigBIclass{I_0}$};
  \node(I1) at (2,0)    {$\bigBIclass{I_1}$};
  \node(dots) at (4,0)    {$\cdots$};
  \node(I4dm1) at (6,0)    {$\bigBIclass{I_{4d-1}}$};
  \node(J1) at (8,1)    {$\bigBIclass{J_1}$};
  \node(J3) at (10,1)    {$\bigBIclass{J_3}$};
  \node(J2) at (8,-1)    {$\bigBIclass{J_2}$};
  \node(J0) at (10,-1)    {$\bigBIclass{J_0}$};

  \path[->,shorten >=1pt,>=stealth,semithick]
     (I0.east) edge (I1.west)
     (I1.east) edge (dots.west)
     (dots.east) edge (I4dm1.west)
     (I4dm1.east) edge (J1.west)
     (I4dm1.east) edge (J2.west)
     (J1.east) edge (J3.west)
     (J2.east)+(0,2pt) edge ([yshift=2pt]J0.west)
     (J0.west)+(0,-2pt) edge ([yshift=-2pt]J2.east)
     (J3.south) edge (J2.north);
\end{tikzpicture}
\par\par\hspace*{-11em}
\begin{tikzpicture}
  \node(J0) at (0,1)    {$\bigBIclass{J_0}$};
  \node(J3) at (0,0)    {$\bigBIclass{J_3}$};
  \node(I0) at (4,2)    {$\bigBIclass{I_0}$};
  \node(Id) at (4,1)    {$\bigBIclass{I_{d}}$};
  \node(I2d) at (4,0)    {$\bigBIclass{I_{2d}}$};
  \node(I3d) at (4,-1)    {$\bigBIclass{I_{3d}}$};

  \path[->,shorten >=1pt,>=stealth,semithick]
      (J0.east) edge (I0.west)
      (J0.east) edge ([yshift=2pt]Id.west)
      (J0.east) edge ([yshift=2pt]I2d.west)
      (J0.east) edge ([yshift=2pt]I3d.west)
      (J3.east) edge ([yshift=-2pt]Id.west)
      (J3.east) edge ([yshift=-2pt]I2d.west)
      (J3.east) edge ([yshift=-2pt]I3d.west);
\end{tikzpicture}
\end{center}
This ends the proof of (b).

To prove (c) we will use
Propositions~\ref{Markov partitionprojectiontoSI}
and~\ref{propositionmonotoneTopEnt},
and Theorem~\ref{theoremrome}.

First notice that \eqref{images} and (i--iv) above imply that
$P_n \cup Q_n$ is a short Markov partition with respect to $F_n.$
Then, Propositions~\ref{Markov partitionprojectiontoSI}
and~\ref{propositionmonotoneTopEnt},
imply that $f_n$ is a Markov map and
\[
    h(f_n) = \log\max\{\sigma(M_n),1\}
\]
where $M_n$ is the Markov matrix of $f_n$ with respect to
$\bigemap{P_n \cup Q_n}.$
Moreover, we can identify the set $\bigSBI{\bigemap{P_n \cup Q_n}}$
with the set of all equivalence classes of $P_n \cup Q_n$-basic intervals
(i.e. the set of all vertices of the Markov graph modulo 1 of $F_n$).
Then, the matrix $M_n$ coincides with the
transition matrix of the Markov graph modulo 1 of $F_n$ which,
by definition and Proposition~\ref{Markov partitionprojectiontoSI},
is a
$\Card\bigl(\bigSBI{\bigemap{P_n \cup Q_n}}\bigr) \times \Card\bigl(\bigSBI{\bigemap{P_n \cup Q_n}}\bigr)$
matrix
$M_n = (m_{\BIclass{I},\BIclass{J}})_{\BIclass{I},\BIclass{J} \in \SBI[\emap{P_n \cup Q_n}]}$
such that
\[
  m_{\BIclass{I},\BIclass{J}}=\begin{cases}
      1 & \text{if $\BIclass{I}$ $f$-covers $\BIclass{J}$}\\
      0 & \text{otherwise}
\end{cases}.
\]

To compute $\sigma(M_n)$ we will use Theorem~\ref{theoremrome} with
\[
\textsf{Rom}_n = \{\textsf{r}_1 = \BIclass{J_0}, \textsf{r}_2 = \BIclass{J_3}\}
\]
as a rome
(being their elements marked in the statement with a box with double
border and sloping lines background pattern).
Moreover, recall that $M_{\textsf{Rom}_n}(x) = (a_{ij}(x))$
is the matrix such that $a_{ij}(x) = \sum_p x^{-\ell(p)},$
where the sum is taken over all simple paths
starting at $r_i$ and ending at $r_j$
(since $M_n$ is  a matrix of zeroes and ones the width of every path is 1).
Then, the matrix $M_{\textsf{Rom}_n}(x)$ is:
\[
 \begin{cases}
   \begin{pmatrix*}[r]
        x^{-2} + x^{-4d-2} +\alpha(x) & x^{-4d-2} +\alpha(x)\\
        x^{-2}  +\alpha(x)            & \alpha(x)
   \end{pmatrix*} & \text{\parbox{9em}{when $n=4k-1$ and $d=\tfrac{n-1}{2}$ for some $k \in \N,$ and}}\\[5ex]
   \begin{pmatrix*}[r]
         x^{-2} + x^{-4d+2} +  \alpha(x) & x^{-4d+2} +\alpha(x)\\
         x^{-2}   +\alpha(x)             & \alpha(x)
  \end{pmatrix*} & \text{\parbox{9em}{when $n=4k+1$ and $d=\tfrac{n+1}{2}$ for some $k \in \N,$}}
\end{cases}
\]
where
\[
 \alpha(x) = \begin{cases}
   x^{-3d -2} +x^{-2d-2}+x^{-d-2}& \text{\parbox{14em}{when $n=4k-1$ and $d=\tfrac{n-1}{2}$ for some $k \in \N,$ and}}\\[2ex]
   x^{-3d} +x^{-2d} +x^{-d} & \text{\parbox{14em}{when $n=4k+1$ and $d=\tfrac{n+1}{2}$ for some $k \in \N.$}}
\end{cases}
\]

By Theorem~\ref{theoremrome}, the characteristic polynomial $T_n(x)$
of $M_n$ is
\begin{multline*}
    T_n(x) = (-1)^{2n} x^{2n+2} \det(M_{\textsf{Rom}_n}(x) - \mathbf{I}_2) = \\
             x^{2n}(x^2 - 1) -2x^{\tfrac{3n+1}{2}} -2x^{n+1}-2x^{\tfrac{n+3}{2}} -x^2 -1,
\end{multline*}
where $\mathbf{I}_2$ is the unit matrix $2\times 2.$

By direct inspection of the Markov graph modulo 1 of $F_n$ (see (b)),
given any two vertices $\BIclass{K}$ and $\BIclass{L}$
in the graph, there exists a path from $\BIclass{K}$ to $\BIclass{L}.$
This means that the transition matrix $M_n$ of the
Markov graph modulo 1 of $F_n$ is non-negative and irreducible.
Then, by the Perron-Frobenius Theorem, $\sigma(M_n) > 1$
is the largest eigenvalue of $M_n$.
Hence, $\sigma(M_n)$ is the largest root (larger than one) of $T_n.$
\end{proof}

\begin{proof}[Proof of Theorem~\ref{theoremfirstexamplebcncircle}]
Since $y_{i + 2n\ell} = y_i + \ell$ and $F_n(y_i) = y_{i+n+2}$
for every $i,\ell \in \Z,$ it follows that
\[
 F^{2n}_n(y_i) = y_{i + 2n(n+2)} = y_i + n+2
\]
for every $i \in \Z.$
Moreover, let $j \in \{1,2,\dots, 2n-1\}$ and assume that
\[ F^j_n(y_i) - y_i = y_{i + j(n+2)} - y_i = \ell \in \Z. \]
Then, $y_{i + j(n+2)} = y_i + \ell = y_{i + 2n\ell}$ and, thus,
$j(n+2) = 2n\ell;$
a contradiction because $n+2$ and $2n$ are relatively prime.
Hence, $P_n$ is a lifted periodic orbit of $F_n$ of period $2n$
and rotation number $\tfrac{n+2}{2n}.$
Moreover, $F_n\evalat{P_n}$ is increasing and, thus,
$P_n$ is a twist lifted periodic orbit of $F_n$ of period $2n$
and rotation number $\tfrac{n+2}{2n}.$
In a similar manner, $Q_n$ is a twist lifted periodic orbit of $F_n$
of period $2$ and rotation number $\tfrac{1}{2}.$

Now we will show tat
$\Rot(F_n) = \left[\tfrac{1}{2}, \tfrac{n+2}{2n}\right]$
by using Theorem~\ref{theoremrotationintwathermap}.
To this end we need to compute the rotation number of the
lower map $(F_n)_l$ and of the upper map $(F_n)_u$
(see Figures~\ref{upper-lowermaps} and~\ref{figuretheoremfirstexamplebcncircle}).
Observe that
\[
F_n(y_{2n-5}) = y_{3n-3} = y_{n-3} + 1 < x_1 + 1 < y_{n-2} + 1 = y_{3n-2} = F_n(y_{2n-4}).
\]
Hence, there exists a unique $u^n_l \in (y_{2n-5}, y_{2n-4})$
such that $F_n\bigl(u^n_l\bigr) = x_1 + 1.$ So,\pagebreak[1]
\begin{multline*}
 (F_n)_l(x) = \inf\set{F_n(y)}{y \ge x} = \\
 \begin{cases}
    \inf\bigl[F_n(x), +\infty\bigr) = F_n(x)& \text{for $x \in \bigl[0, u^n_l\bigr]$,}\\
    \inf\bigl[x_1+1, +\infty\bigr) = x_1+1 & \text{for $x \in \bigl[u^n_l,1\bigr]$,}\\
    (F_n)_l(x - \floor{x}) + \floor{x} & \text{if $x \notin [0,1].$}
 \end{cases}
\end{multline*}
On the other hand,
$
F_n(x_0) = x_1 < y_{n+1} < y_{n+2} = F_n(y_0),
$
which implies that there exists a unique $u^n_u \in (x_0,y_0)$
such that $F_n\bigl(u^n_u\bigr) = y_{n+1} = F_n(y_{2n-1})-1.$ So,
\begin{multline*}
 (F_n)_u(x) = \sup\set{F_n(y)}{y \le x} = \\
 \begin{cases}
    \sup\bigl[-\infty, y_{n+1}\bigr) = y_{n+1} & \text{for $x \in \bigl[0, u^n_u\bigr]$,}\\
    \sup\bigl[-\infty, F_n(x)\bigr) = F_n(x)   & \text{for $x \in \bigl[u^n_u,y_{2n-1}\bigr]$,}\\
    \sup\bigl[-\infty, F_n(y_{2n-1})\bigr) = y_{n+1}+1 & \text{for $x \in \bigl[y_{2n-1}, 1\bigr]$,}\\
    (F_n)_u(x - \floor{x}) + \floor{x} & \text{if $x \notin [0,1].$}
 \end{cases}
\end{multline*}

To compute $\rho\bigl((F_n)_l\bigr)$ observe that
\[ P_n \cap [0,1] \subset \bigl[u^n_u,y_{2n-1}\bigr] \andq[and, hence,]
   (F_n)_u\evalat{P_n} = F_n\evalat{P_n}.
\]
So,
$\rho\bigl((F_n)_u\bigr) = \rho_{F_n}(P_n) = \tfrac{n+2}{2n}.$
In a similar way,
$Q_n \cap [0,1] \subset \bigl[0, u^n_l\bigr],$
$(F_n)_l\evalat{Q_n} = F_n\evalat{Q_n}$ and
$\rho\bigl((F_n)_l\bigr) = \rho_{F_n}(Q_n) = \tfrac{1}{2}.$
Hence, $\Rot(F_n) = \left[\tfrac{1}{2}, \tfrac{n+2}{2n}\right]$
by Theorem~\ref{theoremrotationintwathermap}.

Next we will compute $\Per(f_n).$
By Theorem~\ref{theoremMisiurewicz} we have
\[
\Per(F_n) = Q_{F_n}\mspace{-8.0mu}\left(\tfrac{1}{2}\right) \cup
            M\left(\tfrac{1}{2}, \tfrac{n+2}{2n}\right) \cup
            Q_{F_n}\mspace{-8.0mu}\left(\tfrac{n+2}{2n}\right).
\]
We will compute separately the sets
$Q_{F_n}\mspace{-8.0mu}\left(\tfrac{1}{2}\right)$ and
$
  M\left(\tfrac{1}{2}, \tfrac{n+2}{2n}\right) \cup
  Q_{F_n}\mspace{-8.0mu}\left(\tfrac{n+2}{2n}\right),
$
starting with $Q_{F_n}\mspace{-8.0mu}\left(\tfrac{1}{2}\right).$

To compute $Q_{F_n}\mspace{-8.0mu}\left(\tfrac{1}{2}\right)$
we will use Proposition~\ref{propPerLoops} with $X$ and $f$
replaced by $\SI$ and $f_n,$ respectively.
Hence, we will use the Markov graph of $f_n.$
However, from the proof of
Proposition~\ref{propositionmarkovgraphfirstexamplebcn}
we already know that $P_n \cup Q_n$ is a short Markov partition with
respect to $F_n$ and,
by Proposition~\ref{Markov partitionprojectiontoSI},
$f_n$ is a Markov map with respect to the Markov
partition $\bigemap{P_n \cup Q_n}.$
Moreover,
the Markov graph of $f_n$ with respect to
$\bigemap{P_n \cup Q_n}$
and the Markov graph modulo 1 of $F_n$ with respect to
$P_n \cup Q_n$ coincide, provided that we identify
$\BIclass{I}$ with $\bigemap{\BIclass{I}}$ for every
$I \in \bigSBI{P_n \cup\; Q_n}$.
Thus, to perform our arguments we will use
Proposition~\ref{propositionmarkovgraphfirstexamplebcn} and
the Markov graph modulo 1 of $F_n$ with respect to
$P_n \cup Q_n$
instead of the Markov graph of $f_n$ with respect to
$\bigemap{P_n \cup Q_n}.$

By Proposition~\ref{propositionmarkovgraphfirstexamplebcn},
in the Markov graph modulo 1 of $F_n$
there is no simple loop of length $4$
and the only repetitive loop of length 4 is the $2$-repetition
of
$
\BIclass{J_0} \longrightarrow \BIclass{J_2} \longrightarrow \BIclass{J_0}.
$
Since
\begin{align*}
& F_n(x_1) = x_2 = x_0 + 1 < y_{0} + 1 = y_{2n} = F_n(y_{n-2})\text{ and}\\
& F_n(x_0) = x_1 < y_{n+2} = F_n(y_0),
\end{align*}
$
\BIclass{J_0} \longrightarrow \BIclass{J_2} \longrightarrow \BIclass{J_0}
$
is a positive loop. Then,
$4 \notin \Per(F_n) \supset Q_{F_n}\mspace{-8.0mu}\left(\tfrac{1}{2}\right)$
by Proposition~\ref{propPerLoops}(b).
Therefore, with the notation from the definition of $Q_F(c)$
in Subsection~\ref{RotTheor}, we have $s=2$ and $s_{1/2} = 1.$
Consequently,
$Q_{F_n}\mspace{-8.0mu}\left(\tfrac{1}{2}\right) = \{2\}.$

Now we compute
$
  M\left(\tfrac{1}{2}, \tfrac{n+2}{2n}\right) \cup
  Q_{F_n}\mspace{-8.0mu}\left(\tfrac{n+2}{2n}\right).
$
Since $\tfrac{n+2}{2n} - \tfrac{1}{2} = \tfrac{1}{n},$
it follows that for every all $q \in \N,$ $q > n,$
there exists $p \in \Z$ such that
$\tfrac{1}{2} < \tfrac{p}{q} < \tfrac{n+2}{2n}.$
On the other hand, since $n$ is odd,
$(n+1)/2 \in \Z$ and
$\tfrac{1}{2} < \tfrac{(n+1)/2}{n} < \tfrac{n+2}{2n}.$
Summarizing,
\[
   M\left(\tfrac{1}{2}, \tfrac{n+2}{2n}\right) \supset
   \succs{n} \supset
   \set{2n\ell}{\ell \in \N} \supset
   Q_{F_n}\mspace{-8.0mu}\left(\tfrac{n+2}{2n}\right).
\]
Thus,
\begin{multline*}
   M\left(\tfrac{1}{2}, \tfrac{n+2}{2n}\right) \cup Q_{F_n}\mspace{-8.0mu}\left(\tfrac{n+2}{2n}\right) =
   M\left(\tfrac{1}{2}, \tfrac{n+2}{2n}\right) =\\
   \succs{n} \cup \set{q \in M\left(\tfrac{1}{2}, \tfrac{n+2}{2n}\right)}{q < n}.
\end{multline*}

Now we need to compute
$\set{q \in M\left(\tfrac{1}{2}, \tfrac{n+2}{2n}\right)}{q < n}.$
To this end, assume that
$\tfrac{1}{2} < \tfrac{p}{q} < \tfrac{n+2}{2n}$ with
$p \in \Z$ and $q \in \N,$ $q \le n-1.$
We claim that $q$ is odd.
Otherwise, $q = 2\ell \le n-1$ with $\ell \in \N.$
Then, the expression
$\tfrac{1}{2} < \tfrac{p}{2\ell} < \tfrac{n+2}{2n}$
is equivalent to
\[
  \ell < p < \ell + \frac{2\ell}{n} \le
  \ell + \frac{n-1}{n} <
  \ell + 1;
\]
a contradiction.
This proves the claim.

Now assume that $q = 2\ell + 1 \le n-2$ with $\ell \in \N$
(recall that $n$ is odd).
We have,
\[
  \frac{1}{2} < \frac{p}{2\ell + 1} < \frac{n+2}{2n}
\]
which is equivalent to
\begin{multline*}
  \ell + \frac{1}{2}  = \frac{2\ell + 1}{2} < p <
  (\ell + 2) \frac{n(2\ell + 1) + 2(2\ell + 1)}{2n(\ell + 2)} \le \\
  (\ell + 2) \frac{n(2\ell + 1) + 2(n-2)}{2n\ell + 4n} =
  (\ell + 2) \frac{2n\ell + 3n - 4}{2n\ell + 4n} < \ell + 2.
\end{multline*}
Consequently, $p = \ell + 1$ and, hence,
\[
  \frac{1}{2} < \frac{\ell + 1}{2\ell + 1} < \frac{n+2}{2n}.
\]
The second inequality is equivalent to
\[
 2n\ell + 2n = 2n(\ell + 1) < (n+2)(2\ell + 1) = 2n\ell + n + 2(2\ell + 1).
\]
Thus, since $n = 4k \pm 1$ with $k \in \N,$ this is equivalent to
\[
 2\ell + 1 \ge \frac{n+1}{2} = \frac{4k + r}{2} = 2k + \frac{r}{2} \andq[with] r \in \{0,2\}.
\]
Hence, since $q = 2\ell + 1$ is odd,
\[
 2\ell + 1 \ge 2k + 1.
\]

Summarizing, we have seen that
\[
  \set{q \in M\left(\tfrac{1}{2}, \tfrac{n+2}{2n}\right)}{q < n} =
  \set{2\ell+1}{\ell\in\N \text{ and } 2k+1 \le 2\ell+1 \le n-2}
\]
and, consequently,
\begin{align*}
\Per(f_n) & = \Per(F_n) = Q_{F_n}\mspace{-8.0mu}\left(\tfrac{1}{2}\right) \cup
              M\left(\tfrac{1}{2}, \tfrac{n+2}{2n}\right) \cup
              Q_{F_n}\mspace{-8.0mu}\left(\tfrac{n+2}{2n}\right)\\
          & = \{2\} \cup \succs{n} \cup \set{q \in M\left(\tfrac{1}{2}, \tfrac{n+2}{2n}\right)}{q < n}\\
          & = \{2\} \cup \succs{n} \cup \set{q\ \text{odd}}{2k+1 \le q \le n-2}.
\end{align*}
Moreover, $\sbc(f_n) = n$ and $2k+1 \in \sbcset(f_n)$.
So, $\bc(f_n)$ exists and verifies $2k+1 \le \bc(f_n) \le n.$

Next we will show that $f_n$ is totally transitive by using
Theorem~\ref{theoremTransitivityexpansivemap}.
We already know that
$f_n$ is a Markov map with respect to the Markov
partition $\bigemap{P_n \cup Q_n},$
and the transition matrix of the Markov graph of $f_n$
with respect to $\bigemap{P_n \cup Q_n}$
coincides with the transition matrix of the Markov graph modulo 1
of $F_n$ with respect to $P_n \cup Q_n.$
From the proof of
Proposition~\ref{propositionmarkovgraphfirstexamplebcn}(c),
it follows that this transition matrix is non-negative and irreducible.
By direct inspection
(see Proposition~\ref{propositionmarkovgraphfirstexamplebcn}(b))
it follows that the vertex $\BIclass{J_0}$ of
the Markov graph modulo 1 of $F_n$ is the beginning of 4 arrows.
That is, the transition matrix of the Markov graph of $f_n$
with respect to $\bigemap{P_n \cup Q_n}$ has a row with
4 non-zero elements and, thus, it cannot be a permutation matrix.

To use Theorem~\ref{theoremTransitivityexpansivemap} we also need to know
that $f_n$ is a $\bigemap{P_n \cup Q_n}$-expansive Markov map in the sense of
Definition~\ref{expansiveMM}.
Since we already know that $f_n$ is a Markov map we have to show that
$f_n$ is expansive on every $\bigemap{P_n \cup Q_n}$-basic interval.
To do this we need two ingredients, a distance $d_I$ on every
$\bigemap{P_n \cup Q_n}$-basic interval $I$ and an appropriate way of
writing the fact that the maps $F_n$ are affine on every basic interval.

Let $[a,b] \in \bigSBI{P_n \cup Q_n}.$
The fact that $F_n\evalat{[a,b]}$ is affine can be written as
\begin{equation}\label{affineontheline}
 \frac{\abs{F_n(x)-F_n(y)}}{\abs{F_n(a)-F_n(b)}} = \frac{\abs{y-x}}{b-a}
\end{equation}
for every $x,y \in [a,b].$

A distance $d_I$ on every basic interval
$I \in \bigSBI{\bigemap{P_n \cup Q_n}}$
can be defined as follows.
Write $I = \bigemap{[x_I, y_I]}$ where $[x_I, y_I]$ is a $P_n \cup Q_n$-basic interval.
Then, for every $x,y \in [x_I, y_I]$, we define
\[
 d_I(\emap{x},\emap{y}) := \frac{\abs{x-y}}{\abs{x_I-y_I}}.
\]

Observe that $I = \bigemap{[x_I, y_I]}$ implies
\[
f_n(I) = \bigemap{F_n\bigl([x_I, y_I]\bigr)}
       = \bigemap{\chull{F_n(x_I), F_n(y_I)}[\R]}.
\]

Consider first the case $f_n(I) \in \bigSBI{\bigemap{P_n \cup Q_n}}$
(which is equivalent to
$\chull{F_n(x_I), F_n(y_I)}[\R]\in \SBI[P_n \cup Q_n]$).
Hence, since $F_n$ is affine on $[x_I, y_I],$
\begin{multline*}
 d_{f_n(I)}\bigl(f_n(\emap{x}),f_n(\emap{y})\bigr) =
 d_{f_n(I)}\bigl(\emap{F_n(x)},\emap{F_n(y)}\bigr) =\\
 \frac{\abs{F_n(x)-F_n(y)}}{\abs{F_n(x_I)-F_n(y_I)}} = \frac{\abs{x-y}}{y_I-x_I} = d_I(\emap{x},\emap{y}).
\end{multline*}

Now assume that $f_n(I) = \bigemap{\chull{F_n(x_I), F_n(y_I)}[\R]}$
contains more than one $\bigemap{P_n \cup Q_n}$-basic interval and
let $x,y \in [x_I, y_I]$ be such that
\[
  \chull{f_n(\emap{x}), f_n(\emap{y})}[f_n(I)] \subset J = \bigemap{[x_J, y_J]}
  \andq[with]
  [x_J, y_J] \in \SBI[P_n \cup Q_n].
\]
In this case we set
\begin{multline*}
 \lambda_I := \min\left\{ \tfrac{\abs{F_n(x_I)-F_n(y_I)}}{y_J - x_J}\,\colon [x_J, y_J] \in \SBI[P_n \cup Q_n] \text{ and}\right.\\
                  \left.[x_J, y_J] \subset \chull{F_n(x_I), F_n(y_I)}[\R]\vphantom{\tfrac{\abs{F_n(x_I)-F_n(y_I)}}{y_J - x_J}}\right\},
\end{multline*}
and we have
\begin{multline*}
 d_{J}\bigl(f_n(\emap{x}),f_n(\emap{y})\bigr) =
 d_{J}\bigl(\emap{F_n(x)},\emap{F_n(y)}\bigr) =\\
 \frac{\abs{F_n(x)-F_n(y)}}{y_J - x_J} =
 \frac{\abs{F_n(x_I)-F_n(y_I)}}{y_J - x_J} \frac{\abs{F_n(x)-F_n(y)}}{\abs{F_n(x_I)-F_n(y_I)}} \ge\\
 \lambda_I \frac{\abs{y-x}}{y_I -x_I} = \lambda_I  d_I(\emap{x},\emap{y}).
\end{multline*}
So, we have proved that $f_n$ is $\bigemap{P_n \cup Q_n}$-expansive,
and thus, $f_n$ is transitive by Theorem~\ref{theoremTransitivityexpansivemap}.
Moreover, since $\Per(f_n) \supset \succs{n},$
$\Per(f_n)$ is cofinite and $f_n$ is totally transitive
by Theorem~\ref{theoremtotallytransitivefromadrr}.

To prove that $\lim_{n\to\infty} h(f_n) = 0$ we will use
Proposition~\ref{propositionmarkovgraphfirstexamplebcn}(c)
which states that
$h(f_n) = \log \rho_n,$ where $\rho_n > 1$ is the largest root
of the polynomial
\[
    T_n(x)  = x^{2n}(x^2 - 1) -2x^{\tfrac{3n+1}{2}} -2x^{n+1}-2x^{\tfrac{n+3}{2}} -x^2 -1.
\]
Set
\begin{align*}
q_n(x) & := x^{2n}(x^2 - 1),\\
t_n(x) & := 2x^{\tfrac{3n+1}{2}} +2x^{n+1}+2x^{\tfrac{n+3}{2}} +x^2 +1,\text{ and}\\
\xi_n(x) & := \frac{q(x)}{t(x)} =
  \frac{x^{\tfrac{n-1}{2}}(x^2 -1)}{2 + 2x^{\tfrac{-n+1}{2}} + 2 x^{-n +1} + x^{\tfrac{-3n+3}{2}} + x^{-\tfrac{3n+1}{2}}}
  \text{ for $x > 0$.}
\end{align*}

With this notation, the expression
$0 = T_n(\rho_n) = q_n(\rho_n) - t_n(\rho_n)$
is equivalent to
$\xi_n(\rho_n) = 1.$
Observe that, for $x \ge 1,$
\begin{equation}\label{ineqxi}
\xi_n(x) = \frac{x^{\tfrac{n-1}{2}}(x^2 -1)}{2 + 2x^{\tfrac{-n+1}{2}} + 2 x^{-n +1} + x^{\tfrac{-3n+3}{2}} + x^{-\tfrac{3n+1}{2}}}
\ge \frac{x^{\tfrac{n-1}{2}}(x^2 -1)}{8},
\end{equation}
\[
\text{and}\hspace*{4em}\frac{x^{\tfrac{n-1}{2}}(x^2 -1)}{8} = 1 \Longleftrightarrow x^{\tfrac{n-1}{2}} = \frac{8}{x^2 -1}.
\]
Now we remark that\par\par\bigskip\par\par\noindent
\begin{minipage}{0.63\textwidth}
\begin{enumerate}[(i)]
\item The map $x \mapsto \frac{8}{x^2 -1}$ is strictly decreasing on $(1,+\infty)$,
      $\lim_{x\to1^+} \frac{8}{x^2 -1} = +\infty$ and
      $\lim_{x\to\infty} \frac{8}{x^2 -1} = 0.$
\item For every $n$ odd and every $x \ge 1$,
      the map $x \mapsto x^{\tfrac{n-1}{2}}$ is strictly increasing and
      $x^{\tfrac{n-1}{2}}\evalat{x=1} = 1.$
\item For every $n,m \in \N,$  $n,m$ odd, $n < m$ and $x > 1,$
      $x^{\tfrac{n-1}{2}} < x^{\tfrac{m-1}{2}}.$
\end{enumerate}
Then, for each $n$ odd, there exists a unique real
\end{minipage}\hfill\begin{minipage}{0.35\textwidth}\flushright\begin{tikzpicture}
\draw[->, >=stealth] (1,-0.1) -- (1,4); \draw[->, >=stealth] (1,0) -- (5,0);
\draw (1,1) -- (0.9,1); \node[left] at (1,1) {\tiny$1$};
\draw (1,0) -- (1,-0.1); \node[below] at (1,0) {\tiny$1$};
\node[below] at (4.5,0) {$x$};

\draw[densely dashed] (1.262,-0.3) -- +(0,1.675); \node[below] at (1.55,-0.2) {\tiny$\gamma_{m}$};
\draw[densely dashed] (1.304,-0.1) -- +(0,1.27); \node[below right] at (1.15,0) {\tiny$\gamma_{n}$};

\draw[domain=2:5, smooth, variable=\x, blue] plot ({\x},{0.84/(\x*\x-1)});
\draw[domain=1.1:2, smooth, variable=\x, blue] plot ({\x},{0.84/(\x*\x-1)});

\draw[domain=1:1.89, smooth, variable=\x, red] plot ({\x},{\x^4/4 + 0.75});
\draw[domain=1:3.57, smooth, variable=\x, red] plot ({\x},{\x^2/4 + 0.75});

\node[red, right] at (1.1,3.5) {$\cdots$}; \node[red, right] at (3.5,3.5) {$\cdots$};

\node[blue, above] at (3.5,0) {\tiny$\tfrac{8}{x^2 -1}$};
\node[red, right] at (2.1,2) {\tiny$x^{\tfrac{n-1}{2}}$};
\node[red, right] at (1.7,3.5) {\tiny$x^{\tfrac{m-1}{2}}$};
\end{tikzpicture}\label{fig:arrels}
\end{minipage}
number $\gamma_n > 1$ such that
$\gamma_n^{\tfrac{n-1}{2}} = \tfrac{8}{\gamma_n^2 -1},$
$x^{\tfrac{n-1}{2}} > \tfrac{8}{x^2 -1}$ for every $x > \gamma_n,$
the sequence $\{\gamma_n\}_n$ is strictly decreasing and
$\lim_{n\to\infty} \gamma_n = 1.$
Hence, by \eqref{ineqxi}
\[
  \xi(x) \ge \frac{x^{\tfrac{n-1}{2}}(x^2 -1)}{8} > 1
\]
for every $x > \gamma_n.$
Consequently, $\rho_n \le \gamma_n$ for every $n$ odd and, thus,
\[
 \lim_{n\to\infty} \log\rho_n  \le \lim_{n\to\infty} \log\gamma_n = 0.
\]
\end{proof}

Next we prove Theorem~\ref{theoremfirstexamplebcngraph} by
``exporting'' the maps $F_n$ from
Theorem~\ref{theoremfirstexamplebcncircle} to any arbitrary graph.

\begin{proof}[Proof of Theorem~\ref{theoremfirstexamplebcngraph}]
If $G = \SI$ then there is nothing to prove since
Theorem~\ref{theoremfirstexamplebcncircle} already gives the desired
sequence of maps.
So, we assume that $G \ne \SI.$

Let $n \ge 7$ odd, and let $F_n,$ $P_n,$ $Q_n$ and $f_n$ be as in
Theorem~\ref{theoremfirstexamplebcncircle} and
Proposition~\ref{propositionmarkovgraphfirstexamplebcn}.
Recall also that
$I_i,$ $0 \le i \le 2n-3$ and $J_j,$ $0 \le j \le 3$
are $P_n \cup Q_n$-basic intervals which generate all the
equivalence classes of $P_n \cup Q_n$-basic intervals.

Next we fix the general notation to be used in this proof:
Let $C$ be a circuit of $G,$
let $I \subset C$ be an interval such that $I \cap V(G) = \emptyset$
and let $\map{\eta}{\SI}[C]$ be a homeomorphism such that
\[
  I \supset \left( \bigcup_{\substack{i=0\\ i \ne 2}}^{2n-3} \widetilde{I}_i \right)
      \cup \etaemap{P_n \cup Q_n}
      \andq
      C\mkern-2mu\setminus\mkern-3mu\Int(I) = \widetilde{I}_2,
\]
where
$\widetilde{I}_i := \etaemap{I_i}$ for $i \in \{0,1,\dots, 2n-3\}.$
Clearly,
$X := G \setminus \Int(I) \supset \widetilde{I}_2$
is a subgraph of $G$
(see Figure~\ref{figuregraphGandmapgnfirstexampleBCN}).
Observe that $\tilde{z}^2_0, \tilde{z}^2_1 \in I$ are endpoints
(and thus vertices) of $X$ but they cannot be vertices of $G$ because
$I \cap V(G) = \emptyset.$ Moreover,
$V(X) = V(G) \cup \bigl\{\tilde{z}^2_0,\tilde{z}^2_1\bigr\}.$

Recall that, for every $i \in \{0,1,\dots, 2n-3\},$
the endpoints of $I_i$ are
$y_{\modulo{n+1+i(n+2)}{2n}} < y_{\modulo{n+2+i(n+2)}{2n}}.$
Then, for $j=0,1$ we set
\[
  \tilde{z}^i_j := \bigeta{\bigemap{y_{\modulo{n+j+1+i(n+2)}{2n}}}}
\]
so that,
$\partial\widetilde{I}_i = \bigl\{\tilde{z}^i_0, \tilde{z}^i_1\bigr\}.$

Under the assumptions of Theorem~\ref{theoremfirstexamplebcncircle}
with $n \ge 7$ odd we have
\[ 0=x_0 < y_3 < y_4 < y_7 < y_8 < y_{n+5} < y_{n+6} \le y_{2n-1} < 1. \]
Hence,
\[
 I_1 = [y_3, y_4],\quad
 I_2 = [y_{n+5}, y_{n+6}] \andq
 I_3 = [y_7, y_8]
\]
are pairwise disjoint and, consequently, so are
$\widetilde{I}_1,$ $\widetilde{I}_2$ and $\widetilde{I}_3.$

To define the maps $\map{g_n}{G}$ for every $n$ odd, $n \ge 7,$
we will use Lemma~\ref{lemmaPhiPsi} for the subgraph $X$
with $a$ replaced by $\tilde{z}^2_0$ and
$b$ replaced by $\tilde{z}^2_1$
(see Figure~\ref{mapsfromarr}).
\begin{figure}[t]
\begin{tikzpicture}
\begin{scope} \clip (-2.3,-6.3) rectangle (9.3,5.7);
\draw[blue, very thick, rounded corners] (6.63,1.464) .. controls (6.635,0.8) and (6.5125,0.1) .. (6.485,0)  .. controls (6.5105,-0.1) and (6.714,-1) .. (6.5825,-1.99);
\draw[blue, rounded corners, postaction={decorate}, decoration={markings,
       mark=at position 2cm with {\draw[black, thick] (0,-3pt) -- (0,3pt);}, mark=at position 2cm with {\node[black, right] {$\tilde{z}^2_0$};},
       mark=at position 5cm with {\draw[black, thick] (0,-3pt) -- (0,3pt);}, mark=at position 5cm with {\node[black, below, xshift=2pt] {\vbox{$\tilde{z}^1_1$\\$\shortparallel$\\$s^1_m$}};},
       mark=at position 9cm with {\draw[black, thick] (0,-3pt) -- (0,3pt);}, mark=at position 9cm with {\node[black, below] {\vbox{$\tilde{z}^1_0$\\$\shortparallel$\\$s^1_0$}};},
       mark=between positions 7.8cm and 8.3cm step 0.5cm with {\draw[black] (0,-3pt) -- (0,3pt);},
       mark=at position 7.8cm with {\node[black, below, xshift=4pt] {$s^1_{l+1}$};},
       mark=at position 8.3cm with {\node[black, below, xshift=-1pt] {$s^1_l$};},
       mark=at position 7.97cm with {\node[black, above, yshift=-2pt] {$L_l$};},
       mark=at position 12cm with {\draw[black, thick] (0,-3pt) -- (0,3pt);}, mark=at position 12cm with {\node[black, left] {$\tilde{z}^0_1$};},
       mark=at position 14cm with {\draw[black, thick] (0,-3pt) -- (0,3pt);}, mark=at position 14cm with {\node[black, left] {$\tilde{z}^0_0$};},
       mark=at position 17cm with {\draw[black, thick] (0,-3pt) -- (0,3pt);}, mark=at position 17cm with {\node[black, below, xshift=4pt] {$\tilde{z}^3_1$};},
       mark=at position 19cm with {\draw[black, thick] (0,-3pt) -- (0,3pt);}, mark=at position 19cm with {\node[black, below, xshift=2pt] {$\tilde{z}^3_0$};},
       mark=at position 18cm with {\node[black, below, xshift=3pt]{$\widetilde{I}_3$};},
       mark=at position 24cm with {\draw[black, thick] (0,-3pt) -- (0,3pt);}, mark=at position 24cm with {\node[black, right] {$\tilde{z}^2_1$};}
    }] (6.5,0)  .. controls (7,-3) and (6,-3.5) ..
       (4,-3.8) .. controls (2.4,-4) and (1,-3.5) ..
       (0.4,-3) .. controls (-0.8,-2) and (-1.2, -1.5) ..
       (-1,0)   .. controls (-1,1.5) and (-0.8, 2) ..
       (0.4,3)  .. controls (1.2,3.5) and (2.4,4) ..
       (4,3.8)  .. controls (6,3.5) and (7,3) ..
       (6.5,0);
\draw[decorate, decoration={brace, amplitude=5pt, raise=0.5pt}] (0.64,-3.05) -- (4.42,-3.62) node[midway, above, xshift=2pt, yshift=2pt]  {$\widetilde{I}_1$};
\node[above, rotate=-12] at (2, -3.8) {$\cdots$};
\node[above, rotate=-22] at (0.9, -3.45) {$\cdots$};
\node[left] at (-0.3,0.18) {$\widetilde{I}_0$};
\path[->, >=stealth, thick, blue] (-0.5,0.18) edge [out=0, in=90] node[above, xshift=2pt] {$\ g_n$} (2.6,-2.6);
\path (6,0)         edge [bend right] (8.01,2)
      (8,1)         edge [bend left]  (9.3,4)
      (7,0.45)      edge [bend left]  (7.5, 0.3)
      (7.5,0.3)     edge [bend left]  (9,-3)
      (7.5,0.3)     edge [bend right] (9,-3)
      (7.5,0.3)     edge (9,-3)
      (8,-0.8)      edge (8.78,-1.25)
      (8.285,-1.43) edge (8.22, -2.3)
      (8.5,-1.9)    edge (9, -2.1)
      (8,-0.8)      edge (7.775,-1.6);
\node[right] at (8.25, -0.4) {$U_j$};
\draw [thick] ([xshift=-1.5pt,yshift=-1.2pt]8.1,-0.2)  -- ([xshift=1.5pt,yshift=1.2pt]8.1,-0.2);
\draw [thick] ([xshift=-1.5pt,yshift=-1pt]8.62,-0.9)  -- ([xshift=1.5pt,yshift=1pt]8.62,-0.9);

\begin{scope}[color=red]
\draw (2,-5.5) -- (4,-5.5);
\foreach \s in {-1,-0.8,-0.6,-0.4,-0.2,0.8,1}{ \draw[thick] ({3+\s},-5.6)   -- +(0,0.2); }
\node[above] at (2,-5.5) {\tiny 0}; \node[above] at (4,-5.5) {\tiny 1};
\node[above] at (2.67,-5.55) {\tiny $s_{l}$};
\node[below] at (2.97,-5.5) {\tiny $s_{l+1}$};
\node[above] at (3.4,-5.55) {\tiny $\cdots$};
\path[->, >=stealth, thick] (2.5,-3.9)  edge node[right] {$\xi$} (3.1,-5.3);
\path[->, >=stealth, thick] (3.5,-5.7) edge [out=290, in=320] node[below right]  {$\varphi_{\tilde{z}^2_0, \tilde{z}^2_1}$} (6.8,0.1);
\end{scope}

\begin{scope}[color=red]
\draw (1,5) -- (3,5);
\draw[thick] (1,4.9) -- +(0,0.2); \draw[thick] (3,4.9) -- +(0,0.2);
\node[below] at (1,4.9) {\tiny 1}; \node[below] at (3,4.9) {\tiny 0};
\path[->, >=stealth, thick] (2,4.9)  edge node[right] {$\zeta$} (2,3.8);
\path[->, >=stealth, thick] (7.3,0.9) edge [out=90, in=45] node[above] {$\ \psi_{\tilde{z}^2_0, \tilde{z}^2_1}$} (2,5.1);
\end{scope}
\end{scope}
\end{tikzpicture}
\caption{A topological graph $G$ and the definition of $g_n$.
         The circuit \textcolor{blue}{$C$} (with apple shape) is drawn in \textcolor{blue}{blue}.
         Then, the interval $I$ is the (thin) path in $C$ from $\tilde{z}^2_1$ to $\tilde{z}^2_0$ counter-clockwise
         ($\partial{}I = \{\tilde{z}^2_0,\tilde{z}^2_1\}$)
         and the interval $C\setminus\Int(I) = \widetilde{I}_2$ is
         the \textbf{thick} path
         in $C$ from $\tilde{z}^2_1$ to $\tilde{z}^2_0$ clockwise.}\label{figuregraphGandmapgnfirstexampleBCN}
\end{figure}

We define the map $\map{g_n}{G}$
(see Figure~\ref{figuregraphGandmapgnfirstexampleBCN}) by:
\[
  g_n(x) := \begin{cases}
      \varphi_{\tilde{z}^2_0, \tilde{z}^2_1} (\xi(x)) & \text{if $x \in \widetilde{I}_1$;}\\
      \zeta \bigl(\psi_{\tilde{z}^2_0, \tilde{z}^2_1}(x)\bigr)  & \text{if $x\in X$;}\\
      (\eta \circ f_n \circ \eta^{-1})(x) & \text{if $x \in I\setminus\Int\bigl(\widetilde{I}_1\bigr).$}
\end{cases}
\]
Recall that
$\widetilde{I}_1,$ $\widetilde{I}_2$ and $\widetilde{I}_3$
are pairwise disjoint
and, for every $\ell \in \Z$ and $j \in \{0,1\},$
$
  \bigemap{y_{\modulo{n+j+1+\ell(n+2)}{2n}}} =
  \bigemap{y_{n+j+1+\ell(n+2)}}
$
because $y_{i + 2n\ell} = y_i + \ell$ for every $i,\ell \in \Z.$
Hence, by Lemma~\ref{lemmaPhiPsi},
\begin{align*}
\varphi_{\tilde{z}^2_0, \tilde{z}^2_1} \bigr(\xi\bigl(\tilde{z}^1_j\bigr)\bigr)
  & = \varphi_{\tilde{z}^2_0, \tilde{z}^2_1} (j) = \tilde{z}^2_j
    = \bigeta{\bigemap{y_{n+j+1+2(n+2)}}}\\
  & = \bigeta{\bigemap{F_n\bigl(y_{n+j+1+(n+2)}\bigr)}}
    = \bigeta{f_n\bigl(\bigemap{y_{n+j+1+(n+2)}}\bigr)}\\
  & = \bigeta{f_n\bigl(\eta^{-1}\bigl(\bigeta{\bigemap{y_{n+j+1+(n+2)}}}\bigr)\bigr)}
    = \eta \circ f_n \circ \eta^{-1} \bigl(\tilde{z}^1_j\bigr),\text{ and}\\
\zeta\bigl(\psi_{\tilde{z}^2_0, \tilde{z}^2_1}\bigr(\tilde{z}^2_j\bigr)\bigr)
  & = \zeta(j) = \tilde{z}^3_j = \bigeta{\bigemap{y_{n+j+1+3(n+2)}}}\\
  & = \bigeta{\bigemap{F_n\bigl(y_{n+j+1+2(n+2)}\bigr)}}
    = \bigeta{f_n\bigl(\bigemap{y_{n+j+1+2(n+2)}}\bigr)}\\
  & = \bigeta{f_n\bigl(\eta^{-1}\bigl(\bigeta{\bigemap{y_{n+j+1+2(n+2)}}}\bigr)\bigr)}
    = \eta \circ f_n \circ \eta^{-1}\bigl(\tilde{z}^2_j\bigr)
\end{align*}
for $j \in \{0,1\}.$ So, $g_n$ is continuous because the maps
$\varphi_{\tilde{z}^2_0, \tilde{z}^2_1} \circ \xi\evalat{\widetilde{I}_1},$
$\zeta \circ \psi_{\tilde{z}^2_0, \tilde{z}^2_1}\evalat{X},$ and
$\eta \circ f_n \circ \eta^{-1}\evalat{I\setminus\Int\bigl(\widetilde{I}_1\bigr)}$
are continuous.

To be able to compute and use a Markov partition for the map $g_n$
we introduce the following notation.
Set
\begin{equation}\label{RnDefinition}
\begin{split}
s^1_i &= \xi^{-1}(s_i)\text{ for $i = 0,1,\dots,m,$ and}\\
R_n &=\ \bigeta{\bigemap{Q_n \cup P_n}} \cup \set{s^1_i}{i \in \{0,1,\dots,m\}} \cup \\
    &   \hspace*{10em}\set{\varphi_{\tilde{z}^2_0, \tilde{z}^2_1}(s_i)}{i \in \{0,1,\dots,m\}}
\end{split}
\end{equation}
and observe that
\begin{multline*}
\bigeta{\bigemap{Q_n \cup P_n}} =\\
\etaemap{Q_n} \cup\ \set{\tilde{z}^i_j}{i \in \{0,1,\dots,2n-3\}
   \text{ and } j \in \{0,1\}} \subset \\
\hspace*{24em}I\setminus\Int\bigl(\widetilde{I}_1\bigr),\\
 \hspace*{-10em}\set{s^1_i}{i \in \{0,1,\dots,m\}} \subset
 \widetilde{I}_1 \text{ and}\\
 \set{\varphi_{\tilde{z}^2_0, \tilde{z}^2_1}(s_i)}{i \in \{0,1,\dots,m\}}
 \subset X.
\end{multline*}
Moreover, by Lemma~\ref{lemmaPhiPsi}(a),
\[
R_n \supset
  \set{\varphi_{\tilde{z}^2_0, \tilde{z}^2_1}(s_i)}{i \in \{0,1,\dots,m\}}
  \supset V(X) \supset V(G).
\]
Hence, $R_n$ will be a Markov invariant set
provided that it is $g_n$-invariant and the closure of each
connected component of $G\setminus R_n$ is an interval in $G.$
Fix a point $x \in R_n.$ Then, $g_n(x)$ is
\begin{equation}\label{RnInvariant}
  \left\{\setlength{\tabcolsep}{0pt}\begin{tabular}{p{13mm}p{75mm}} & \\[-1.5ex]
      \multicolumn{2}{p{100mm}}{$\bigeta{f_n\bigl(\eta^{-1}(\etaemap{t})\bigr)} = \bigeta{f_n(\emap{t})} =$\newline\hspace*{12em}
                                $\bigeta{\bigemap{F_n(t)}} \in \bigeta{\bigemap{Q_n \cup P_n}}$}\\
      & if $x = \etaemap{t} \in I\setminus\Int\bigl(\widetilde{I}_1\bigr)$ with $t \in Q_n \cup P_n$\newline (here we use Theorem~\ref{theoremfirstexamplebcncircle}),\\
      \multicolumn{2}{l}{$\varphi_{\tilde{z}^2_0, \tilde{z}^2_1} \bigl(\xi\bigl(s^1_i\bigr)\bigr) = \varphi_{\tilde{z}^2_0, \tilde{z}^2_1} (s_i) \in \set{\varphi_{\tilde{z}^2_0, \tilde{z}^2_1}(s_i)}{i \in \{0,1,\dots,m\}}$}\\
      & if $x = s^1_i \in \widetilde{I}_1$ with $i \in \{0,1,\dots,m\},$\\
      \multicolumn{2}{l}{$\zeta\Bigl(\psi_{\tilde{z}^2_0, \tilde{z}^2_1}\Bigl(\varphi_{\tilde{z}^2_0, \tilde{z}^2_1}(s_i)\Bigr)\Bigr) \in \zeta\bigl(\{0,1\}\bigr) = \bigl\{\tilde{z}^3_0, \tilde{z}^3_1\bigr\} \in \bigeta{\bigemap{Q_n \cup P_n}}$}\\
      & if $x = \varphi_{\tilde{z}^2_0, \tilde{z}^2_1}(s_i) \in X$ with $i \in \{0,1,\dots,m\}$\newline (here we use Lemma~\ref{lemmaPhiPsi}(d)).\\[-1.5ex] &
  \end{tabular}\right.
\end{equation}
In either case, $g_n(x) \in R_n$ and, consequently,
$R_n$ is $g_n$-invariant.

Let $K$ be a connected component of $G\setminus R_n.$
Since $I$ is an interval with endpoints
$\bigl\{\tilde{z}^2_0, \tilde{z}^2_1\bigr\} \subset R_n,$
either $\Clos(K) \subset I$ or $\Clos(K) \subset X.$
In the first case, $\Clos(K)$ is clearly an interval.
Now assume that $K \subset \Clos(K) \subset X.$
Clearly, $K$ is a connected component of
$
  X\setminus R_n =
  X\setminus
     \set{\varphi_{\tilde{z}^2_0, \tilde{z}^2_1}(s_i)}{i \in \{0,1,\dots,m\}}.
$
Since the map
$\map{\varphi_{\tilde{z}^2_0, \tilde{z}^2_1}}{[0,1]}[X]$
is surjective and
\[
  K \cap \set{\varphi_{\tilde{z}^2_0, \tilde{z}^2_1}(s_i)}{i \in \{0,1,\dots,m\}}
  = \emptyset,
\]
by Lemma~\ref{lemmaPhiPsi}(a),
$\varphi_{\tilde{z}^2_0, \tilde{z}^2_1}\bigl((s_i,s_{i+1})\bigr) = K$
for some $i \in \{0,1,\dots,m-1\}.$
Hence, by Lemma~\ref{lemmaPhiPsi}(b),
$
\Clos(K) = \varphi_{\tilde{z}^2_0, \tilde{z}^2_1}\bigl([s_i,s_{i+1}]\bigr)
$
is an interval.
This shows that $R_n$ is a Markov invariant set for $g_n.$

The $R_n$-basic intervals are:
\begin{align*}
&\widetilde{I}_0, \widetilde{I}_3, \widetilde{I}_4, \dots, \widetilde{I}_{2n-3} \subset I\setminus\Int\bigl(\widetilde{I}_1\bigr)\\
&\widetilde{J}_j := \bigeta{\bigemap{J_j}} \subset I\setminus\Int\bigl(\widetilde{I}_1\bigr) \andq[for] j=0,1,2,3,\\
&L_i := \chull{s^1_i, s^1_{i+1}}[\widetilde{I}_1] = \xi^{-1}([s_i,s_{i+1}]) \subset \widetilde{I}_1 \andq[for] i=0,1,\dots,m-1, \text{ and}\\
&\varphi_{\tilde{z}^2_0, \tilde{z}^2_1}\bigl([s_i,s_{i+1}]\bigr) \subset X \andq[for] i=0,1,\dots,m-1.
\end{align*}
Moreover, unlike other $R_n$-basic intervals, the elements of
\[
 \set{\varphi_{\tilde{z}^2_0, \tilde{z}^2_1}\bigl([s_i,s_{i+1}]\bigr)}{i\in\{0,1,\dots,m-1\}}
\]
may coincide.
The (pairwise different) elements of this set will be denoted by
$U_0, U_1,\dots,U_r$ with $r \le m-1,$ so that
\[
 \{U_0, U_1,\dots,U_r\} = \set{\varphi_{\tilde{z}^2_0, \tilde{z}^2_1}\bigl([s_i,s_{i+1}]\bigr)}{i\in\{0,1,\dots,m-1\}} \subset X.
\]

Next we will show that $g_n$ is a Markov map with respect to $R_n$
and we will compute the Markov graph of $g_n$ with respect to $R_n.$
More precisely, we will show that $g_n$ is monotone at every basic
interval and derive the Markov graph of $g_n$ with respect to $R_n$
from the Markov graph of $f_n$ with respect to
$\bigemap{P_n \cup Q_n}$ which,
by the proof of Theorem~\ref{theoremfirstexamplebcncircle},
coincides with the Markov graph modulo 1 of $F_n$ with respect to
$P_n \cup Q_n$
(see Proposition~\ref{propositionmarkovgraphfirstexamplebcn}(b))
provided that we identify
$\BIclass{I}$ with $\bigemap{\BIclass{I}} = \emap{I}$ for every
$I \in \bigSBI{P_n \cup Q_n}.$

We start by observing that if
$
  K \in \bigl\{\widetilde{I}_0, \widetilde{I}_3, \widetilde{I}_4, \dots,
          \widetilde{I}_{2n-3}, \widetilde{J}_0, \widetilde{J}_1, \widetilde{J}_2,
          \widetilde{J}_3\bigr\},
$
(that is, $K \in \bigSBI{R_n}$ with
 $K \subset I\setminus\Int\bigl(\widetilde{I}_1\bigr)$),
then $\eta^{-1}(K) \in \bigSBI{\bigemap{Q_n \cup P_n}}$
and $g_n(K) = (\eta \circ f_n) \bigl(\eta^{-1}(K)\bigr),$
which is equivalent to $\eta^{-1}(g_n(K)) = f_n\bigl(\eta^{-1}(K)\bigr).$
Consequently, $g_n$ is monotone on $K$ because
$f_n$ is a Markov map with respect to $\bigemap{P_n \cup Q_n}$
(see the proof of Theorem~\ref{theoremfirstexamplebcncircle})
and, for every interval
$
  L \in \bigl\{\widetilde{I}_0, \widetilde{I}_3, \widetilde{I}_4, \dots,
          \widetilde{I}_{2n-3}, \widetilde{J}_0, \widetilde{J}_1, \widetilde{J}_2,
          \widetilde{J}_3\bigr\} \cup \bigl\{\widetilde{I}_1\bigr\},
$
it follows that
$K$ $g_n$-covers $L$ if and only if $\eta^{-1}(K)$ $f_n$-covers $\eta^{-1}(L).$
With the help of
Proposition~\ref{propositionmarkovgraphfirstexamplebcn}(b)
this gives the Markov graph of $g_n$ on the intervals
$J \in \bigSBI{R_n}$ with $J \subset I\setminus\Int\bigl(\widetilde{I}_1\bigr)$
and shows that $\widetilde{I}_0$ $g_n$-covers $L_i$
for $i=0,1,\dots,m-1$
(to illustrate this fact see
Figure~\ref{fig:MarkovGraphofgnfirstexamplebcn} which shows the
Markov graph of $g_n$ in the case $n = 4k-1$ and $d = 2k-1$
with $k \in \N$).\vspace{-1ex}
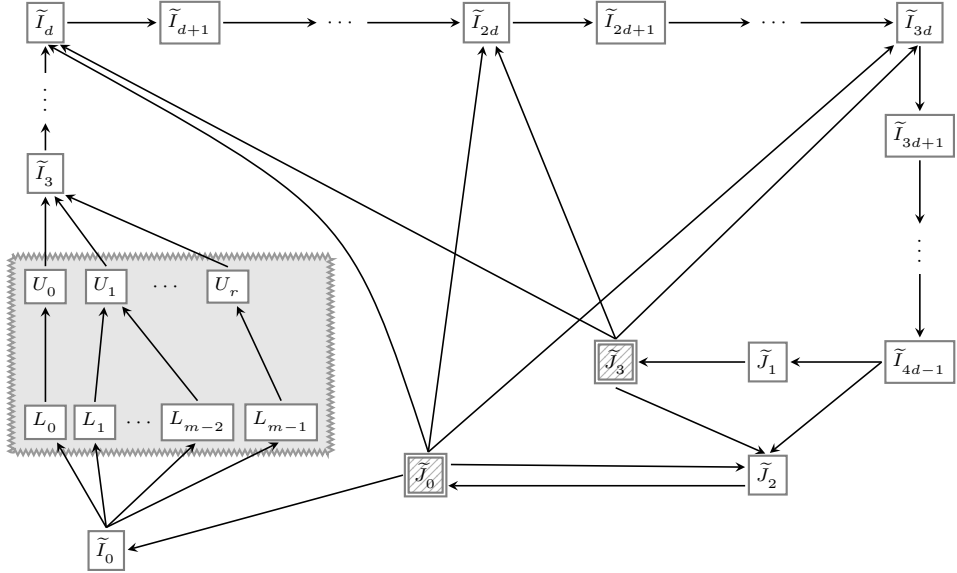
\begin{figure}[ht]
\begin{center}\tiny
  \tikzstyle{place}=[rectangle,draw=black!50,fill=black!0,thick]
  \tikzstyle{post}=[->,shorten >=1pt,>=stealth,semithick]
\begin{tikzpicture}
\filldraw[thick, draw=black!40, fill=black!10, decorate, decoration={zigzag, amplitude=1pt, segment length=2pt}] (-0.45,0.3) rectangle  (3.75,2.9);
\def\bigBIclass#1{\widetilde #1}
\subfigurefngraphModOne
\node[rotate=90] (dots4) at (0,5)    {$\cdots$};
\node[place] (i3)        at (0,4)    {$\widetilde{I}_3$};
\node[place] (i0)        at (0.8,-1) {$\widetilde{I}_0$};
\node[place](l0)       at (0,0.7)    {$L_0$};
\node[place](l1)       at (0.65,0.7) {$L_1$};
\node(dotsl)           at (1.25,0.7) {$\cdots$};
\node[place](lm)       at (2,0.7)    {$L_{m-2}$};
\node[place](lm1)      at (3.1,0.7)  {$L_{m-1}$};
\node[place](u0)       at (0,2.5)    {$U_0$};
\node[place](u1)       at (0.8,2.5)  {$U_1$};
\node[place](ur)       at (2.4,2.5)  {$U_r$};
\node(dotsu)           at (1.6,2.5)  {$\cdots$};

\path[post]
     (i3.north)+(0,1pt)    edge (dots4)
     (dots4.east)+(0,1pt)  edge (id)
     (j0.west)+(-1pt,0)    edge (i0.east)
     (i0.north)+(0,1pt)    edge (l0)
     (i0.north)+(0,1pt)    edge (l1.south)
     (i0.north)+(0,1pt)    edge (lm.south)
     (i0.north)+(0,1pt)    edge (lm1.south)
     (l0.north)+(0,1pt)    edge (u0)
     (l1.north)+(0,1pt)    edge (u1)
     (lm.north)+(0,1pt)    edge (u1)
     (lm1.north)+(0,1pt)   edge (ur)
     (u0.north)+(0,1pt)    edge (i3)
     (u1.north)+(0,1pt)    edge (i3.290)
     (ur.north)+(0,1pt)    edge (i3.310);
\draw[post] (j0.north)+(1pt,1pt) .. controls (4,3.5) ..  (id.south);
\end{tikzpicture}
\end{center}
\caption{The Markov graph of $g_n$ in the case when
$n = 4k-1$ and $d = 2k-1$ with $k \in \N$
(being the other case when $n = 4k+1$ and $d = 2k+1$).
The part of the Markov graph of $g_n$ with respect to $R_n$
which differs from the Markov graph of $f_n$ with respect to
$\bigemap{P_n \cup Q_n}$ is shown inside a grey box with a
zigzag border
(see Proposition~\ref{propositionmarkovgraphfirstexamplebcn}(b)).
The arrows between the intervals $L_i$ and $U_j$
are just illustrative.}\label{fig:MarkovGraphofgnfirstexamplebcn}
\end{figure}

\medskip

Now we consider the intervals $L_i.$ Clearly, by Lemma~\ref{lemmaPhiPsi}(b),
\[
g_n(L_i) = \varphi_{\tilde{z}^2_0, \tilde{z}^2_1} (\xi(L_i))
         = \varphi_{\tilde{z}^2_0, \tilde{z}^2_1} ([s_i,s_{i+1}])
         \in \{U_0, U_1,\dots,U_r\}
\]
and $g_n$ is monotone on $L_i.$
This shows that in the Markov graph of $g_n$ every interval $L_i$
$g_n$-covers a unique interval $U_j$ but different intervals $L_i$
can $g_n$-cover the same interval $U_j$
(see again Figure~\ref{fig:MarkovGraphofgnfirstexamplebcn}).

Finally, we consider the intervals $U_j.$
Clearly, by Lemma~\ref{lemmaPhiPsi}(e),
\[
g_n(U_j) = \zeta \bigl(\psi_{\tilde{z}^2_0, \tilde{z}^2_1}(U_j)\bigr)
         = \zeta \bigl(\psi_{\tilde{z}^2_0, \tilde{z}^2_1}\bigl(\varphi_{\tilde{z}^2_0, \tilde{z}^2_1} ([s_{i_j},s_{i_j+1}])\bigr)\bigr)
         = \zeta([0,1]) = \widetilde{I}_3
\]
and $g_n$ is monotone on $U_j.$
This shows that in the Markov graph of $g_n$ every interval $U_j$
$g_n$-covers a unique interval $\widetilde{I}_3$
(see once more Figure~\ref{fig:MarkovGraphofgnfirstexamplebcn}).

We just have seen that $g_n$ is a Markov map with respect to $R_n$
such that $f(K)$ is a (non-degenerate) interval for every
$K \in \SBI[R_n].$
Then, by Lemma~\ref{convertingtoexpansiveMM},
the map $g_n$ can be modified
without altering $g_n\evalat{R_n}$ and $g_n(K)$
for every $K \in \SBI[R_n]$ in such a way that
$g_n$ becomes $R_n$-expansive. So, we can use
Theorem~\ref{theoremTransitivityexpansivemap}
to prove that $g_n$ is transitive.
The Markov graph of $g_n$
(see Figure~\ref{fig:MarkovGraphofgnfirstexamplebcn} and
Proposition~\ref{propositionmarkovgraphfirstexamplebcn}(b))
tells us that the Markov matrix of $g_n$
with respect to $R_n$ is not a permutation matrix because there is
at least one basic interval which $g_n$-covers more than one basic interval
(for instance the interval $\widetilde{J}_0$ that $g_n$-covers 4
different intervals).
Moreover, by direct inspection of the Markov graph of $g_n,$
given any two vertices $\widetilde{I}$ and $\widetilde{J}$
in the graph, there exists a path from
$\widetilde{I}$ to $\widetilde{J}.$
This means that the transition matrix of the
Markov graph of $g_n$ is non-negative and irreducible.
Thus, $g_n$ is transitive by
Theorem~\ref{theoremTransitivityexpansivemap}.

Next we will show that $\Per(g_n) = \Per(f_n)$
(which will also be helpful in showing that $g_n$
is totally transitive).

In what follows, given $q\in \N$ and $A\subset \N$ we will
denote the set $\set{q\ell}{\ell \in A}$ by $q\cdot A.$

Observe that, by Theorem~\ref{theoremfirstexamplebcncircle},
$q\in\Per(f_n),\ q \ne 2$ implies $q \ge 2k+1 > \tfrac{n}{2}.$
Thus, for every $\ell \in \N\setminus\{1\},$
$\ell q > n$ and, hence, $\ell q \in \succs{n} \subset \Per(f_n).$
Consequently, again by Theorem~\ref{theoremfirstexamplebcncircle},

\begin{multline*}
\Per(f_n) = \{2\} \cup \bigcup_{\substack{q\in\Per(f_n)\\ q \ne 2}} q\cdot\{1\}
          \subset \{2\} \cup \bigcup_{\substack{q\in\Per(f_n)\\ q \ne 2}} q\cdot\N = \\
          \Per(f_n) \cup \bigcup_{\substack{q\in\Per(f_n)\\ q \ne 2}} q\cdot\bigl(\N\setminus\{1\}\bigr) \subset \Per(f_n) \cup \succs{n} = \Per(f_n).
\end{multline*}
So, to prove that $\Per(g_n) = \Per(f_n)$ it is enough to show that
\[
 \Per(g_n) = \{2\} \cup \bigcup_{\substack{q\in\Per(f_n)\\ q \ne 2}} q\cdot\N.
\]

First we will show that $2 \in \Per(g_n).$
Recall that, by Theorem~\ref{theoremfirstexamplebcncircle}, $Q_n$
is a lifted periodic orbit of $F_n$ of period 2, which implies that
$\bigemap{Q_n}$ is a periodic orbit of $f_n$ of period 2.
Moreover, $\bigeta{\bigemap{Q_n}} \subset I\setminus\Int\bigl(\widetilde{I}_1\bigr)$
and, hence,
$g_n(y) = \bigl(\eta \circ f_n \circ \eta^{-1}\bigr)(y)$
for every $y \in \bigeta{\bigemap{Q_n}}$
(here and in the rest of the proof that $\Per(g_n) = \Per(f_n)$
we use the fact that $g_n\evalat{R_n}$ has not been modified).
Thus, $\bigeta{\bigemap{Q_n}}$ is a periodic orbit of $g_n$ of period 2.

Let $Y$ be a periodic orbit of $g_n$ of period $p \ne 2$.
We have to see that
$p = \ell q$ with $\ell \in \N$ and
$q\in\Per(f_n),$ $q \ne 2.$
In the same way as before,
$\bigeta{\bigemap{P_n}}$ is a periodic orbit of $g_n$ of period $2n,$
and $\bigemap{P_n}$ is a periodic orbit of $f_n$ of period $2n.$
So, if $Y = \bigeta{\bigemap{P_n}},$
$p\in\Per(f_n),$ $p \ne 2$ and we are done in this case.

In the rest of the proof we assume that
$Y \ne \bigeta{\bigemap{P_n}}.$
Then, $Y \cap R_n = \emptyset.$
Indeed, otherwise,
\begin{multline*}
 \emptyset \ne g_n^2\bigl(Y \cap R_n\bigr) \subset
 g_n^2(Y) \cap g_n^2\bigl(R_n\bigr) \subset \\
 g_n^2(Y) \cap \bigeta{\bigemap{Q_n \cup P_n}} =
 Y \cap \Bigl(\bigeta{\bigemap{Q_n}} \cup \bigeta{\bigemap{P_n}} \Bigr)
\end{multline*}
by \eqref{RnDefinition} and \eqref{RnInvariant}.
Thus, $Y \cap \bigeta{\bigemap{Q_n}}  \ne \emptyset$ because
$Y \ne \bigeta{\bigemap{P_n}}$ implies $Y \cap\; \bigeta{\bigemap{P_n}} = \emptyset.$
Then,
since both $Y$ and $\bigeta{\bigemap{Q_n}}$ are periodic orbits of $g_n,$
it follows that $Y = \bigeta{\bigemap{Q_n}};$ a contradiction because
we are assuming that $p \ne 2.$
Therefore we have shown that $Y$ is disjoint from $R_n.$
Consequently, by Proposition~\ref{propPerLoops}, there is a loop
$
\lambda = K_0\longrightarrow K_1\longrightarrow \cdots
             \longrightarrow K_{p-1}\longrightarrow K_0
$
in the Markov graph of $g_n$ of length $p$ associated to $Y.$

Now we define a projection $\map{\pi}{\bigSBI{R_n}}[{\bigSBI{\bigemap{P_n \cup Q_n}}}]$
from the set of basic intervals of $g_n$ to
the set of basic intervals of $f_n$ in the following way.
For every $K \in \bigSBI{R_n}$ we set:
\[
  \pi(K) := \begin{cases}
     \eta^{-1}(K) & \text{if $K \subset I\setminus \widetilde{I}_1$;}\\
     \emap{I_1}   & \text{if $K\subset\widetilde{I}_1$;}\\
     \emap{I_2}   & \text{if $K\subset X.$}
\end{cases}
\]
It is clear by construction
(see Figure~\ref{fig:MarkovGraphofgnfirstexamplebcn} and
Proposition~\ref{propositionmarkovgraphfirstexamplebcn}(b))
that if there is an arrow $J \longrightarrow L$
in the Markov graph of $g_n,$ then there is an arrow
$\pi(J) \longrightarrow \pi (L)$ in the Markov graph of $f_n.$
Moreover,
since $\lambda$ is a loop of length $p$ in the Markov graph of $g_n,$
the projection of $\lambda$
\[
\pi(\lambda) := \pi(K_0)\longrightarrow \pi(K_1)\longrightarrow \cdots
                       \longrightarrow \pi(K_{p-1})\longrightarrow
               \pi(K_0)
\]
is a loop in the Markov graph of $f_n$ of the same length.
By Lemma~\ref{lemmaInterv},
there exists $x\in \pi(K_0)$ such that
$f_n^i(x)\in \pi(K_i),$ $0 \le i \le p-1$ and $f^{p}(x) = x.$
Then, the $f_n$-period of $x$ is $q,$ a divisor of $p,$ so that
$p = \ell q$ with $\ell \in \N$ and $q\in \Per(f_n).$
To end the proof that $\Per(g_n) = \Per(f_n)$
only it remains to show that $q \ne 2.$

By way of contradiction we assume that $q = 2$
(so that the $f_n$-orbit of $x$ is $\{x, f_n(x)\}$).
Clearly, in this case, $\ell \ge 2$ (and, hence, $p > q$)
because $p \ne 2.$
On the other hand,
as it it has been already justified after enumerating the
$R_n$-basic intervals,
the Markov graph of $f_n$ with respect to
$\bigemap{P_n \cup Q_n}$
coincides with the Markov graph modulo 1 of $F_n$ with respect to
$P_n \cup Q_n$ provided that we identify
$\BIclass{I}$ with $\bigemap{\BIclass{I}} = \emap{I}$ for every
$I \in \bigSBI{P_n \cup Q_n}.$ Consequently,
Proposition~\ref{propositionmarkovgraphfirstexamplebcn}(b)
gives the Markov graph of $f_n$ with respect to
$\bigemap{P_n \cup Q_n}.$

First we consider the case when
$\{x, f_n(x)\} \cap \bigemap{P_n \cup Q_n} = \emptyset.$
By Proposition~\ref{propPerLoops}(a), there exists a loop associated
to $\{x, f_n(x)\}$ in the Markov graph of $f_n.$
By Proposition~\ref{propositionmarkovgraphfirstexamplebcn}(b),
$\{x, f_n(x)\}$ is associated to the loop
$
\emap{J_0}\longrightarrow \emap{J_2} \longrightarrow \emap{J_0},
$
which is the unique loop of length 2 in the Markov graph of $f_n.$
Moreover, since $\{x, f_n(x)\} \cap \bigemap{P_n \cup Q_n} = \emptyset,$
and $f_n^i(x)\in \pi(K_i)$ for $0 \le i \le p-1$ and $f^{p}(x) = x,$
it follows that $\pi(\lambda)$ is an $\ell$-repetition of
$
\emap{J_0}\longrightarrow \emap{J_2} \longrightarrow \emap{J_0}.
$
Hence, in view of the definition of $\pi,$ we have that
$\pi^{-1}(\pi(\lambda)) = \lambda$ and, thus,
$\lambda$ is an $\ell$-repetition of the loop
\[
  \widetilde{J}_0 = \pi^{-1}(\emap{J_0}) \longrightarrow
  \widetilde{J}_2 = \pi^{-1}(\emap{J_2}) \longrightarrow
  \widetilde{J}_0 = \pi^{-1}(\emap{J_0}).
\]
By Proposition~\ref{propPerLoops}(b) applied to $\lambda$
it follows that $\ell = 2$ and
$
\widetilde{J}_0 \longrightarrow \widetilde{J}_2 \longrightarrow \widetilde{J}_0
$
must be negative.
However, with the notation of
Proposition~\ref{propositionmarkovgraphfirstexamplebcn}
it follows that $F_n\evalat{J_0}$ and $F_n\evalat{J_2}$
are strictly increasing and
\[
   F_n(J_0) = [x_1, y_{n+2}] \varsupsetneq [x_1,y_{n-2}] = J_2 \andq
   F_n(J_2) = [x_2, y_{2n}] = J_0 + 1.
\]
Thus, the loop
$
\widetilde{J}_0 \longrightarrow \widetilde{J}_2 \longrightarrow \widetilde{J}_0
$
is positive; a contradiction.

Now  we consider the case when
$\{x, f_n(x)\} \cap \bigemap{P_n \cup Q_n} \ne \emptyset.$
Clearly, since
$\{x, f_n(x)\},\ \bigemap{Q_n}$ and $\bigemap{P_n}$
are periodic orbits of $f_n,$ we have either
$\{x, f_n(x)\} = \bigemap{Q_n}$ or
$\{x, f_n(x)\} = \bigemap{P_n}.$
Furthermore, since the $f_n$ period of $\bigemap{P_n}$ is $2n,$
it follows that $\{x, f_n(x)\} = \bigemap{Q_n}.$
On the other hand, observe that the only intervals of
$\bigSBI{\bigemap{P_n \cup Q_n}}$ containing a point from $\bigemap{Q_n}$
are $\emap{J_0},\ \emap{J_1},\ \emap{J_2}$ and $\emap{J_3}$
(see Proposition~\ref{propositionmarkovgraphfirstexamplebcn}).
Thus, since $f_n^i(x) \in \bigemap{Q_n} \cap \pi(K_i)$ for $0 \le i \le p-1,$
it follows that
\[
  \pi(K_i)  \in \{\emap{J_0},\emap{J_1},\emap{J_2},\emap{J_3}\}
  \andq[for]
  0 \le i \le p-1.
\]
Moreover, as it can be checked in
Proposition~\ref{propositionmarkovgraphfirstexamplebcn}(b),
$\emap{J_1}$ is not $f_n$-covered by any of these four intervals.
So, $\emap{J_1}$ cannot appear in $\pi(\lambda).$
In a similar way, since $\emap{J_3}$ is not $f_n$ covered by
$\emap{J_0}$ and $\emap{J_2}$ and is $f_n$-covered only by
$\emap{J_1}$ which, as we already know, does not take part in
$\pi(\lambda),$ $\emap{J_1}$ does not appear in $\pi(\lambda).$
Consequently,
$
\pi(K_i)  \in \{\emap{J_0},\emap{J_2}\}
$
for $0 \le i \le p-1$ and, as in the previous case,
$\pi^{-1}(\pi(\lambda)) = \lambda.$  So,
\[
  K_i  \in \bigl\{\widetilde{J}_0 = \pi^{-1}(\emap{J_0}),\widetilde{J}_2 = \pi^{-1}(\emap{J_2})\bigr\}
  \andq[for]
  0 \le i \le p-1.
\]
We recall that the only loop in the Markov graph of $g_n$
consisting only on intervals $\widetilde{J}_0$ and $\widetilde{J}_2$
is
$
\widetilde{J}_0 \longrightarrow \widetilde{J}_2 \longrightarrow \widetilde{J}_0
$
(see Figure~\ref{fig:MarkovGraphofgnfirstexamplebcn})
and that this loop is positive.
Thus, either
$
\lambda = \widetilde{J}_0 \longrightarrow \widetilde{J}_2 \longrightarrow \widetilde{J}_0
$
or $\lambda$ is an $\ell$-repetition of
$
\widetilde{J}_0 \longrightarrow \widetilde{J}_2 \longrightarrow \widetilde{J}_0.
$
In view of Proposition~\ref{propPerLoops}(b),
this last option is not possible because, in that case,
$\lambda$ would be a repetition of a positive loop and hence
$
\lambda = \widetilde{J}_0 \longrightarrow \widetilde{J}_2 \longrightarrow \widetilde{J}_0
$
and $p = 2$; a contradiction.
This ends the proof that $\Per(g_n) = \Per(f_n).$

On the other hand, since $\Per(g_n) = \Per(f_n) \supset \succs{n}$
and $g_n$ is transitive, it follows that $\Per(g_n)$ is cofinite in
$\N$ and $g_n$ is totally transitive by
Theorem~\ref{theoremtotallytransitivefromadrr}.

To end the proof of the theorem we will estimate $h(g_n)$
in a similar way as in
Proposition~\ref{propositionmarkovgraphfirstexamplebcn}.
So, we  choose
$\textsf{Rom}_n = \{\widetilde{J}_0, \widetilde{J}_3\}$
as a rome in both cases: $n = 4k + 1$ and $n = 4k - 1.$
Then, the matrix $M_{\textsf{Rom}_n}(x)$ is:
\[
 \begin{cases}
   \begin{pmatrix*}[r]
        x^{-2} + m x^{-4d-2} +\alpha(x) & m x^{-4d-2} +\alpha(x)\\
        x^{-2}  +\alpha(x)            & \alpha(x)
   \end{pmatrix*} & \text{\parbox{9em}{when $n=4k-1$ and $d=\tfrac{n-1}{2}$ for some $k \in \N,$ and}}\\[5ex]
   \begin{pmatrix*}[r]
         x^{-2} + m x^{-4d+2} +  \alpha(x) & m x^{-4d+2} +\alpha(x)\\
         x^{-2}   +\alpha(x)             & \alpha(x)
  \end{pmatrix*} & \text{\parbox{9em}{when $n=4k+1$ and $d=\tfrac{n+1}{2}$ for some $k \in \N,$}}
\end{cases}
\]
where
\[
 \alpha(x) = \begin{cases}
   x^{-3d -2} +x^{-2d-2}+x^{-d-2}& \text{\parbox{14em}{when $n=4k-1$ and $d=\tfrac{n-1}{2}$ for some $k \in \N,$ and}}\\[2ex]
   x^{-3d} +x^{-2d} +x^{-d} & \text{\parbox{14em}{when $n=4k+1$ and $d=\tfrac{n+1}{2}$ for some $k \in \N.$}}
\end{cases}
\]
Finally we are ready to compute the characteristic polynomial
$P_n(x)$ of the Markov matrix of $g_n$ with respect to $R_n$
by using Theorem~\ref{theoremrome}. As in
Proposition~\ref{propositionmarkovgraphfirstexamplebcn}
it turns out that it is the same in both cases:
$n=4k-1$ and $d=\tfrac{n-1}{2}$ or
$n=4k+1$ and $d=\tfrac{n+1}{2}.$
We get
\[
  P_n(x)  = x^{2n}(x^2 - 1) -2x^{\frac{3n+1}{2}} -2x^{n+1}-2x^{\frac{n+3}{2}} -mx^2 -m
\]
and $h(g_n) = \log \rho_n$ where $\rho_n$
is the largest root (larger than one) of $P_n.$

The polynomial $P_n$ is very similar to the polynomial $T_n$ in
Proposition~\ref{propositionmarkovgraphfirstexamplebcn}.
Thus, reasoning as at the end of the proof of
Theorem~\ref{theoremfirstexamplebcncircle},
we conclude that, for each $n$ odd,
there exists a real number $\gamma_n \ge \rho_n$ such that
the sequence $\{\gamma_n\}_n$ is strictly decreasing and
$\lim_{n\to\infty} \gamma_n = 1.$
Consequently, $\lim_{n\to\infty} h(g_n) = 0$
by Proposition~\ref{propositionmonotoneTopEnt}.
\end{proof}

\subsection{Example with low non-constant periods}

This subsection is devoted to construct and prove

\begin{CustomNumberedExample}[\ref{examplemontevideuexampleintroduction}]
For every $n \in \N,\ n\ge 3$ there exists $f_n,$  a totally transitive
continuous circle map of degree one having a lifting $F_n \in \dol$
such that
\[
    \Rot(F_n) = \left[\tfrac{2n-1}{2n^2}, \tfrac{2n+1}{2n^2}\right] =
        \left[\tfrac{1}{n}-\tfrac{1}{2n^2},
        \tfrac{1}{n}+\tfrac{1}{2n^2}\right],
\]
$\lim_{n\to\infty} h(f_n) = 0$ and
\begin{multline*}
 \Per(f_n) =
   \{n\}\ \cup\\
   \set{t n + k}{t \in \{2,3,\dots,\nu-1\} \text{ and }
                    -\tfrac{t}{2} < k \le \tfrac{t}{2},\ k \in \Z} \cup \\
   \succs{n\nu+1-\tfrac{\nu}{2}}
\end{multline*}
with
\[
\nu = \begin{cases}
        n & \text{if $n$ is even, and}\\
        n-1 & \text{if $n$ is odd.}
     \end{cases}
\]
Moreover, $\sbc(f_n) = n\nu+1-\tfrac{\nu}{2}$
and $\bc(f_n)$ exists and verifies $n\le \bc(f_n) \le n\nu - 1 -\tfrac{\nu}{2}$
(and hence, $\lim_{n\to\infty} \bc(f_n) = \infty$).\smallskip

Furthermore, given any graph $G$ with a circuit, the sequence of maps
$\{f_n\}_{n=4}^\infty$ can be extended to a sequence of continuous totally
transitive maps $\map{g_n}{G}$ such that $\Per(g_n) = \Per(f_n)$
and $\lim_{n\to\infty} h(g_n) = 0.$
\end{CustomNumberedExample}

As in the previous subsection,
Example~\ref{examplemontevideuexampleintroduction} will be split
into Theorem~\ref{theoremexamplemontevideucircle}
which shows the existence of the circle maps $f_n$ by constructing them
along the lines of Subsection~\ref{Examples:PhilandIntro-HowTo},
and Theorem~\ref{theoremexamplemontevideugraph} that
extends these maps to a generic graph that is not a tree.
The proof of Theorem~\ref{theoremexamplemontevideucircle}, in turn, will
use a proposition that computes the Markov graph  modulo 1 of the
liftings $F_n.$

\begin{theorem}\label{theoremexamplemontevideucircle}
Let $n \in \N,\ n\ge 3,$ $p=2n-1,$ $r=2n+1$ and $q=2n^2,$ and let
\begin{align*}
Q_n &= \{\dots x_{-1}, x_0, x_1, x_2, \dots, x_{q-1}, x_{q}, x_{q+1}, \dots\} \subset \R,\andq \\
P_n &= \{\dots y_{-1}, y_0, y_1, y_2, \dots, y_{q-1}, y_{q}, y_{q+1}, \dots\} \subset \R
\end{align*}
be infinite sets such that the points of $P_n$ and $Q_n$
are intertwined  so that
\begin{tiny}
\[\def\sseepp{< \cdots <} \arraycolsep=1pt
 \begin{array}{cccccccccccl}
 \multicolumn{3}{l}{x_0 = 0< x_1 \sseepp x_{n-1} <}  & & &         &               &   &            &         &              &   \\[0.5ex]
    x_{n}    & \sseepp & x_{p+n-1}  &   &                &    <    &               &   & y_0        & \sseepp & y_{n}        & < \\
    x_{p+n}  & \sseepp & x_{2p+n-1} & < & y_{n+1}        & \sseepp & y_{2n}        & < & y_{r}      & \sseepp & y_{r+n}      & < \\
    x_{2p+n} & \sseepp & x_{3p+n-1} & < & y_{r+n+1}      & \sseepp & y_{r+2n}      & < & y_{2r}     & \sseepp & y_{2r+n}     & < \\
     \vdots  &         & \vdots     &   & \vdots         &         & \vdots        &   & \vdots     &         & \vdots       &   \\[-2ex]
             & \sseepp &            & < &                & \sseepp &               & < &            & \sseepp &              & < \\[-2.6ex]
     \vdots  &         & \vdots     &   & \vdots         &         & \vdots        &   & \vdots     &         & \vdots       &   \\[-0.5ex]
\hspace*{1.1em}x_{(n-1)p+n} & \sseepp & x_{np+n-1} & < & y_{(n-2)r+n+1} & \sseepp & y_{(n-2)r+2n} & < & y_{(n-1)r} & \sseepp & y_{(n-1)r+n} & < \hspace*{1.1em}\\[1ex]
             &         &            &   &                &         &               &   &            &         & \multicolumn{2}{r}{x_q = 1,}
\end{array}
\]
\end{tiny}
and $x_{i + q\ell} = x_i + \ell$ and $y_{i + q\ell} = y_i + \ell$
for every $i,\ell \in \Z.$

We define a lifting $F_n \in \dol$ such that, for every $i \in \Z,$
$F_n(x_i) = x_{i+p}$ and $F_n(y_i) = y_{i+r},$
and $F_n$ is affine between consecutive points of $P_n \cup Q_n.$
Then, $Q_n$ and $P_n$ are twist lifted periodic orbits of $F_n$
both of period $q$ such that
$Q_n$ has rotation number $\tfrac{p}{q}$ and
$P_n$ has rotation number $\tfrac{r}{q}.$
Moreover, the map $F_n$ has
$\Rot(F_n) = \left[\tfrac{p}{q}, \tfrac{r}{q}\right]$
as rotation interval.

Let $\map{f_n}{\SI}$ be the continuous map which has $F_n$ as a lifting.
Then, $f_n$ is totally transitive, $\lim_{n\to\infty} h(f_n) = 0,$
\begin{multline*}
 \Per(f_n) =
   \{n\}\ \cup\\
   \set{t n + k}{t \in \{2,3,\dots,\nu-1\} \text{ and }
                    -\tfrac{t}{2} < k \le \tfrac{t}{2},\ k \in \Z} \cup \\
   \succs{n\nu+1-\tfrac{\nu}{2}}.
\end{multline*}
Moreover, $\sbc(f_n) = n\nu+1-\tfrac{\nu}{2}$
and $\bc(f_n)$ exists and verifies $n\le \bc(f_n) \le n\nu - 1 -\tfrac{\nu}{2}$
(and hence, $\lim_{n\to\infty} \bc(f_n) = \infty$).
\end{theorem}

\begin{theorem}\label{theoremexamplemontevideugraph}
Let $G$ be a graph with a circuit. Then, the sequence of maps
$\{f_n\}_{n=4}^\infty$ from
Theorem~\ref{theoremexamplemontevideucircle}
can be extended to a sequence
of continuous totally transitive self maps of $G$,
$\{g_n\}_{n=4}^\infty,$
such that
$\Per(g_n) = \Per(f_n)$ and $\lim_{n\to\infty} h(g_n) = 0.$
\end{theorem}

Before proving Theorem~\ref{theoremexamplemontevideucircle} we will
study the Markov graph modulo 1 of the liftings $F_n.$

\long\def\subfigurefngraphModOne{%
\node[place] (xp1xp)   at (0,14) {$\bigBIclass{x_{p-1},x_{p}}$};
\node[place] (x2p1x2p) at (3,14) {$\bigBIclass{x_{2p-1},x_{2p}}$};
\node(dotsx)          at (7,14) {$\cdots $};
\node[place, double, pattern color=black!40, pattern=north east lines] (xn1xn) at (11,14) {$\arraycolsep=0pt\begin{array}{c}
                               \bigBIclass{x_{n-1},x_n} \\
                               \text{\rotatebox{90}{$=$}}\\[-7pt]
                               \bigBIclass{x_{(q-n)p-1},x_{(q-n)p}}
                         \end{array}$};
\node[place, double, pattern color=black!40, pattern=north east lines] (xpn1y0) at (5.5,8.3) {$\bigBIclass{x_{p+n-1},y_0}$};
\node[place] (x2pn1yn)    at (5.5,6.8) {$\bigBIclass{x_{2p+n-1},y_{n+1}}$};
\node[place] (x3pn1yrn)   at (5.5,5.3) {$\bigBIclass{x_{3p+n-1},y_{r+n+1}}$};
\node[rotate=90] (dotsxy) at (5.5,3.8) {$\cdots$};
\node[place, double, pattern color=black!40, pattern=north east lines] (xnpn1yn2rn) at (5.5,2.3) {$\bigBIclass{x_{np+n-1},y_{(n-2)r+n+1}}$};

\node[place] (yr1yr)   at (0,10)   {$\bigBIclass{y_{r-1},y_{r}}$};
\node[place] (y2r1y2r) at (2.5,10) {$\bigBIclass{y_{2r-1},y_{2r}}$};
\node(dotsy) at (4.7,10) {$\cdots $};
\node[place, double, pattern color=black!40, pattern=north east lines] (yn1ryn1r1) at (7.5,10) {$\arraycolsep=0pt\begin{array}{c}
                                    \bigBIclass{y_{(n-1)r},y_{(n-1)r +1}} \\
                                    \text{\rotatebox{90}{$=$}} \\[-7pt]
                                    \bigBIclass{y_{(q-n)r-1},y_{(q-n)r}}
                              \end{array}$};
\node[place] (yx1p)       at (11,10) {$\bigBIclass{y_{n},x_{p+n}}$};
\node[place] (yrx2p)      at (11,8)  {$\bigBIclass{y_{r+n},x_{2p+n}}$};
\node[place, double, pattern color=black!40, pattern=north east lines] (yxq) at (11,0) {$\bigBIclass{y_{q-1}, x_q}$};

\path[post]
    (xp1xp.east)+(1pt,0)      edge (x2p1x2p.west)
    (x2p1x2p.east)+(1pt,0)    edge (dotsx.west)
    (dotsx.east)+(1pt,0)      edge (xn1xn.west)
    (xn1xn.south)+(0,-1pt)    edge (yx1p)
    (xpn1y0.south)+(0,-1pt)   edge (x2pn1yn)
    (x2pn1yn.south)+(0,-1pt)  edge (x3pn1yrn)
    (x3pn1yrn.south)+(0,-1pt) edge (dotsxy)
    (dotsxy.west)+(0,-1pt)    edge (xnpn1yn2rn)
    (yr1yr.east)+(1pt,0)      edge (y2r1y2r.west)
    (y2r1y2r.east)+(1pt,0)    edge (dotsy.west)
    (dotsy.east)+(1pt,0)      edge (yn1ryn1r1.west)
    (yn1ryn1r1.east)+(1pt,0)  edge (yx1p)
    (yx1p.south)+(0,-1pt)     edge (yrx2p);
\draw[post] ([xshift=-1pt,yshift=4pt]xnpn1yn2rn.west) .. controls (3,2.43) .. (3, 3) .. controls (3,6.2) .. ([xshift=10pt]xpn1y0.south west);
\draw[post] ([xshift=5pt,yshift=-1pt]yn1ryn1r1.south) .. controls (7.7,7)  .. ([yshift=4pt]x2pn1yn.east);
\draw[post] ([yshift=-1pt]xn1xn.south) .. controls (9.6,12.5) .. (9.6,10) .. controls (9.6, 8.5) and ([xshift=8em]xpn1y0.east) .. ([xshift=2em]xpn1y0.east) -- (xpn1y0.east);
\draw[post] ([xshift=-1pt,yshift=6pt]yxq.west)  .. controls ([xshift=-56pt,yshift=6pt]yxq.west)  .. (8,1.25) -- (8,1.25)   .. controls (8,7)                 .. (yx1p);
\draw[post] ([xshift=-1pt,yshift=2pt]yxq.west)  .. controls ([xshift=-60pt,yshift=2pt]yxq.west)  .. ([xshift=-4pt]8,1.25) .. controls ([xshift=-4pt]8, 8.2) .. ([yshift=-4pt]xpn1y0.east);
\draw[post] ([xshift=-1pt,yshift=-2pt]yxq.west) .. controls ([xshift=-64pt,yshift=-2pt]yxq.west) .. ([xshift=-8pt]8,1.25) .. controls ([xshift=-8pt]8, 6.8) .. (x2pn1yn.east);

\node[rectangle, ultra thick, draw=black!60, fill=black!15] at (7,12.5) (caixa) {$\bigcup\limits_{i=0}^{n-1} \bigBIclass{y_i,y_{i+1}}$};
\draw[double, ->,shorten >=1pt,>=stealth,semithick] (xn1xn.south)+(0,-1pt) -- (caixa);
\node[rectangle, ultra thick, draw=black!60, fill=black!15] at (0.8,8.3) (caixa) {$ \bigcup\limits_{i=1}^{n} \bigBIclass{y_{n+i},y_{n +i+1}}$};
\draw[double, ->,shorten >=1pt,>=stealth,semithick] (xpn1y0.west)+(-1pt,0) -- (caixa.east);
\node[rectangle, ultra thick, draw=black!60, fill=black!15] at (0.8,2.3) (caixa) {$ \bigcup\limits_{i=-1}^{n-2} \bigBIclass{x_{p+i},x_{p+i+1}}$};
\draw[double, ->,shorten >=1pt,>=stealth,semithick] (xnpn1yn2rn.west)+(-1pt,0) -- (caixa.east);
\node[rectangle, ultra thick, draw=black!60, fill=black!15] at (3.5,12.5) (caixa) {$ \bigcup\limits_{i=0}^{p-2} \bigBIclass{x_{p +n+i},x_{p+n+i+1}}$};
\draw[double, ->,shorten >=1pt,>=stealth,semithick] (yn1ryn1r1.west)+(-1pt,8pt) -- (caixa.240);
\node[rectangle, ultra thick, draw=black!60, fill=black!15] at (3,-0.1) (caixa) {$\begin{aligned}
        \left(\bigcup\limits_{i=p}^{p+n-2} \bigBIclass{x_i,x_{i+1}}\right) &\cup
        \left(\bigcup\limits_{i=p+n}^{2p+n-2} \bigBIclass{x_i,x_{i+1}}\right)\cup \\
        \left(\bigcup\limits_{i=0}^{n-1} \bigBIclass{y_i,y_{i+1}}\right) &\cup
        \left(\bigcup\limits_{i=n+1}^{2n-1=r-2} \bigBIclass{y_i,y_{i+1}}\right)
    \end{aligned}$};
\draw[double, ->,shorten >=1pt,>=stealth,semithick] (yxq.west)+(-1pt,-7pt) -- ([yshift=-4pt]caixa.east);
} 

\begin{proposition}[$\calB(P_n \cup Q_n)$ and the $F_n$-Markov graph modulo 1]\label{propositionmarkovgraphexamplemontevideu}
In the assumptions of Theorem~\ref{theoremexamplemontevideucircle} we have
\begin{enumerate}[(a)]
\item The Markov graph modulo 1 of $F_n$ is the one shown in
Figure~\ref{graphlemmafngraphmontevideu},
where the double arrows arriving to the the boxes in grey mean that
there is an arrow arriving to each class of basic intervals modulo 1
in the box.
\begin{figure}[t]
\begin{center}\small
  \tikzstyle{place}=[rectangle,draw=black!60, thick]%
  \tikzstyle{post}=[->,shorten >=1pt,>=stealth,semithick]%
\begin{tikzpicture}[scale=0.9]
\subfigurefngraphModOne
\node[place] (y2rx3p)     at (11,6)  {$\bigBIclass{y_{2r+n},x_{3p+n}}$};
\node[rotate=90] (dotsyx) at (11,4)  {$\cdots$};
\node[place] (yn2rxn1p)   at (11, 2) {$\bigBIclass{y_{(n-2)r+n},x_{(n-1)p+n}}$};

\path[->,shorten >=1pt,>=stealth,semithick]
    (yrx2p.south)+(0,-1pt)    edge (y2rx3p)
    (y2rx3p.south)+(0,-1pt)   edge (dotsyx)
    (dotsyx.west)+(0,-1pt)    edge (yn2rxn1p)
    (yn2rxn1p.south)+(0,-1pt) edge (yxq.north);
\end{tikzpicture}
\end{center}
\caption{The Markov graph modulo 1 of the map $F_n$ from
Theorem~\ref{theoremexamplemontevideucircle}.
The double arrows arriving to the the boxes in grey mean that
there is an arrow arriving to each interval in the box.}\label{graphlemmafngraphmontevideu}
\end{figure}
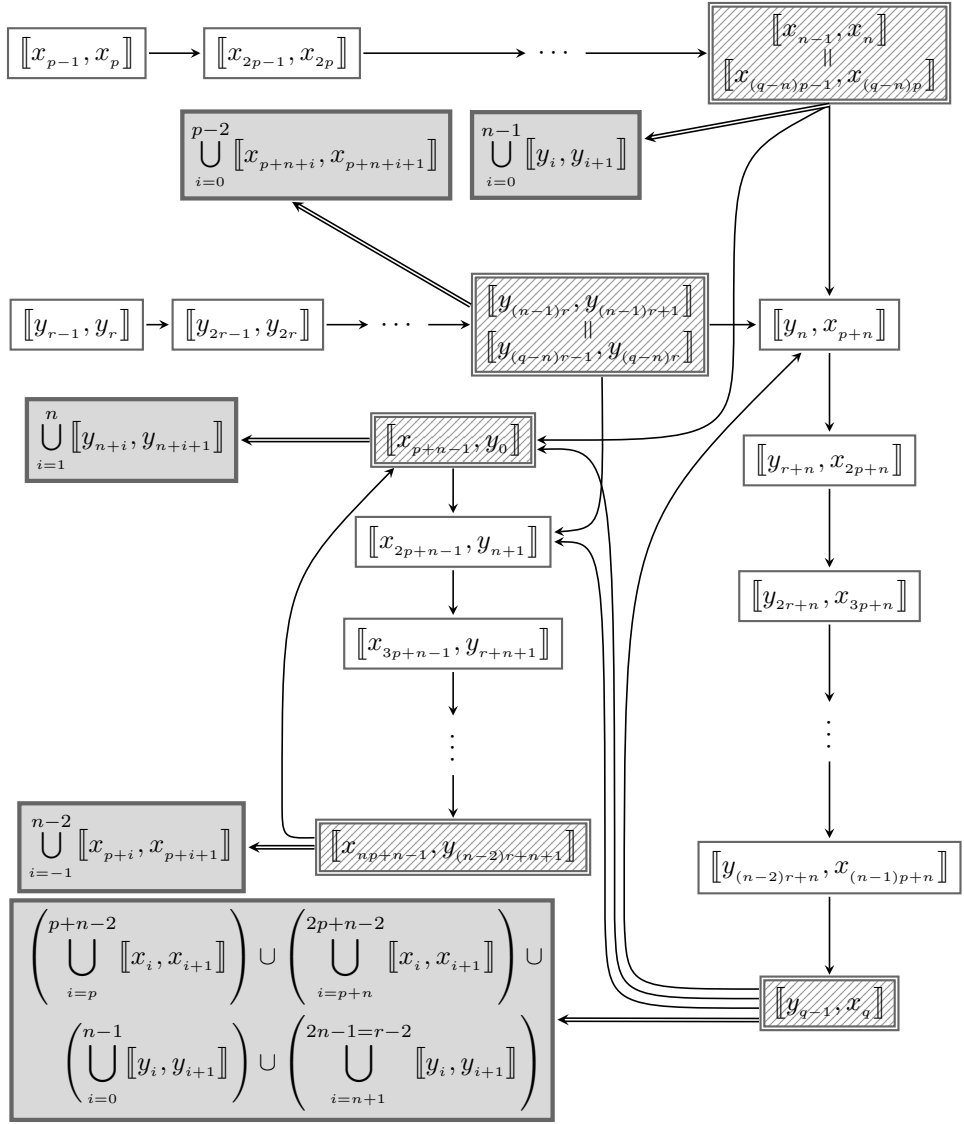

\item $h(f_n) = \log \rho_n,$ where $\rho_n > 1$ is the largest root
of the polynomial
\[
    T_n(x)  = \kappa_2(x)\bigl(x^{2q} + 1\bigr) + \kappa_1(x)x^{q+n}
              - 2\bigl(x^{4n} - 2x^{2n-1} + 1\bigr) ,
\]
where
\begin{align*}
   \kappa_2(x) &=  x^{4n} - 2x^{3n} - x^{2n+1} - 2x^{2n} - 3x^{2n-1} - 2x^{n} + 1,\\
\intertext{and}
   \kappa_1(x) &= 4x^{2n} + 2x^{n+1} + 4x^{n} + 2x^{n-1} + 4.
\end{align*}
\end{enumerate}
\end{proposition}

\begin{proof}
We start by proving (a) but before it is helpful to introduce
a new auxiliary definition.
Let
$
\alpha = \bigBIclass{I_0}\longrightarrow \bigBIclass{I_1}
                        \longrightarrow \bigBIclass{I_2}
                        \longrightarrow \cdots
                        \longrightarrow \bigBIclass{I_k}
$
be a path in the Markov graph modulo 1 of $F_n$ with respect to
$P_n \cup Q_n.$
We say that $\alpha$ is a \emph{road\/}
if $\bigBIclass{I_i}$ only $F_n$-covers $\bigBIclass{I_{i+1}}$
for $i=0,1,2,\dots,k-1$
(that is $\bigBIclass{I_i}\longrightarrow\bigBIclass{I_{i+1}}$ is the
unique arrow beginning at $\bigBIclass{I_i}$ in the whole Markov graph
modulo 1 of $F_n$ with respect to $P_n \cup Q_n$) and
$\bigBIclass{I_k}$ $F_n$-covers more than one equivalence class
(modulo 1) of $P_n \cup Q_n$-basic intervals.
In the trivial case when $k=0$ a road consists on a single
class (modulo 1) which $F_n$-covers more than one equivalence class.

The Markov graph modulo 1 of $F_n$ with respect to $P_n \cup Q_n$
can be decomposed in the following roads:
\stepcounter{equation}
\begin{enumerate}[({3.\arabic{equation}}.a)]
\item $
  \bigBIclass{x_{p-1},x_{p}} \longrightarrow
  \bigBIclass{x_{2p-1},x_{2p}} \longrightarrow\cdots \longrightarrow$\\
  \hspace*{\fill}$\bigBIclass{x_{(q-n)p-1},x_{(q-n)p}} = \bigBIclass{x_{n-1},x_{n}}
$\\[\smallskipamount]
is a road of length $q-n-1$ which is formed by all the
$q-n$ equivalence classes (modulo 1) of
$P_n \cup Q_n$-basic intervals having both endpoints in $Q_n.$
More concretely,\addtocounter{equation}{-1}
\begin{equation}\label{road1-Classes}
\begin{split}
 & \set{\bigBIclass{x_{\ell p-1},x_{\ell p}}}{\ell \in \{1,2,\dots,q-n\}} = \\
 & \hspace*{6em}\left\{\bigBIclass{x_\ell,x_{\ell+1}}\,\colon \ell \in \{0,1,\dots,q-1\}\setminus \right. \\
 & \hspace*{14em}\left.\set{i p + n-1}{i\in\{1,2,\dots,n\}}\right\}.
\end{split}
\end{equation}

\item The class $\bigBIclass{x_{n-1},x_{n}}$ $F_n$-covers
  $\bigBIclass{x_{n-1+p},y_0},$ $\bigBIclass{y_n,x_{n+p}}$
  and all the classes
  $\bigBIclass{y_i,y_{i+1}}$ for $i=0,1,\dots, n-1.$
\end{enumerate}
\stepcounter{equation}
\begin{enumerate}[({3.\arabic{equation}}.a)]
\item $
  \bigBIclass{y_{r-1},y_{r}} \longrightarrow
  \bigBIclass{y_{2r-1},y_{2r}} \longrightarrow\cdots \longrightarrow$\\
  \hspace*{\fill}$\bigBIclass{y_{(q-n)r-1},y_{(q-n)r}} = \bigBIclass{y_{(n-1)r},y_{(n-1)r + 1}}
$\\[\smallskipamount]
is a road of length $q-n-1$ which is formed by all the
$q-n$ equivalence classes (modulo 1) of
$P_n \cup Q_n$-basic intervals having both endpoints in $P_n.$
More concretely,\addtocounter{equation}{-1}
\begin{equation}\label{road2-Classes}
\begin{split}
 & \set{\bigBIclass{y_{\ell r-1},y_{\ell r}}}{\ell \in \{1,2,\dots,q-n\}} = \\
 & \hspace*{6em}\left\{\bigBIclass{y_\ell,y_{\ell+1}}\,\colon \ell \in \{0,1,\dots,q-1\}\setminus \right. \\
 & \hspace*{14em}\left.\set{i r + n}{i\in\{0,1,\dots,n-1\}}\right\}.
\end{split}
\end{equation}

\item The class $\bigBIclass{y_{(n-1)r},y_{(n-1)r + 1}}$ $F_n$-covers
  $\bigBIclass{y_n,x_{p+n}},$ $\bigBIclass{x_{2p+n-1},y_{n+1}}$
  and all the classes
  $\bigBIclass{x_{p+n+i},x_{p+n+i+1}}$ for $i=0,1,\dots, p-2.$
\end{enumerate}
\stepcounter{equation}
\begin{enumerate}[({3.\arabic{equation}}.a)]
\item $
  \bigBIclass{y_n,x_{p+n}} \longrightarrow
  \bigBIclass{y_{r+n},x_{2p+n}} \longrightarrow\cdots \longrightarrow$\\
  \hspace*{\fill}$\bigBIclass{y_{(n-1)r+n},x_{np+n}} = \bigBIclass{y_{q-1},x_q}
$\\[\smallskipamount]
is a road of length $n-1$ which is formed by all the $n$ equivalence
classes (modulo 1) of $P_n \cup Q_n$-basic intervals
verifying that  each class has a representative basic interval
whose first endpoint belongs to $P_n$ and the second one to $Q_n.$

\item The class $\bigBIclass{y_{q-1},x_q}$ (negatively) $F_n$-covers
  $\bigBIclass{x_{p+n-1},y_0},$ $\bigBIclass{y_n,x_{p+n}},$
  $\bigBIclass{x_{2p+n-1},y_{n+1}},$
  and all the classes
  \begin{align*}
   & \bigBIclass{x_{i},x_{i+1}} \andq[for] i=p,p+1,\dots, p+n-2,\\
   & \phantom{\bigBIclass{x_{i},x_{i+1}} \andq[for] i=}p+n,p+n+1,\dots,2p+n-2,\text{ and}\\
   & \bigBIclass{y_{i},y_{i+1}} \andq[for] i=0,1,\dots, n-1,\\
   & \phantom{\bigBIclass{y_{i},y_{i+1}} \andq[for] i=}n+1,n+1,\dots,2n-1=r-2.
  \end{align*}
\end{enumerate}
\stepcounter{equation}
\begin{enumerate}[({3.\arabic{equation}}.a)]
\item $\bigBIclass{x_{p+n-1},y_0}$ is a trivial road.

\item The class $\bigBIclass{x_{p+n-1},y_0}$ $F_n$-covers
  $\bigBIclass{x_{2p+n-1},y_{n+1}}$
  and all the classes
  $\bigBIclass{y_{n+i},y_{n+i+1}}$ for $i=1,2,\dots, n.$
\end{enumerate}
\stepcounter{equation}
\begin{enumerate}[({3.\arabic{equation}}.a)]
\item $
  \bigBIclass{x_{2p+n-1},y_{n+1}} \longrightarrow
  \bigBIclass{x_{3p+n-1},y_{r+n+1}} \longrightarrow\cdots \longrightarrow$\\
  \hspace*{\fill}$\bigBIclass{x_{np+n-1},y_{(n-2)r+n+1}}
$\\[\smallskipamount]
is a road of length $n-2$ which is formed by all the $n-1$ equivalence
classes (modulo 1) of $P_n \cup Q_n$-basic intervals
verifying that  each class has a representative basic interval
whose first endpoint belongs to $Q_n$ and the second one to $P_n$
except for $\bigBIclass{x_{p+n-1},y_0},$ which gives the previous trivial
road.

\item The class $\bigBIclass{x_{np+n-1},y_{(n-2)r+n+1}}$ $F_n$-covers
  $\bigBIclass{x_{p+n-1},y_0}$
  and all the classes
  $\bigBIclass{x_{p+i},x_{p+i+1}}$ for $i=-1,0,1,\dots, n-2.$
\end{enumerate}

We will prove Statements~(3.12.a,b).
Statements~(3.13.a) through (3.16.b) follow analogously.
We will start by proving that
\[ \bigBIclass{x_{(q-n)p-1},x_{(q-n)p}} = \bigBIclass{x_{n-1},x_{n}} \]
and \eqref{road1-Classes},
which shows that the path from (3.12.a) is formed
by all the $q-n$ equivalence classes (modulo 1) of
$P_n \cup Q_n$-basic intervals having both endpoints in $Q_n.$

To do this we will use the following facts (which are easy to check)
about the numbers $n, \ p,\ r$ and $q:$
$np = q-n,\ nr = q + n,$
$rp \equiv -1 \pmod{q}$ and
$(p,q) = (r,q) = 1.$

Observe that $(p,q) = 1$ implies
$\modulo{\ell p}{q} \in \{1,2,\dots,q-1\}$
for every $\ell \in \{0,1,\dots,q-1\}$
and, hence,
$\modulo{\ell p - 1}{q} \in \{0,1,\dots,q-2\}.$
Summarizing,
\begin{equation}\label{road1-consecutive}
\modulo{\ell p}{q} =
\modulo{\ell p-1}{q} + 1 \andq[whenever] \ell \in \{1,2,\dots,q-1\}.
\end{equation}

Since $np = q-n,$ we have
$(q-n)p - 1 = q(p-1) + (n-1).$ 
Hence, $\modulo{(q-n)p - 1}{q} = n-1,$ and
$\bigBIclass{x_{(q-n)p-1},x_{(q-n)p}} = \bigBIclass{x_{n-1},x_{n}}.$
\smallskip

Now we will prove \eqref{road1-Classes}.
Fix $\ell \in \{1,2,\dots,q-n\}.$
From \eqref{road1-consecutive} it follows that
$\left[x_{\ell p-1}, x_{\ell p}\right]$ and
$
  \left[x_{\modulo{\ell p-1}{q}}, x_{\modulo{\ell p-1}{q}+1}\right] =
  \left[x_{\ell p-1}, x_{\ell p}\right] -
        \floor{\tfrac{\ell p-1}{q}}
$\footnote{This equality follows from
$\ell p-1 = q\cdot \floor{\tfrac{\ell p-1}{q}} + \modulo{\ell p-1}{q}$
and $x_{i + q\ell} = x_i + \ell$ for every $i,\ell \in \Z$.}
are basic intervals provided that their endpoints
($x_{\ell p-1}$ and $x_{\ell p}$ in the first case and
 $x_{\modulo{\ell p-1}{q}}$ and $x_{\modulo{\ell p-1}{q}+1}$ in the second one)
are consecutive in $P_n \cup Q_n.$
In view of the crucial assumption on the relative positions
of the points of $P_n \cup Q_n$ from
Theorem~\ref{theoremexamplemontevideucircle},
this happens whenever
$\modulo{\ell p-1}{q} \neq i p + n-1$ with $i\in\{1,2,\dots,n\}.$
By way of contradiction assume that
$\modulo{\ell p-1}{q} = i p + n-1$ for some $i\in\{1,2,\dots,n\}.$
By using again \eqref{road1-consecutive} this is equivalent to
\[
 \modulo{\ell p}{q} = i p + n
    \quad\Longleftrightarrow\quad
 (\ell - i) p + kq = n
\]
for some $k\in \Z.$
The last equality holds if and only if
\[
    (\ell -i, k) \in \set{(q-n+tq, 1-p-tp)}{t \in \Z}
\]
(recall that $np = q-n$).
Since $1 \le \ell \le q-n$ and $1 \le i \le n,$
it follows that
\[
  1-n \le \ell - i = q-n+tq \le q-n-1;
\]
a contradiction because
$q-n+tq < 1-n$ for $t < 0$ and
$q-n+tq > q-n-1$ for $t \ge 0.$
So, we already know that the path from (3.12.a) is formed
by all the $q-n$ equivalence classes (modulo 1) of
$P_n \cup Q_n$-basic intervals having both endpoints in $Q_n.$
To see that this path is a road and to prove (3.12.b)
we need to compute the images of the corresponding
$P_n \cup Q_n$-basic intervals.

Since $F_n(x_i) = x_{i+p}$ and $x_{i + q\ell} = x_i + \ell$
for every $i,\ell \in \Z;$
and $F_n$ is monotone on every interval from $\bigSBI{P_n \cup Q_n}$
(bearing in mind the assumption on the relative ordering of the
points of $P_n \cup Q_n$ in
Theorem~\ref{theoremexamplemontevideucircle}) we see that
$
  F_n\left(\left[x_{\ell p-1},x_{\ell p}\right]\right) =
  \left[x_{(\ell+1) p-1},x_{(\ell+1) p}\right]
$ for $\ell = 1,2,\dots,q-n-1.$
On the other hand,
\begin{multline*}
  F_n\left(\left[x_{n-1},x_{n}\right]\right) =
  \left[x_{n-1+p},x_{n+p}\right] = \\
  \left[x_{n-1+p},y_0\right] \cup
     \left(\bigcup_{i=0}^{n-1} \left[y_i,y_{i+1}\right]\right) \cup
  \left[y_n,x_{n+p}\right].
\end{multline*}
Thus, Statements~(3.12.a,b) hold.

To end the proof of Statement~(a)
(Figure~\ref{graphlemmafngraphmontevideu})
observe that there are exactly
$2q$ $P_n \cup Q_n$-basic intervals in the interval $[0,1]$ and,
hence, there exist $2q$ equivalence classes (modulo 1) of
$P_n \cup Q_n$-basic intervals.
So, the above list of roads given in Statements~(3.12--16.a)
displays all vertices in the Markov graph modulo 1 of $F_n$
with respect to $P_n \cup Q_n$ (classified according to roads).
The arrows between vertices in this Markov graph are those given
by the previous roads and the arrows beginning at the last class
of every road given in Statements~(3.12--16.b)
All these vertices and arrows between them are, precisely,
the ones packaged in Figure~\ref{graphlemmafngraphmontevideu}.

To prove (b), as in the previous subsection, we will use
Propositions~\ref{Markov partitionprojectiontoSI}
and~\ref{propositionmonotoneTopEnt},
and Theorem~\ref{theoremrome}.
Notice that Statements~(3.12--16.a,b) above imply that
$P_n \cup Q_n$ is a short Markov partition with respect to $F_n.$
Then, as before, $f_n$ is a Markov map,
the Markov matrix $M_n$ of $f_n$ with respect to
$\bigemap{P_n \cup Q_n}$ is non-negative and irreducible
and, by Propositions~\ref{Markov partitionprojectiontoSI}
and~\ref{propositionmonotoneTopEnt},
$
    h(f_n) = \log \sigma(M_n)
$
where, by the Perron-Frobenius Theorem,
$\sigma(M_n)$ is the largest eigenvalue of $M_n$ and, hence,
the largest root (larger than one) of
the characteristic polynomial of $M_n$.
So, to end the proof of the proposition we need to compute
the characteristic polynomial of $M_n$.

As before, we identify the set $\bigSBI{\bigemap{P_n \cup Q_n}}$
with the set of all equivalence classes of $P_n \cup Q_n$-basic intervals
(i.e. the set of all vertices of the Markov graph modulo 1 of $F_n$).
Then, the matrix $M_n$ coincides with the
transition matrix of the Markov graph modulo 1 of $F_n$
given in Figure~\ref{graphlemmafngraphmontevideu}.

To compute $T_n$ we will use Theorem~\ref{theoremrome} with
\begin{multline*}
\textsf{Rom}_n = \Bigl\{\textsf{r}_1 = \bigBIclass{x_{n-1},x_n},
                   \textsf{r}_2 = \bigBIclass{x_{p+n-1},y_0},\\
                   \textsf{r}_3 = \bigBIclass{x_{np+n-1},y_{(n-2)r +n+1}}, \\
                   \textsf{r}_4 = \bigBIclass{y_{(n-1)r},y_{(n-1)r +1}},
                   \textsf{r}_5 = \bigBIclass{y_{q-1},x_q}\Bigr\}
\end{multline*}
as a rome (being their elements marked in
Figure~\ref{graphlemmafngraphmontevideu} with a box
with double border and sloping lines background pattern).
Then, we recall that $M_{\textsf{Rom}_n}(x) = (a_{ij}(x))$
where $a_{ij}(x) = \sum_p x^{-\ell(p)},$
and the sum is taken over all simple paths
starting at $r_i$ and ending at $r_j$
(since $M_n$ is  a matrix of zeroes and ones the width of every path is 1).
From (a) and Figure~\ref{graphlemmafngraphmontevideu} we have
\[
M_{\textsf{Rom}_n}(x) = \begin{pmatrix}
           0    & x^{-1} &      0     & a_{14}(x) & x^{-n} \\
           0    &     0  & x^{-(n-1)} & a_{24}(x) &   0 \\
      a_{31}(x) & x^{-1} &      0     &      0    &   0 \\
      a_{41}(x) &     0  & x^{-(n-1)} &      0    & x^{-n} \\
      a_{51}(x) & x^{-1} & x^{-(n-1)} & a_{54}(x) & x^{-n} \\
\end{pmatrix}
\]
where
(recall that $p = 2n -1,$ $n + np = q,$ $r = 2n +1,$ $nr = q + n$ and $pr = 2q-1$):
\begin{align*}
   a_{14}(x) &= \sum_{i=0}^{n-1} x^{-\modulo{n+ip}{q}} = \sum_{i=0}^{n-1} x^{-(n+ip)} = 
                x^{-n}(1+\alpha(x));\\
   a_{24}(x) &= \sum_{i=1}^{n} x^{-\modulo{n+(n+i)p}{q}} = \sum_{i=1}^{n} x^{-ip} = x^{-np} + \alpha(x);\\
   a_{31}(x) &= \sum_{i=-1}^{n-2} x^{-\modulo{n+2+(p+i)r}{q}} = x^{-(n+1+q-r)} + \sum_{i=0}^{n-2} x^{-(n+1+ir)}\\ 
             &\hspace*{3em}=\ x^{-(n+1+q-r)} + x^{-(n+1)}\beta(x) = x^{-(n+1)}\bigl(x^{-(q-r)} + \beta(x)\bigr);\\
   a_{41}(x) &= \sum_{i=0}^{p-2} x^{-\modulo{n+2+(p+n+i)r}{q}} = \sum_{i=0}^{p-2} x^{-\modulo{(i+1)r}{q}}\\
             &\hspace*{3em}=\ \sum_{i=0}^{n-2} x^{-(i+1)r} + \sum_{i=n-1}^{p-2} x^{-((i+1)r-q)}\\
             &\hspace*{3em}=\ x^{-r}\beta(x) + \sum_{i=0}^{n-2} x^{-((n+i)r-q)} = x^{-r}\beta(x) + \sum_{i=0}^{n-2} x^{-(n + ir)}\\
             &\hspace*{3em}=\ \bigl(x^{-r}+x^{-n}\bigr)\beta(x);\\
   a_{51}(x) &= \bigl(a_{31}(x) - x^{-(n+1+q-r)}\bigr) + a_{41}(x) = \bigl(x^{-(n+1)} + x^{-r}+x^{-n}\bigr)\beta(x);\\
   a_{54}(x) &= a_{14}(x) + \bigl(a_{24}(x) - x^{-np}\bigr) = x^{-n}(1+\alpha(x)) + \alpha(x)\\
             &\hspace*{3em}=\ x^{-n} + \bigl(1 + x^{-n}\bigr)\alpha(x);\\
\intertext{with}
   \alpha(x) &= \sum_{i=1}^{n-1} x^{-ip} = \frac{x^{-np}-x^{-p}}{x^{-p}-1} = \frac{x^{-(n-1)p}-1}{1-x^{p}},\text{ and}\\
   \beta(x)  &= \sum_{i=0}^{n-2} x^{-ir} = \frac{x^{-(n-1)r}-1}{x^{-r}-1} = \frac{x^{-(n-2)r}-x^r}{1-x^{r}}.
\end{align*}
Next we explain the above computations for the matrix $M_{\textsf{Rom}_n}(x).$
All entries in this matrix can be easily deduced from
Figure~\ref{graphlemmafngraphmontevideu} except for the entries
$a_{14}(x), \ a_{24}(x),\ a_{31}(x),\ a_{41}(x),\ a_{51}(x)$ and $a_{54}(x)$
(these ``complicate'' terms of $M_{\textsf{Rom}_n}(x)$
are determined by the partition of the Markov graph modulo 1 of $F_n$
in roads and the fact that we have chosen the last vertex of each road
to be a member of the rome).
They correspond to simple paths of the form
(see \eqref{road1-Classes} and \eqref{road2-Classes})
either
\begin{align*}
 & \textsf{r}_j \longrightarrow
   \bigBIclass{x_i, x_{i+1}} =
   \bigBIclass{x_{\ell p-1},x_{\ell p}} \longrightarrow
   \bigBIclass{x_{(\ell+1)p - 1},x_{(\ell+1)p}} \longrightarrow\cdots\\
 & \hspace*{18em}\longrightarrow \bigBIclass{x_{(q-n)p-1},x_{(q-n)p}} = \textsf{r}_1,\text{ or}\\
 & \textsf{r}_j \longrightarrow
   \bigBIclass{y_i, y_{i+1}} =
   \bigBIclass{y_{\ell r-1},y_{\ell r}} \longrightarrow
   \bigBIclass{y_{(\ell+1)r - 1},y_{(\ell+1)r}} \longrightarrow\cdots\\
 & \hspace*{18em}\longrightarrow\bigBIclass{y_{(q-n)r-1},y_{(q-n)r}} = \textsf{r}_4,
\end{align*}
for some $j\in\{1,2,3,4,5\}.$
We claim that the length of the first one of the above paths is
$\modulo{n+2+ir}{q}$ and the length of the second one is
$\modulo{n+ip}{q}.$

Then, entry $a_{14}(x)$ can be obtained as follows:
Statement~(3.12.b) tells us that
$\textsf{r}_1 = \bigBIclass{x_{n-1},x_{n}}$ $F_n$-covers all the classes
$\bigBIclass{y_i,y_{i+1}}$ for $i=0,1,\dots, n-1,$
and \eqref{road2-Classes} sows that
$\bigBIclass{y_i, y_{i+1}} = \bigBIclass{y_{\ell r-1},y_{\ell r}}$
for some $\ell \in \{1,2,\dots,q-n\}.$
Hence, every of such paths is a simple path from
$\textsf{r}_1$ to $\textsf{r}_4$ and, by the claim,
it contributes $x^{-\modulo{n+ip}{q}}$ to the entry $a_{14}(x).$
Thus,
\[
a_{14}(x) = \sum_{i=0}^{n-1} x^{-\modulo{n+ip}{q}}
          = \sum_{i=0}^{n-1} x^{-(n+ip)}
\]
because $n + ip \le n + (n-1)p \le q-p$ and, hence,
$\modulo{n+ip}{q} = n+ip$.
The other ``complicate'' entries:
$a_{24}(x),\ a_{31}(x),\ a_{41}(x),\ a_{51}(x)$ and $a_{54}(x)$
can be justified analogously.

Now we prove the first statement of the claim;
the second one follows analogously.
The path
\begin{multline*}
 \bigBIclass{x_i, x_{i+1}} =
   \bigBIclass{x_{\ell p-1},x_{\ell p}} \longrightarrow
   \bigBIclass{x_{(\ell+1)p - 1},x_{(\ell+1)p}} \longrightarrow\cdots\\
 \longrightarrow \bigBIclass{x_{(q-n)p-1},x_{(q-n)p}} = \bigBIclass{x_{n-1},x_{n}}
\end{multline*}
can also be written as
\begin{multline*}
 \bigBIclass{x_i, x_{i+1}} \longrightarrow
   \bigBIclass{x_{i+p},x_{i+1+p}} \longrightarrow\cdots\\
 \longrightarrow \bigBIclass{x_{i+d p},x_{i+1+dp}} = \bigBIclass{x_{n-1},x_{n}},
\end{multline*}
which clearly has length $d$ for some $d\in \{0,1,\dots,q-1\}.$
We have to show that $d = \modulo{n+1+ir}{q}.$
Since $np = q-n$ and $rp \equiv -1 \pmod{q},$
\[
 i + (n+1+ir)p = i + q + (p-n) + irp = q + i (1 + rp) + (n-1) \equiv n-1\mkern-12mu\pmod{q}.
\]
So, the path
\begin{multline*}
 \textsf{r}_j \longrightarrow \bigBIclass{x_i, x_{i+1}} =
   \bigBIclass{x_{\ell p-1},x_{\ell p}} \longrightarrow
   \bigBIclass{x_{(\ell+1)p - 1},x_{(\ell+1)p}} \longrightarrow\cdots\\
 \longrightarrow \bigBIclass{x_{(q-n)p-1},x_{(q-n)p}} = \bigBIclass{x_{n-1},x_{n}}
\end{multline*}
has length $\modulo{n+1+ir}{q} + 1 = \modulo{n+2+ir}{q}$ because,
according to (3.12.a), the length of the path is
$\modulo{n+1+ir}{q} \le q-n-1 < q-1.$
This ends the proof of the claim.

By Theorem~\ref{theoremrome}, the characteristic polynomial
(ignoring the sign) of $M_n$ is
\begin{multline*}
  \pm x^{2q}\det(M_{\textsf{Rom}_n} - \mathbf{I}_{5})  =\\
  \frac{
    \kappa_2(x) x^{2q} +  \kappa_1(x) x^{q+n} + \kappa_2(x) - 2\bigl(x^{4n} - 2x^{2n-1} + 1\bigr)
  }{\bigl(x^{2n-1}-1\bigr)\bigl(x^{2n+1}-1\bigl)} = \\
  \frac{T_n(x)}{\bigl(x^{2n-1}-1\bigr)\bigl(x^{2n+1}-1\bigl)} .
\end{multline*}
Clearly, the largest root (larger than one) of
the characteristic polynomial of $M_n$ coincides with
the largest root (larger than one) of the numerator of
$\pm x^{2q}\det(M_{\textsf{Rom}_n} - \mathbf{I}_{5})$ which is $T_n(x).$
\end{proof}

\begin{proof}[Proof of Theorem~\ref{theoremexamplemontevideucircle}]
In a similar way to the proof of Theorem~\ref{theoremfirstexamplebcncircle}
we see that $Q_n$ and $P_n$ are twist lifted periodic orbits of $F_n$
both of period $q$ such that
$Q_n$ has rotation number $\tfrac{p}{q}$ and
$P_n$ has rotation number $\tfrac{r}{q}.$

The proof that $\Rot(F_n) = \left[\tfrac{p}{q}, \tfrac{r}{q}\right]$
also follows as in Theorem~\ref{theoremfirstexamplebcncircle}
with the following differences. There exists a unique
$u^n_l \in (x_{np+n-1}, y_{(n-2)r+n+1})$
with $F_n\bigl(u^n_l\bigr) = x_p + 1 = F_n(x_0) + 1$
($F_n\Bigl(\bigl[x_{np+n-1}, y_{(n-2)r+n+1}\bigr]\Bigr) = 1 + \bigl[x_{p-1}, y_0\bigr]$)
such that
\begin{multline*}
 (F_n)_l(x) = \inf\set{F_n(y)}{y \ge x} = \\
 \begin{cases}
    F_n(x)& \text{for $x \in \bigl[0, u^n_l\bigr]$,}\\
    x_p+1 & \text{for $x \in \bigl[u^n_l,1\bigr]$,}\\
    (F_n)_l(x - \floor{x}) + \floor{x} & \text{if $x \notin [0,1];$}
 \end{cases}
\end{multline*}
and a unique $u^n_u \in (x_{p+n-1},y_0)$
with $F_n\bigl(u^n_u\bigr) = y_{2n} = y_{r-1} = F_n(y_{q-1}) - 1$
($F_n\bigl([x_{p+n-1},y_0]\bigr) = [x_{2p+n-1}, y_r]$)
such that\pagebreak[2]
\begin{multline*}
 (F_n)_u(x) = \sup\set{F_n(y)}{y \le x} = \\
 \begin{cases}
    y_{2n}   & \text{for $x \in \bigl[0, u^n_u\bigr]$,}\\
    F_n(x)   & \text{for $x \in \bigl[u^n_u,y_{q-1}\bigr]$,}\\
    y_{2n}+1 & \text{for $x \in \bigl[y_{q-1}, 1\bigr]$,}\\
    (F_n)_u(x - \floor{x}) + \floor{x} & \text{if $x \notin [0,1].$}
 \end{cases}
\end{multline*}
In this situation we have
$P_n \cap [0,1] \subset [u^n_u,y_{q-1}],$
$(F_n)_u\evalat{P_n} = F_n\evalat{P_n}$
and, hence,
$\rho\bigl((F_n)_u\bigr) = \rho_{F_n}(P_n) = \tfrac{r}{q}.$
In a similar way,
$Q_n \cap [0,1] \subset [0, u^n_l],$
$(F_n)_l\evalat{Q_n} = F_n\evalat{Q_n}$
and, hence,
$\rho\bigl((F_n)_l\bigr) = \rho_{F_n}(Q_n) = \tfrac{p}{q}.$
Consequently, $\Rot(F_n) = \left[\tfrac{p}{q}, \tfrac{r}{q}\right]$
by Theorem~\ref{theoremrotationintwathermap}.

To compute the set $\Per(f_{n})$ we will start by computing
$M\Bigl(\tfrac{p}{q}, \tfrac{r}{q}\Bigr).$
We claim that
\begin{multline*}
 M\Bigl(\tfrac{p}{q}, \tfrac{r}{q}\Bigr) =
   \{n\}\ \cup\\
   \set{t n + k}{t \in \{2,3,\dots,\nu-1\} \text{ and }
                    -\tfrac{t}{2} < k \le \tfrac{t}{2},\ k \in \Z} \cup \\
   \succs{n\nu+1-\tfrac{\nu}{2}}
\end{multline*}
with
\[
\nu = \begin{cases}
        n & \text{if $n$ is even, and}\\
        n-1 & \text{if $n$ is odd.}
     \end{cases}
\]
In what follows, to simplify the notation, we will denote
\[
 \mathcal{K}_n := \{1-\tfrac{\nu}{2}, 2-\tfrac{\nu}{2},\dots,0,1,\dots,n-\tfrac{\nu}{2}\}.
\]
Taking into account that $n\nu - \tfrac{\nu}{2} = n(\nu-1) + \Bigl(n-\tfrac{\nu}{2}\Bigr)$ and
\begin{multline*}
 \N = \{1,2,\dots,2n-\tfrac{\nu}{2}\}\ \cup \succs{n\nu+1-\tfrac{\nu}{2}} \cup \\
      \set{tn+k}{t \in \{2,3,\dots,\nu-1\},\ k \in \mathcal{K}_n},
\end{multline*}
the claim follows directly from
\begin{enumerate}[(i)]
 \item $M\Bigl(\tfrac{p}{q}, \tfrac{r}{q}\Bigr) \cap \{1,2,\dots,2n-\tfrac{\nu}{2}\} = \{n\},$
 \item $M\Bigl(\tfrac{p}{q}, \tfrac{r}{q}\Bigr) \supset \succs{n\nu+1-\tfrac{\nu}{2}},$ and
 \item $\begin{multlined}[t][0.85\textwidth]
             M\Bigl(\tfrac{p}{q}, \tfrac{r}{q}\Bigr) \cap \set{tn+k}{t \in \{2,3,\dots,\nu-1\}\text{ and } k \in \mathcal{K}_n} = \\[1ex]
             \set{t n + k}{t \in \{2,3,\dots,\nu-1\} \text{ and } -\tfrac{t}{2} < k \le \tfrac{t}{2},\ k \in \Z}. \end{multlined}$
\end{enumerate}
Moreover, to prove these three statements note that the elements of
$M\Bigl(\tfrac{p}{q}, \tfrac{r}{q}\Bigr)$ are those $m \in \N$
for which there exists $\ell \in \N$ such that
\begin{equation}\label{Mcdconditionexamplemontevideucircle}
  \frac{2n-1}{2n^{2}} < \frac{\ell}{m} < \frac{2n+1}{2n^{2}}.
\end{equation}
Simple computations show that
\[
 0 < \tfrac{1}{n+k} \le \frac{2n-1}{2n^{2}} < \frac{1}{n} <
     \frac{2n+1}{2n^{2}} \le \tfrac{2}{n+k} < \tfrac{1}{k}
\]
for every $k \in \{1,2,\dots,n-1\}.$
Thus (i) holds.

To prove (ii) we write
\[
\succs{n\nu+1-\tfrac{\nu}{2}} = \set{tn+k}{t \in \N,\ t \ge \nu\text{ and } k \in \mathcal{K}_n}
\]
(recall that $n\nu - \tfrac{\nu}{2} = n(\nu-1) + (n-\nu/2)$).
Moreover, \eqref{Mcdconditionexamplemontevideucircle} with $m = tn+k$
is equivalent to
\begin{equation}\label{Mcdconditionexamplemontevideucircle-dos}
  (2n-1) k - t n < 2(\ell - t)n^2 < (2n+1) k + t n.
\end{equation}

Assume first that either $n$ is even or $k \le n - \tfrac{\nu}{2} -1.$
In this case \eqref{Mcdconditionexamplemontevideucircle-dos}
with $\ell = t$ holds because
$t \ge \nu,$ $1-\tfrac{\nu}{2} \le k \le n-\tfrac{\nu}{2}$ and
\begin{multline*}
 (2n-1)k - t n \le\\
 (2n-1) \left\{\begin{array}{ll}
  \Bigl(n-\tfrac{\nu}{2}\Bigr)   & \text{when $n$ is even}\\
  \Bigl(n-\tfrac{\nu}{2}-1\Bigr) & \text{when $n$ is odd and $k \le n - \tfrac{\nu}{2} -1$}
 \end{array}\right\} - \nu n =\\
 (2n-1)\tfrac{\nu}{2} - \nu n = -\tfrac{\nu}{2} < 0 < 2n+1 - \tfrac{\nu}{2} =\\
 (2n+1)\Bigl(1-\tfrac{\nu}{2}\Bigr) + \nu n \le (2n+1)k + t n .
\end{multline*}
Thus, $tn+k \in M\Bigl(\tfrac{p}{q}, \tfrac{r}{q}\Bigr)$ in this case.
Now we assume that $n$ is odd and $k = n - \tfrac{\nu}{2}.$
Then, \eqref{Mcdconditionexamplemontevideucircle-dos}
with $\ell = t+1$ holds:
\begin{multline*}
 (2n-1)k - t n \le (2n-1)\Bigl(n - \tfrac{\nu}{2}\Bigr) -  \nu n = \tfrac{3n-1}{2} < 2n^2 < \\
 2n^2 + \Bigl(n - \tfrac{\nu}{2}\Bigr) = (2n+1)\Bigl(n - \tfrac{\nu}{2}\Bigr) + \nu n \le (2n+1)k + t n,
\end{multline*}
and (ii) follows.

Next we prove (iii).
First notice that when $n = 3,$ then $\nu = 2$ and, hence,
\begin{multline*}
 \set{tn+k}{t \in \{2,3,\dots,\nu-1\}\text{ and } k \in \mathcal{K}_n} =  \\
 \set{t n + k}{t \in \{2,3,\dots,\nu-1\} \text{ and } -\tfrac{t}{2} < k \le \tfrac{t}{2},\ k \in \Z} = \emptyset.
\end{multline*}
So, (iii) holds trivially in this case.

In the rest of the proof of (iii) we assume that $n \ge 4$ and
we will again use \eqref{Mcdconditionexamplemontevideucircle-dos}.
Observe that
\begin{multline*}
 - 2n^2 < - 2n^2 + 2n\Bigl(n+\tfrac{3}{2}-\nu\bigr) + \Bigl(\tfrac{\nu}{2} - 1\Bigr) =  (2n-1)\Bigl(1-\tfrac{\nu}{2}\Bigr) - (\nu -1) n \le \\
   (2n-1) k - t n < 2(\ell - t)n^2 < (2n+1) k + t n \le \\
   (2n+1)\Bigl(n-\tfrac{\nu}{2}\Bigr) + (\nu -1) n = 2n^2-\tfrac{\nu}{2} < 2n^2
\end{multline*}
because $t \in \{2,3,\dots,\nu-1\}$ and $1-\tfrac{\nu}{2} \le k \le n-\tfrac{\nu}{2}.$
This implies $\ell-t = 0$ and, then,
\eqref{Mcdconditionexamplemontevideucircle-dos} becomes
\[
  (2n-1)k - t n  < 0 < (2n+1)k + t n,
\]
which is equivalent to
\[
  -\frac{tn}{2n+1} < k < \frac{tn}{2n-1} \andq k \in \mathcal{K}_n.
\]

Observe that, for every $t\in\N,$
\[
 -\frac{t}{2} < -\frac{tn}{2n+1} < 0 < \frac{t}{2} < \frac{tn}{2n-1}.
\]

To prove (iii) we will show that the following three statements hold:
\begin{enumerate}[{(iii.}1)]
 \item $
 \max \left( \Z \cap \Bigl(-\tfrac{tn}{2n+1},\tfrac{t}{2}\Bigr]\right) =
 \left\{\begin{array}{ll}
  \tfrac{t}{2} & \text{if $t$ is even}\\
  \tfrac{t-1}{2} & \text{if $t$ is odd}
 \end{array}\right\} \in \mathcal{K}_n.
$
\item $\tfrac{tn}{2n-1} < 1 + \max \left( \Z \cap \Bigl(-\tfrac{tn}{2n+1},\tfrac{t}{2}\Bigr]\right).$
\item $
 \min \left( \Z \cap \Bigl(-\tfrac{tn}{2n+1},\tfrac{t}{2}\Bigr]\right) =
 \left\{\begin{array}{ll}
  1-\tfrac{t}{2} & \text{if $t$ is even}\\
  \tfrac{1-t}{2} & \text{if $t$ is odd}
 \end{array}\right\} \in \mathcal{K}_n.
$
\end{enumerate}

First we will show that (iii) follows from the above statements and
then we will prove them. From (iii.3) we immediately get that
\[
  \min \left( \Z \cap \Bigl(-\tfrac{tn}{2n+1},\tfrac{t}{2}\Bigr]\right) - 1 =
 \left\{\begin{array}{ll}
  -\tfrac{t}{2} & \text{if $t$ is even}\\
  -\tfrac{t+1}{2} & \text{if $t$ is odd}
 \end{array}\right\} \le -\tfrac{t}{2}.
\]
Consequently, by (iii.1--3),
\begin{multline*}
 \Z \cap \left(-\frac{tn}{2n+1},\frac{t}{2}\right] = \mathcal{K}_n \cap \left(-\frac{tn}{2n+1},\frac{t}{2}\right]
 \andq \\
 \Z \cap \left(-\frac{t}{2}, -\frac{tn}{2n+1}\right] = \Z \cap \left(\frac{t}{2},\frac{tn}{2n-1}\right) = \emptyset.
\end{multline*}
So, since $\mathcal{K}_n \subset \Z,$
\begin{multline*}
 \Z \cap \left(-\frac{t}{2},\frac{t}{2}\right] =
 \Z \cap \left(-\frac{tn}{2n+1},\frac{tn}{2n-1}\right) \supset
 \mathcal{K}_n \cap \left(-\frac{tn}{2n+1},\frac{tn}{2n-1}\right) \supset \\
    \mathcal{K}_n \cap \left(-\frac{tn}{2n+1},\frac{t}{2}\right] =
    \Z \cap \left(-\frac{tn}{2n+1},\frac{t}{2}\right] =
    \Z \cap \left(-\frac{t}{2},\frac{t}{2}\right],
\end{multline*}
which gives
$\Z \cap \Bigl(-\tfrac{t}{2},\tfrac{t}{2}\Bigr] = \mathcal{K}_n \cap \Bigl(-\tfrac{tn}{2n+1},\tfrac{tn}{2n-1}\Bigr).$
Thus, (iii) holds and hence the claim, provided that (iii.1--3) are verified.

We start by checking that (iii.1) holds.
The fact that
\[
 \max \left( \Z \cap \Bigl(-\tfrac{tn}{2n+1},\tfrac{t}{2}\Bigr]\right) =
 \left\{\begin{array}{ll}
  \tfrac{t}{2} & \text{if $t$ is even,}\\
  \tfrac{t-1}{2} & \text{if $t$ is odd,}
 \end{array}\right.
\]
is obvious.
So, we have to see that
$\max \left( \Z \cap \Bigl(-\tfrac{tn}{2n+1},\tfrac{t}{2}\Bigr]\right) \in \mathcal{K}_n.$
Note that $2n > 2n - 1 \ge 2\nu - 1$ and thus,
\[
 \tfrac{t}{2} \le \tfrac{\nu - 1}{2} < n-\tfrac{\nu}{2} = \max \mathcal{K}_n
\]
because $t \le \nu -1.$
So, since $0 \in \mathcal{K}_n$ and
$0 <  \max \left( \Z \cap \Bigl(-\tfrac{tn}{2n+1},\tfrac{t}{2}\Bigr]\right),$
the statement (iii.1) holds.

To show (iii.2), again since $t \le \nu -1$ we have
\[
 2tn < 2tn + (2n - \nu) \le 2n(t+1) - (t+1) = (2n-1)(t+1)
\]
which is equivalent to
\[
 \frac{tn}{2n-1} < \frac{t+1}{2}.
\]
So, by (iii.1),
\[
 \frac{tn}{2n-1} < \frac{t+1}{2} \le 1 + \max \left( \Z \cap \Bigl(-\tfrac{tn}{2n+1},\tfrac{t}{2}\Bigr]\right).
\]

Now we prove (iii.3).
By assumption we have $t \le \nu - 1 \le  n-1.$
Hence,
\begin{multline*}
 (2n+1)(t+1) > (2n+1)t > 2nt > 2nt - 2n + n >\\
        2n(t-1) + (t-1) = (2n+1)(t-1) > (2n+1)(t-2).
\end{multline*}
This gives
\begin{align*}
 & -\frac{t}{2} < -\frac{tn}{2n+1} < 1- \frac{t}{2} = \frac{2-t}{2} \text{ when $t$ is even, and}\\
 & -\frac{t+1}{2} = \frac{1-t}{2} - 1 < -\frac{tn}{2n+1} < \frac{1-t}{2} \text{ when $t$ is odd,}
\end{align*}
which proves that
\[
 \min \left( \Z \cap \Bigl(-\tfrac{tn}{2n+1},\tfrac{t}{2}\Bigr]\right) =
 \left\{\begin{array}{ll}
  1-\tfrac{t}{2} & \text{if $t$ is even,}\\
  \tfrac{1-t}{2} & \text{if $t$ is odd.}
 \end{array}\right.
\]
Furthermore, we need to show that
$\min \left( \Z \cap \Bigl(-\tfrac{tn}{2n+1},\tfrac{t}{2}\Bigr]\right) \in \mathcal{K}_n.$
If $t = \nu - 1$, since $\nu$ is always even we have
\[
 \min \left( \Z \cap \left(-\frac{tn}{2n+1},\frac{t}{2}\right]\right) =
 \frac{1-t}{2} = \frac{1-(\nu-1)}{2} = 1- \frac{\nu}{2} 
 \in \mathcal{K}_n.
\]
When $t \le \nu -2,$
\[
 \min \mathcal{K}_n = \frac{2-\nu}{2} \le -\frac{t}{2} < -\frac{tn}{2n+1}.
\]
Consequently, (iii.3) holds as before because
\[
   \min \left( \Z \cap \Bigl(-\tfrac{tn}{2n+1},\tfrac{t}{2}\Bigr]\right) \le 0 \in \mathcal{K}_n.
\]
This ends the proof of (iii) and the claim with it.

Finally are ready to compute the set $\Per(f_{n})$ by using the above claim.
By Theorem~\ref{theoremMisiurewicz},
\[
 \Per(f_{n}) = Q_{F_n}\Bigl(\tfrac{p}{q}\Bigr) \cup M\Bigl(\tfrac{p}{q}, \tfrac{r}{q}\Bigr) \cup Q_{F_n}\Bigl(\tfrac{r}{q}\Bigr)
\]
and, from the above claim,
\[
  Q_{F_n}\Bigl(\tfrac{p}{q}\Bigr) \cup Q_{F_n}\Bigl(\tfrac{r}{q}\Bigr)
     \subset q\N \subset \succs{2n^2} \subset \succs{n\nu+1-\tfrac{\nu}{2}}
     \subset M\Bigl(\tfrac{p}{q}, \tfrac{r}{q}\Bigr)
\]
because, independently of the parity of $n,$
$n\nu+1-\tfrac{\nu}{2} < n^2 + 1 < 2n^2.$
Consequently, $\Per(f_{n}) = M\Bigl(\tfrac{p}{q}, \tfrac{r}{q}\Bigr),$
which, together with the claim, proves the statement about the set
$\Per(f_{n}).$

Notice that $2 \tfrac{\nu}{2} \le n$ implies
$\tfrac{\nu}{2} - n \le -\tfrac{\nu}{2}.$
Hence, since $\nu$ is always even,
\begin{multline*}
   \max \set{t n + k}{t \in \{2,3,\dots,\nu-1\} \text{ and } -\tfrac{t}{2} < k \le \tfrac{t}{2},\ k \in \Z} = \\
   (\nu-1)n + \tfrac{\nu-2}{2} = \nu n + \Bigl( \tfrac{\nu}{2} - n\Bigr) - 1 \le \nu n - \tfrac{\nu}{2} -1.
\end{multline*}
Then,
$\nu n - \tfrac{\nu}{2} \notin \Per(f_{n})$ and thus,
$\sbc(f_n) = n\nu+1-\tfrac{\nu}{2}.$
On the other hand, $n \in \sbcset(f_n)$ and therefore,
$\bc(f_n)$ exists and verifies
\[ n\le \bc(f_n) \le n\nu - 1 -\tfrac{\nu}{2}. \]

Next we show that $f_n$ is totally transitive.
As in the previous example we have that
$P_n \cup Q_n$ is a short Markov partition with respect to $F_n.$
Then, $f_n$ is an expansive Markov map with respect to the
partition $\bigemap{P_n \cup Q_n},$
and the transition matrix of the Markov graph of $f_n$
with respect to the partition $\bigemap{P_n \cup Q_n}$
coincides with the transition matrix of the Markov graph modulo 1
of $F_n$ with respect to $P_n \cup Q_n.$ Moreover,
this transition matrix is non-negative and irreducible,
and from the proof of
Proposition~\ref{propositionmarkovgraphexamplemontevideu}(a)
(see also Figure~\ref{graphlemmafngraphmontevideu})
it follows that there exist five vertices in the
Markov graph modulo 1 of $F_n$ (indeed all ends of roads)
which are the beginning of more than one arrow.
That is, the transition matrix of the Markov graph of $f_n$
with respect to $\bigemap{P_n \cup Q_n}$ is not
a permutation matrix.
Then, $f_n$ is transitive by Theorem~\ref{theoremTransitivityexpansivemap}
and, since $\Per(f_n) \supset \succs{n\nu+1-\tfrac{\nu}{2}},$
$\Per(f_n)$ is cofinite and $f_n$ is totally transitive
by Theorem~\ref{theoremtotallytransitivefromadrr}.

Next we need to show that $\lim_{n\to\infty} h(f_n) = 0.$
We will use the notation of
Proposition~\ref{propositionmarkovgraphexamplemontevideu}(b)
and we write
\[
    T_n(x)  = \kappa_2(x)\bigl(x^{2q} + 1\bigr) + \kappa_1(x)x^{q+n} - 2\kappa_0(x)
\]
with $\kappa_0(x) := x^{4n} - 2x^{2n-1} + 1.$
Then, for $x \ge 1$ we have the following easy bounds:
\begin{align*}
 \kappa_2(x) &=  x^{4n} - 2x^{3n} - x^{2n+1} - 2x^{2n} - 3x^{2n-1} - 2x^{n} + 1 >\\
             &\hspace*{4em}x^{4n} - 10x^{3n} = x^{3n} (x^n - 10),\\
 \kappa_1(x) &= 4x^{2n} + 2x^{n+1} + 4x^{n} + 2x^{n-1} + 4 > 12x^{n-1},\text{ and}\\
 \kappa_0(x) &= x^{4n} - 2x^{2n-1} + 1 \le x^{4n} - 1 < x^{4n}.
\end{align*}
Hence, now for $x > \sqrt[n]{10} > 1,$
\begin{multline*}
 T_n(x) > x^{3n} (x^n - 10)\bigl(x^{2q} + 1\bigr) + 12x^{n-1}x^{q+n} - 2x^{4n} = \\
          x^{3n} (x^n - 10)\bigl(x^{2q} + 1\bigr) + 2x^{4n} \bigl(6x^{2n(n-1)-1} - 1\bigr) > \\
          x^{3n} (x^n - 10)\bigl(x^{2q} + 1\bigr) > 0.
\end{multline*}
Therefore, $\rho_n \le \sqrt[n]{10}$ and
\[ 0 \le \lim_{n\to\infty} h(f_n) = \lim_{n\to\infty} \log \rho_n \le \lim_{n\to\infty} \log \sqrt[n]{10} = 0. \]
\end{proof}

\begin{proof}[Proof of Theorem~\ref{theoremexamplemontevideugraph}]
As in the proof of proof of Theorem~\ref{theoremfirstexamplebcngraph}
we may assume that $G \ne \SI$ since otherwise
Theorem~\ref{theoremexamplemontevideucircle} already gives the desired
sequence of maps.

The proof in the case $G \ne \SI$ goes along the lines of the
proof of Theorem~\ref{theoremfirstexamplebcngraph} and most of the
details will be omitted.
We will only summarize the parts of the proof which are different from
the proof of Theorem~\ref{theoremfirstexamplebcngraph}, and the ones
needed to fix the notation.

We fix a circuit $C$ of $G$ and an interval $I \subset C$ such that
$I \cap V(G) = \emptyset.$ Also, we choose a homeomorphism
$\map{\eta}{\SI}[C]$ such that
\begin{align*}
 & C\mkern-2mu\setminus\mkern-3mu\Int(I) = \chull{\etaemap{y_{(n-2)r+n}}, \etaemap{x_{(n-1)p+n}}}[C],\text{ and}\\
 & I \supset \bigeta{\bigemap{P_n \cup Q_n}} \cup \mspace{-155mu}
     \bigcup_{\mspace{222mu}\begin{subarray}{l}
                 [x,y] \in \SBI[P_n \cup Q_n]\\
                 [x,y] \notin \BIclass{y_{(n-2)r+n}, x_{(n-1)p+n}}
              \end{subarray}}                  \chull{\etaemap{x}, \etaemap{y}}[C].
\end{align*}

For simplicity, in the rest of the proof we will use the following
notation: Given $x,y\in \R$ we denote by $\BIgraph{x,y}$
the convex hull (in $C$) of $\{\etaemap{x}, \etaemap{y}\}$
(which, of course, coincides with $\bigeta{\bigemap{\chull{x,y}[\R]}}$).
With this notation, the
$\bigeta{\bigemap{P_n \cup Q_n}}$-basic intervals in $C$ are
\[
  \set{\BIgraph{x,y}}{\chull{x,y}[\R] \in \bigSBI{P_n \cup Q_n}}
\]
(see Figure~\ref{graph:graphmontevideuOnTheGraph}).
Clearly, if $\BIclass{x,y} = \BIclass{\tilde x, \tilde y},$ then
$\BIgraph{x,y} = \BIgraph{\tilde x, \tilde y}.$

Observe that
$\BIgraph{y_{(n-2)r+n}, x_{(n-1)p+n}}$ plays the role of $\widetilde{I}_2$
in the proof of Theorem~\ref{theoremfirstexamplebcngraph}
(see Figure~\ref{figuregraphGandmapgnfirstexampleBCN}) and consequently,
$\BIgraph{y_{q-1}, x_{q}}$ plays the role of the interval $\widetilde{I}_3$ while
$\BIgraph{y_{(n-4+j)r+n}, x_{(n-3+j)p+n}}$ play the role of $\widetilde{I}_j$ for $j = 0,1.$
Note that all the intervals are well defined since $n \ge 4$
and they are pairwise disjoint because of the ordering of points defined
in Theorem~\ref{theoremexamplemontevideucircle}.

We set
$X := G \setminus \Int(I) \supset \BIgraph{y_{(n-2)r+n}, x_{(n-1)p+n}},$
and
$V(X) = V(G) \cup \{a,b\}$
with
$a := \bigeta{\bigemap{y_{(n-2)r+n}}}$ and
$b := \bigeta{\bigemap{x_{(n-1)p+n}}}.$
Then, as before, we use Lemma~\ref{lemmaPhiPsi} for the subgraph $X$
(see Figure~\ref{mapsfromarr}).
Let $m=m(X, a, b) \ge 5$ be odd,
consider the partition $0=s_0<s_1<\cdots <s_m=1,$ and
let the maps
$\map{\varphi_{a, b}}{[0,1]}[X]$ and
$\map{\psi_{a, b}}{X}[{[0,1]}]$
be as in Lemma~\ref{lemmaPhiPsi}.
Moreover, as before, we define two arbitrary but fixed homeomorphisms
$\map{\zeta}{[0,1]}[{\BIgraph{y_{q-1}, x_{q}}}]$  and
$\map{\xi}{\BIgraph{y_{(n-3)r+n}, x_{(n-2)p+n}}}[{[0,1]}]$
such that
\begin{multline*}
  \zeta(0) = \etaemap{y_{q-1}},\ \zeta(1) = \etaemap{x_{q}},
  \ \xi\bigl(\etaemap{y_{(n-3)r+n}}\bigr) = 0 \text{ and}\\
  \xi\bigl(\etaemap{x_{(n-2)p+n}}\bigr) = 1
\end{multline*}
(see Figure~\ref{figuregraphGandmapgnfirstexampleBCN} for an analogous situation).

Equipped with all these definitions, for $n \ge 4$ we set
\[
  g_n(x) := \begin{cases}
      \varphi_{a,b} (\xi(x)) & \text{if $x \in \BIgraph{y_{(n-3)r+n}, x_{(n-2)p+n}}$;}\\
      \zeta \bigl(\psi_{a,b}(x)\bigr)  & \text{if $x\in X$;}\\
      (\eta \circ f_n \circ \eta^{-1})(x) & \text{if $x \in I\setminus\Int\BIgraph{y_{(n-3)r+n}, x_{(n-2)p+n}}.$}
\end{cases}
\]
Then, as in the proof of Theorem~\ref{theoremfirstexamplebcngraph}
we can easily but tediously show that $g_n$ is a Markov map
with respect to the partition
\begin{multline*}
R_n = \bigeta{\bigemap{Q_n \cup P_n}} \cup
      \set{\xi^{-1}(s_i)}{i \in \{0,1,\dots,m\}} \cup \\
      \set{\varphi_{a,b}(s_i)}{i \in \{0,1,\dots,m\}},
\end{multline*}
whose $R_n$-basic intervals are:
\begin{align*}
&\Bigl\{\BIgraph{x,y} \,\colon \\
&\hspace*{1.5em}[x,y] \in \bigSBI{P_n \cup Q_n} \setminus \bigl( \BIclass{y_{(n-3)r+n}, x_{(n-2)p+n}} \cup \BIclass{y_{(n-2)r+n}, x_{(n-1)p+n}} \bigr)\\
&\hspace*{17em}\Bigr\} \subset I\setminus\Int\BIgraph{y_{(n-3)r+n}, x_{(n-2)p+n}},\\
& \set{L_i := \xi^{-1}([s_i,s_{i+1}])}{i\in\{0,1,\dots,m-1\}} \subset \BIgraph{y_{(n-3)r+n}, x_{(n-2)p+n}}, \text{ and}\\
&\ \{U_0, U_1,\dots,U_t\} = \set{\varphi_{a,b}\bigl([s_i,s_{i+1}]\bigr)}{i\in\{0,1,\dots,m-1\}} \subset X.
\end{align*}

Next we will derive the Markov graph of $g_n$ with respect to $R_n$
from the Markov graph modulo 1 of $F_n$ with respect to $P_n \cup Q_n,$
which coincides with the Markov graph of $f_n$ with respect to
$\bigemap{P_n \cup Q_n}$ provided that we identify
$\BIgraph{x,y}$ with $\bigemap{\BIclass{x,y}} = \bigemap{[x,y]}$ and this
with $\BIclass{x,y}$ for every $[x,y] \in \bigSBI{P_n \cup Q_n}$
(see the proof of Proposition~\ref{propositionmarkovgraphexamplemontevideu}).
\begin{figure}[ht]
\begin{center}\small
  \tikzstyle{place}=[rectangle,draw=black!60, thick]%
  \tikzstyle{post}=[->,shorten >=1pt,>=stealth,semithick]%
\begin{tikzpicture}[scale=0.9]
\let\bigBIclass\BIgraph
\subfigurefngraphModOne
\node[rotate=90](dotsyx) at (11,6.6) {$\cdots$};
\node[place](yn4rxn3p)   at (11, 5)  {$\bigBIclass{y_{(n-4)r+n},x_{(n-3)p+n}}$};
\filldraw[thick, draw=black!40, fill=black!10, decorate, decoration={zigzag, amplitude=1pt, segment length=2pt}] (8.2,1) rectangle (12.7,3.8);
\node[place](l0)  at (8.65,3.2)  {$L_0$};
\node[place](l1)  at (9.5,3.2)   {$L_1$};
\node[place](l2)  at (10.35,3.2) {$L_2$};
\node(dotsl)      at (11.1,3.2)  {$\cdots$};
\node[place](lm1) at (12,3.2)  {$L_{m-1}$};
\node[place](u0)  at (9.3,1.6)   {$U_0$};
\node[place](u1)  at (10.2,1.6)  {$U_1$};
\node(dotsu)      at (11,1.6)  {$\cdots$};
\node[place](ut)  at (11.7,1.6)  {$U_t$};
\path[->,shorten >=1pt,>=stealth,semithick]
    (yrx2p.south)+(0,-1pt)    edge (dotsyx)
    (dotsyx.west)+(0,-1pt)    edge (yn4rxn3p)
    (yn4rxn3p.south)+(0,-1pt) edge (l0.north)
    (yn4rxn3p.south)+(0,-1pt) edge (l1.north)
    (yn4rxn3p.south)+(0,-1pt) edge (l2.north)
    (yn4rxn3p.south)+(0,-1pt) edge (lm1.north)
    (l0.south)+(0,-1pt)       edge (u0.north)
    (l1.south)+(0,-1pt)       edge (u1.north)
    (l2.south)+(0,-1pt)       edge (u1)
    (lm1.south)+(0,-1pt)      edge (ut.north)
    (u0.south)+(0,-1pt)       edge (yxq)
    (u1.south)+(0,-1pt)       edge (yxq)
    (ut.south)+(0,-1pt)       edge (yxq);
\end{tikzpicture}
\end{center}
\caption{The Markov graph of $g_n.$
The vertices which are intervals used in the Markov graph modulo 1
of $F_n$ must be identified with their images by $\eta \circ \eexp.$
Also, the arrows from the vertices $L_l$ to the vertices $U_j$
inside the grey box circled by a zigzag shape are symbolic
because they cannot be determined precisely.}\label{graph:graphmontevideuOnTheGraph}
\end{figure}
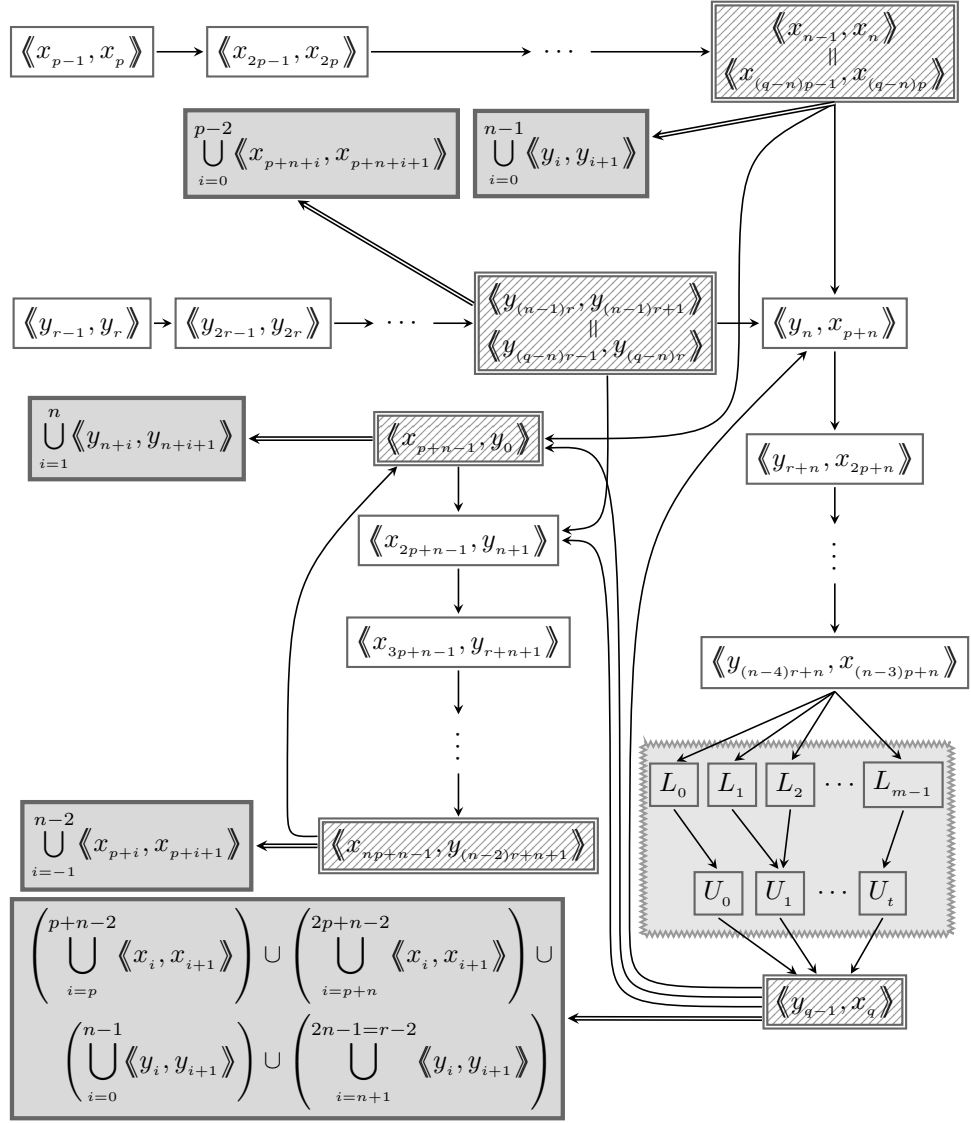
Clearly, the Markov graph of $g_n$ on the intervals
$\BIgraph{x,y}$ such that $[x,y] \in \bigSBI{P_n \cup Q_n}$ and
\[
[x,y] \notin \BIclass{y_{(n-4)r+n}, x_{(n-3)p+n}} \cup \BIclass{y_{(n-3)r+n}, x_{(n-2)p+n}} \cup \BIclass{y_{(n-2)r+n}, x_{(n-1)p+n}}
\]
is isomorphic to the Markov graph modulo 1 of $F_n$ restricted to the
corresponding intervals $\BIclass{x,y}$ (see
Figure~\ref{graph:graphmontevideuOnTheGraph},
Proposition~\ref{propositionmarkovgraphexamplemontevideu}(a) and
Figure~\ref{graphlemmafngraphmontevideu}).
Also, by construction, the interval
$\BIgraph{y_{(n-4)r+n}, x_{(n-3)p+n}}$ $g_n$-covers all the intervals
$L_0,L_1,\dots,L_{m-1} \subset \BIgraph{y_{(n-3)r+n}, x_{(n-2)p+n}}.$
Thus, in the Markov graph of $g_n$ there is an arrow from
$\BIgraph{y_{(n-4)r+n}, x_{(n-3)p+n}}$ to each one of the intervals
$L_0,L_1,\dots,L_{m-1}$ (see Figure~\ref{graph:graphmontevideuOnTheGraph}).
Moreover, every interval $L_i$ $g_n$-covers a unique interval $U_j$
but different intervals $L_i$ can $g_n$-cover the same interval $U_j,$
and every interval $U_j$ $g_n$-covers the same interval
$\BIgraph{y_{q-1}, x_{q}}.$
Hence, the Markov graph of $g_n$ with respect to $R_n$ is the one
shown in Figure~\ref{graph:graphmontevideuOnTheGraph}
where, again, the double arrows arriving to the boxes in grey mean
that there is an arrow arriving to each basic interval in the box
and the arrows between the intervals $L_i$ and $U_j$
are just illustrative.
The part of the Markov graph of $g_n$ with respect to $R_n$
which differs from the Markov graph modulo 1 of $F_n$ with respect to
$P_n \cup Q_n$ is shown inside a grey box with a zigzag border.

As before, by Lemma~\ref{convertingtoexpansiveMM},
the map $g_n$ can be modified
without altering $g_n\evalat{R_n}$ and $g_n(K)$
for every $K \in \SBI[R_n]$ in such a way that
$g_n$ becomes $R_n$-expansive. So, we can use
again Theorem~\ref{theoremTransitivityexpansivemap}
to prove that $g_n$ is transitive.
The Markov graph of $g_n$ tells us that the Markov matrix of $g_n$
with respect to $R_n$ is not a permutation matrix because there are
six basic intervals which $g_n$-cover more than one basic interval.
Moreover, by direct inspection of the Markov graph of $g_n,$
given any two vertices in the graph, there exists a path from the
first to the second one. This means that the transition matrix of the
Markov graph of $g_n$ is non-negative and irreducible.
Thus, $g_n$ is transitive by
Theorem~\ref{theoremTransitivityexpansivemap}.

Concerning the set of periods, it is easy to see that in this example,
\[ \Per(f_n) = \bigcup_{w\in\Per(f_n)} w\cdot\N. \]
So,
\[
\Per(g_n) = \bigcup_{w\in\Per(f_n)} w\cdot\N
\]
exactly as in the proof of Theorem~\ref{theoremfirstexamplebcngraph}
with the difference that, here, the situation is much simpler because
$2 \notin \Per(f_n).$
This proves that the set of periods does not change:
$\Per(g_n) = \Per(f_n).$
Therefore,
$\Per(g_n) \supset \succs{n\nu+1+\tfrac{\nu}{2}}$ which implies that
$\Per(g_n)$ is cofinite and,
by Theorem~\ref{theoremtotallytransitivefromadrr}, $g_n$ is totally transitive.

Now we will estimate $h(g_n)$ by using
Proposition~\ref{propositionmonotoneTopEnt} and
Theorem~\ref{theoremrome} to show that
$\lim_{n\to\infty} h(g_n) = 0.$
We use the rome which corresponds to the one used in the proof of
Proposition~\ref{propositionmarkovgraphexamplemontevideu}:
\begin{multline*}
\widetilde{\textsf{Rom}}_n = \Bigl\{\widetilde{\textsf{r}}_1 = \BIgraph{x_{n-1},x_n},
                   \widetilde{\textsf{r}}_2 = \BIgraph{x_{p+n-1},y_0},\\
                   \widetilde{\textsf{r}}_3 = \BIgraph{x_{np+n-1},y_{(n-2)r +n+1}}, \\
                   \widetilde{\textsf{r}}_4 = \BIgraph{y_{(n-1)r},y_{(n-1)r +1}},
                   \widetilde{\textsf{r}}_5 = \BIgraph{y_{q-1},x_q}\Bigr\}
\end{multline*}
being their elements marked in
Figure~\ref{graph:graphmontevideuOnTheGraph} with a box
with double border and sloping lines background pattern.
Observe that the simple paths from
$\textsf{r}_i$ to $\textsf{r}_j$ in the Markov graph of $f_n,$
computed in the proof of
Proposition~\ref{propositionmarkovgraphexamplemontevideu},
are in one-to-one correspondence with the simple simple paths from
$\widetilde{\textsf{r}}_i$ to $\widetilde{\textsf{r}}_j$
in the Markov graph of $g_n,$ except for the simple paths ending at
$\textsf{r}_5$ and $\widetilde{\textsf{r}}_5.$
Indeed, every simple path in the Markov graph of $f_n$ ending at
$\textsf{r}_5$ is of the form
\begin{multline*}
\textsf{r}_{\ell} \longrightarrow \BIclass{y_n,x_{p+n}} \longrightarrow
\BIclass{y_{r+n},x_{2p+n}} \longrightarrow \dots \longrightarrow
\BIclass{y_{(n-4)r+n}, x_{(n-3)p+n}} \longrightarrow \\
\BIclass{y_{(n-3)r+n}, x_{(n-2)p+n}} \longrightarrow
\BIclass{y_{(n-2)r+n}, x_{(n-1)p+n}} \longrightarrow \textsf{r}_5
\end{multline*}
with $\ell \in \{1,4,5\}.$
However, this path corresponds to the following $m$
paths of the same length in the Markov graph of $g_n$:
\begin{multline*}
\widetilde{\textsf{r}}_{\ell} \longrightarrow \BIclass{y_n,x_{p+n}} \longrightarrow
\BIclass{y_{r+n},x_{2p+n}} \longrightarrow \dots \longrightarrow \\
\BIclass{y_{(n-4)r+n}, x_{(n-3)p+n}} \longrightarrow
L_i \longrightarrow U_j \longrightarrow \widetilde{\textsf{r}}_5
\end{multline*}
with $i = 0,1,\dots, m-1$ and $j \in \{0,1,\dots,t\}$
because $L_0,L_1,\dots,L_{m-1}$
are pairwise different intervals.
So, every non-zero term in the fifth column of the matrix
$M_{\textsf{Rom}_n}(x)$ associated to the Markov graph of $f_n$
(see the proof of Proposition~\ref{propositionmarkovgraphexamplemontevideu}),
which is $x^{-n},$ must be replaced by $mx^{-n}$
in the matrix $M_{\widetilde{\textsf{Rom}}_n}(x)$ associated to the
Markov  graph of $g_n,$ and these are the only changes
when comparing $M_{\textsf{Rom}_n}(x)$ with
$M_{\widetilde{\textsf{Rom}}_n}(x)$.
Therefore,
\[
M_{\widetilde{\textsf{Rom}}_n}(x) = \begin{pmatrix}
           0    & x^{-1} &      0     & a_{14}(x) & mx^{-n} \\
           0    &     0  & x^{-(n-1)} & a_{24}(x) &   0 \\
      a_{31}(x) & x^{-1} &      0     &      0    &   0 \\
      a_{41}(x) &     0  & x^{-(n-1)} &      0    & mx^{-n} \\
      a_{51}(x) & x^{-1} & x^{-(n-1)} & a_{54}(x) & mx^{-n} \\
\end{pmatrix}
\]
where $a_{14}(x), a_{24}(x), a_{31}(x), a_{41}(x), a_{51}(x)$ and
$a_{54}(x)$ are the same as in the proof of
Proposition~\ref{propositionmarkovgraphexamplemontevideu}.

By Theorem~\ref{theoremrome}, the characteristic polynomial
(ignoring the sign) of the transition matrix of the Markov  graph of
$g_n$ is
\[
  \pm x^{2(q-1)+m+t}\det(M_{\widetilde{\textsf{Rom}}_n}(x) - \mathbf{I}_{5}) =
   x^{t+m-2}\ \frac{\widetilde{T}_n(x)}{\bigl(x^{2n-1}-1\bigr)\bigl(x^{2n+1}-1\bigl)}
\]
where
\begin{align*}
   \widetilde{T}_n(x) &= \widetilde{\kappa}_2(x) (x^{2q} + 1) +  \widetilde{\kappa}_1(x) x^{q}  - (m+1)\widetilde{\kappa}_0(x),\\
   \widetilde{\kappa}_2(x) &= x^{4n} - (m+1)\left[ x^{3n} + x^{2n} + x^{n}\right] - x^{2n+1} - (m+2)x^{2n-1} + 1,\\
   \widetilde{\kappa}_1(x) &= (m-1)\left[x^{4n} + 1\right] + 2(m+1)\left[x^{3n} + x^{2n} + x^{n}\right] +\\
                           &\hspace*{18em}2\left[x^{2n+1} + x^{2n-1}\right],\\
\intertext{and}
   \widetilde{\kappa}_0(x) &= x^{4n} -2x^{2n-1} + 1.
\end{align*}
Then, as in the proof of Theorem~\ref{theoremexamplemontevideucircle},
we use the following easy bounds for
$\widetilde{\kappa}_2(x),\ \widetilde{\kappa}_1(x)$ and
$\widetilde{\kappa}_0(x),$ which are valid for $x \ge 1$:
\begin{align*}
 \widetilde{\kappa}_2(x) &> x^{4n} - (4m+6)x^{3n} = x^{3n} (x^n - (4m+6)),\\
 \widetilde{\kappa}_1(x) &> (7m+9)x^{n},\text{ and}\\
 \widetilde{\kappa}_0(x) &= x^{4n} - 2x^{2n-1} + 1 \le x^{4n} - 1 < x^{4n}.
\end{align*}
Hence, now for $x > \sqrt[n]{4m+6} > 1,$
\begin{multline*}
 \widetilde{T}_n(x) > x^{3n} (x^n - (4m+6))\bigl(x^{2q} + 1\bigr) + (7m+9)x^{n}x^{q} - (m+1)x^{4n} = \\
          x^{3n} (x^n - (4m+6))\bigl(x^{2q} + 1\bigr) + (m+1)x^{4n} \Bigl(\tfrac{7m+9}{m+1}x^{n(2n-3)} - 1\Bigr) > \\
          x^{3n} (x^n - (4m+6))\bigl(x^{2q} + 1\bigr) > 0.
\end{multline*}
Therefore, $\tilde{\rho}_n,$ the largest root of $\widetilde{T}_n(x)$
verifies $\rho_n \le \sqrt[n]{4m+6}$ and hence,
\[ 0 \le \lim_{n\to\infty} h(g_n) = \lim_{n\to\infty} \log \tilde{\rho}_n \le \lim_{n\to\infty} \log \sqrt[n]{4m+6} = 0 \]
because $m$ is a fixed number that depends on the topology of the graph and is independent on $n.$
\end{proof}
\subsection{The dream example} The last example that we construct
consists of maps without low periods:

\begin{CustomNumberedExample}[\ref{exampleBCNintroduction}]
For every positive integer $n \ge 3$ there exists $f_n,$
a totally transitive continuous circle map of degree one having a
lifting $F_n \in \dol$ such that
$\Rot(F_n) = \left[\tfrac{1}{2n-1}, \tfrac{2}{2n-1}\right],$
$\Per(f_n) = \succs{n}$ and
$\lim_{n\to\infty} h(f_n) = 0.$
Hence, $\bc(f_n) = \sbc(f_n) = n$ and $\lim_{n\to\infty} \bc(f_n) = \infty$.
\smallskip

Furthermore, given any graph $G$ with a circuit, the sequence of maps
$\{f_n\}_{n\ge 3}$ can be extended to a sequence of
continuous totally transitive maps $\map{g_n}{G}$
such that $\Per(g_n) = \Per(f_n)$ and
$\lim_{n\to\infty} h(g_n) = 0.$
\end{CustomNumberedExample}

As in the previous subsections,
Example~\ref{exampleBCNintroduction} will be split
into a theorem that shows the existence of the circle maps $f_n$
by constructing them along the lines of
Subsection~\ref{Examples:PhilandIntro-HowTo},
and a theorem that extends these maps to a generic graph that is not
a tree. The proof of these results will, in turn, use a proposition
that computes the Markov graph  modulo 1 of the liftings $F_n.$

\begin{theorem}\label{theoremexampleBCNintroduction}
Let $n \in \N,\ n\ge 3,$ and let
\begin{align*}
Q_n &= \{\dots x_{-1}, x_0, x_1, x_2, \dots, x_{2n-2}, x_{2n-1}, x_{2n}, \dots\} \subset \R,\andq \\
P_n &= \{\dots y_{-1}, y_0, y_1, y_2, \dots, y_{2n-2}, y_{2n-1}, y_{2n}, \dots\} \subset \R
\end{align*}
be infinite sets such that the points of $P_n$ and $Q_n$
are intertwined  so that
\begin{equation}\label{intertiwining}\arraycolsep=2pt
 \begin{array}{lccccccl}
 0=x_0 < x_1 < \cdots < x_{n-1} < y_0 < & x_{n}    & < & y_1      & < & y_2      & < &\\
                                        & x_{n+1}  & < & y_{3}    & < & y_{4}    & < &\\
                                        & \vdots   &   & \vdots   &   & \vdots   &   &\\
                                        & x_{2n-2} & < & y_{2n-3} & < & y_{2n-2} & < & x_{2n-1} = 1
\end{array}
\end{equation}
and $x_{i + (2n-1)\ell} = x_i + \ell$ and $y_{i + (2n-1)\ell} = y_i + \ell$
for every $i,\ell \in \Z.$

We define a lifting $F_n \in \dol$ such that, for every $i \in \Z,$
$F_n(x_i) = x_{i+1}$ and $F_n(y_i) = y_{i+2},$
and $F_n$ is expansive between consecutive points of $P_n \cup\; Q_n.$
Then, $Q_n$ and $P_n$ are twist lifted periodic orbits of $F_n$
both of period $2n-1$ such that
$Q_n$ has rotation number $\tfrac{1}{2n-1}$ and
$P_n$ has rotation number $\tfrac{2}{2n-1}.$
Moreover, $F_n$ has
$\Rot(F_n) = \left[\tfrac{1}{2n-1}, \tfrac{2}{2n-1}\right]$
as rotation interval.

Let $\map{f_n}{\SI}$ be the continuous map which has $F_n$ as a lifting.
Then,
$f_n$ is totally transitive,
$\Per(f_n) = \succs{n}$ and
$\lim_{n\to\infty} h(f_n) = 0.$
Hence, $\bc(f_n) = \sbc(f_n) = n$ and $\lim_{n\to\infty} \bc(f_n) = \infty$.
\end{theorem}

\begin{remark}\label{Farey-setofperiods-exampleBCNintroduction}
Our choice of the rotation interval in this example was influenced by
the Farey sequence of order $2n-1$
(which is the ordered sequence of rationals $\tfrac{p}{q}$ such that
$0 \le p \le q \le 2n-1,\  (p, q) = 1$).
It follows that two elements $\tfrac{p}{q} < \tfrac{r}{s}$
in a Farey sequence are consecutive (called \emph{Farey neighbours})
if and only if $qr-ps = 1.$
Hence, the endpoints of the rotation interval of
Example~\ref{exampleBCNintroduction} and
Theorem~\ref{theoremexampleBCNintroduction} belong to the
Farey sequence of order $2n-1$ and the elements of this sequence
between them are
\[
 \tfrac{1}{2n-1} < \tfrac{1}{2n-2} < \tfrac{1}{2n-3} < \tfrac{1}{2n-4} < \dots < \tfrac{1}{n} < \tfrac{2}{2n-1}.
\]
This tells us that
$\Per(f_n) \cap \{1,2,\dots, 2n-2\} = \{n, n+1,\dots , 2n-2\}$
which was the kind of set of periods we were looking for.
\end{remark}

\begin{theorem}\label{theoremexampleBCNintroductiongraph}
Let $G$ be a graph with a circuit. Then, the sequence of maps
$\{f_n\}_{n=5}^\infty$ from
Theorem~\ref{theoremexampleBCNintroduction}
can be extended to a sequence
of continuous totally transitive self maps of $G$,
$\{g_n\}_{n=5}^\infty,$
such that
$\Per(g_n) = \Per(f_n)$ and $\lim_{n\to\infty} h(g_n) = 0.$
\end{theorem}

Before proving Theorem~\ref{theoremexampleBCNintroduction} we will
study the Markov graph modulo 1 of the liftings $F_n.$

\long\def\subfigurefngraphModOne{%
\filldraw[ultra thick, draw=black!60, fill=black!15] (0.2,7.45) rectangle  (2.8,1.95);
\node[place](y1y2)      at (9.5,7)   {$\BIclass{y_1, y_2}$};
\node[place](y7y8)      at (9.5,2.8) {$\BIclass{y_7,y_8}$};
\node[rotate=90](dots)  at (9.5,1.4) {$\cdots$};
\node[place](y2n3y2n2)  at (9.5,0)   {$\BIclass{y_{2n-3}, y_{2n-2}}$};
\node[place](xny1)      at (6.8,4.4)  {$\BIclass{x_n, y_1}$};
\node[place](xn1y3)     at (6.8,3.2)  {$\BIclass{x_{n+1},y_{3}}$};
\node[rotate=90](dots2) at (6.8,2)    {$\cdots$};
\node[place, double, pattern color=black!30, pattern=north east lines](x2n2y2n3) at (6.8,0.8) {$\BIclass{x_{2n-2},y_{2n-3}}$};
\node[place](y0xn)      at (4.4,6)    {$\BIclass{y_0,x_n}$};
\node[place](y2xn1)     at (4.4,4.5)  {$\BIclass{y_2,x_{n+1}}$};
\node[rotate=90](dots4) at (4.4,3.15)    {$\cdots$};
\node[place, double, pattern color=black!30, pattern=north east lines](y2n2x2n1) at (4.4,1.8) {$\BIclass{y_{2n-2},x_{2n-1}}$};
\node[place](x0x1)      at (1.5,1)    {$\BIclass{x_0,x_1}$};
\node[place](x1x2)      at (1.5,2.5)  {$\BIclass{x_1,x_2}$};
\node[rotate=90](dots3) at (1.5,4)    {$\cdots $};
\node[place](xn2xn1)    at (1.5,5.5)  {$\BIclass{x_{n-2}, x_{n-1}}$};
\node[place](xn1y0)     at (1.5,7)    {$\BIclass{x_{n-1},y_0}$};
\draw[-, semithick, double] (x2n2y2n3.south)+(0pt,-1pt) .. controls (6.8,0.2) .. (6.5,0.2) -- (4,0.2);
\draw[post, double] (y2n2x2n1.south)+(0pt,-1pt) -- (4.4,0.8) .. controls (4.4,0.2) .. (4,0.2) -- (1,0.2) .. controls (0.5, 0.2) .. (0.5,0.4) -- (0.5,1.945);
\path[post]
    (xn2xn1.east)+(1pt,0)    edge ([yshift=1pt]y0xn.west)
    (y7y8.south)+(0,-1pt)    edge (dots.east)
    (dots.west)+(0,-1pt)     edge (y2n3y2n2.north)
    (y0xn.south)+(0, -1pt)   edge (y2xn1.north)
    (y2xn1.south)+(0,-1pt)   edge (dots4.east)
    (dots4.west)+(0,-1pt)    edge (y2n2x2n1.north)
    (x0x1.north)+(0,1pt)     edge (x1x2.south)
    (x1x2.north)+(0,1pt)     edge (dots3.west)
    (dots3.east)+(0,1pt)     edge (xn2xn1)
    (xn2xn1.north)+(0,1pt)   edge (xn1y0.south)
    (xny1.south)+(0,-1pt)    edge (xn1y3.north)
    (xn1y3.south)+(0,-1pt)   edge (dots2.east)
    (dots2.west)+(0,-1pt)    edge (x2n2y2n3.north)
    (xn1y0.east)+(1pt,4pt)   edge ([yshift=4pt]y1y2.west)
    (x2n2y2n3.west)+(-1pt,0) edge (x0x1.east);
\draw[post] (7.8, 3.573) .. controls (7.675,3.7) .. (7.675, 4) -- (7.675, 5.5) .. controls (7.625, 6) .. (7.125, 6) -- (y0xn.east);
\draw[post] (9.3,0.36) .. controls (8.15,3.5) .. ([xshift=2pt]xny1.south);
\draw[post] ([xshift=4pt,yshift=1pt]y2n2x2n1.north) .. controls ([xshift=10pt,yshift=4em]y2n2x2n1.north) .. (xny1.west);
\draw[post] (xn1y0.east)+(1pt,-1pt) .. controls (6,6.9) .. (xny1.north);
\draw[post] (y2n2x2n1.west)+(-1pt,0) .. controls (3,1.8) .. (3,3) .. controls (3,5.2) .. ([yshift=-1pt]y0xn.west);
} 
\begin{proposition}[$\calB(P_n \cup Q_n)$ and the $F_n$-Markov graph modulo 1]\label{propositionmarkovgraphexampleBCNintroduction}
In the assumptions of Theorem~\ref{theoremexampleBCNintroduction} we have:
\begin{enumerate}[(a)]
\item The Markov graph modulo 1 of $F_n$ is:
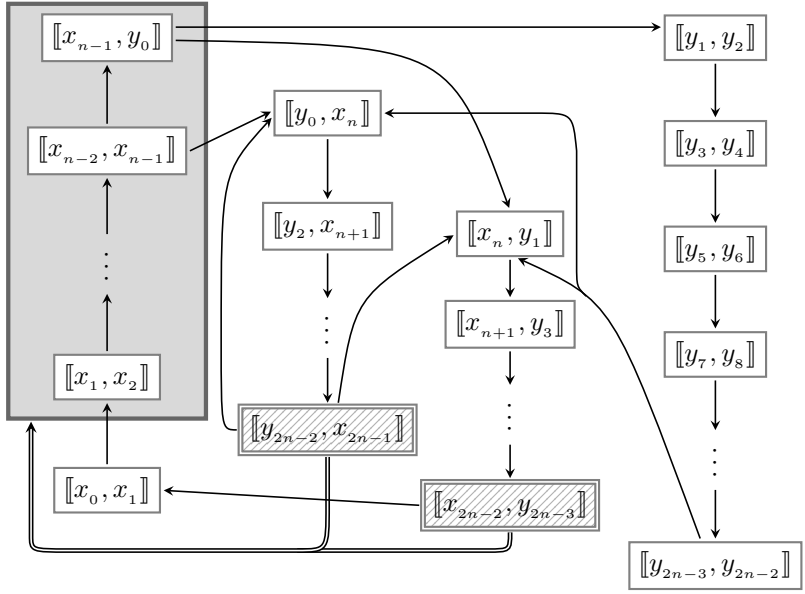
\begin{center}\small
\tikzstyle{place}=[rectangle,draw=black!50,fill=black!0,thick]
\tikzstyle{post}=[->,shorten >=1pt,>=stealth,semithick]
\medskip
\hspace*{-1.5em}\begin{tikzpicture}
\subfigurefngraphModOne
\node[place](y3y4)      at (9.5,5.6)   {$\BIclass{y_3, y_4}$};
\node[place](y5y6)      at (9.5,4.2) {$\BIclass{y_5,y_6}$};
\path[post]
    (y1y2.south)+(0,-1pt) edge (y3y4.north)
    (y3y4.south)+(0,-1pt) edge (y5y6.north)
    (y5y6.south)+(0,-1pt) edge (y7y8.north);
\end{tikzpicture}
\end{center}
where the double arrows arriving to the the box in grey mean that
there is an arrow arriving to each interval in the box.

\item $h(f_n) = \log \rho_n,$ where $\rho_n > 1$ is the largest root
of the polynomial
\[
T_n(x) = \bigl(x^{4n-2}-1\bigr)(x-1) -2x^n\bigl(x^{2n-1} - 1\bigr).
\]
\end{enumerate}
\end{proposition}

\begin{proof}
The proof that the Markov graph modulo 1 of $F_n$ is the one depicted
in (a) follows easily from the computation of the images of the basic
intervals.
To do this recall that, for every $i,\ell \in \Z,$
$F_n(x_i) = x_{i+1},\ F_n(y_i) = y_{i+2},$
$x_{i + (2n-1)\ell} = x_i + \ell$ and $y_{i + (2n-1)\ell} = y_i + \ell.$
Moreover, $F_n$ is strictly monotone (in fact affine)
between consecutive points of $P_n \cup\; Q_n.$
Then, by using \eqref{intertiwining} to determine the basic intervals
we get
\begin{enumerate}[(i)]
\item $F_n([x_i, x_{i+1}]) = [x_{i+1},x_{i+2}]$ for $i \in \{0,1,\dots,n-3\}.$
\item $F_n([x_{n-2},x_{n-1}]) = [x_{n-1},x_{n}] = [x_{n-1}, y_0]\cup [y_0,x_n].$
\item $F_n([x_{n-1}, y_0]) = [x_n, y_2] = [x_n, y_1]\cup [y_1,y_2].$
\item $F_n([y_{2i}, x_{n+i}]) = [y_{2(i+1)}, x_{n+i+1}]$ for $i \in \{0,1,\dots,n-2\}.$
\item $\begin{multlined}[t][0.9\textwidth]
         F_n([y_{2n-2},x_{2n-1}]) = [x_{2n}, y_{2n}] = [x_1, y_1] + 1 =\\
             \left(\left(\bigcup\limits_{i=1}^{n-2} [x_i, x_{i+1}]\right) \cup
             [x_{n-1}, y_0] \cup [y_0, x_n] \cup [x_n, y_1]\right) + 1 \in \\
             \left(\bigcup\limits_{i=1}^{n-2} \BIclass{x_i, x_{i+1}}\right) \cup
             \BIclass{x_{n-1}, y_0}\cup
             \BIclass{y_0, x_n} \cup
             \BIclass{x_n, y_1}.
\end{multlined}$\newline
Moreover, $\BIclass{y_{2n-2},x_{2n-1}}$ is the only class of basic intervals
where $F_n$ is decreasing.
\item $F_n([x_{n+i}, y_{2i+1}]) = [x_{n+(i+1)}, y_{2(i+1)+1}]$ for $i \in \{0,1,\dots,n-3\}.$
\item $\begin{multlined}[t][0.85\textwidth]
          F_n([x_{2n-2},y_{2n-3}]) = [x_{2n-1}, y_{2n-1}] = [x_0, y_0] + 1 = \\
             \left(\bigcup\limits_{i=0}^{n-2} \bigl([x_i, x_{i+1}] + 1\bigr)\right) \cup
             \bigl([x_{n-1}, y_0] + 1\bigr) \in \\
             \left(\bigcup\limits_{i=0}^{n-2} \BIclass{x_i, x_{i+1}}\right) \cup \BIclass{x_{n-1}, y_0}.
\end{multlined}$
\item $F_n([y_{2i+1},y_{2(i+1)}]) = [y_{2(i+1)+1},y_{2(i+2)}]$ for $i \in \{0,1,\dots,n-3\}.$
\item $\begin{multlined}[t][0.85\textwidth]
          F_n([y_{2n-3},y_{2n-2}]) = [y_{2n-1},y_{2n}] = [y_0,y_1] + 1 =\\
                \bigl([y_0, x_n]+1\bigr) \cup \bigl([x_n,y_1]+1\bigr) \in \BIclass{y_0, x_n} \cup \BIclass{x_n,y_1}.
\end{multlined}$
\end{enumerate}
Then, (a) (the Markov graph modulo 1 of $F_n$) is a direct translation
of the above list of images to the language of combinatorial graphs.

Now we prove (b). In this case, clearly,
\[
\textsf{Rom}_n = \{\textsf{r}_1 = \BIclass{x_{2n-2},y_{2n-3}}, \textsf{r}_2 = \BIclass{y_{2n-2},x_{2n-1}}\}
\]
as a rome of two elements
(being their elements marked in (a) with a box with double
border and sloping lines background pattern).
Then, the matrix $M_{\textsf{Rom}_n}(x)$ is:
{\small\[
 \begin{pmatrix*}[r]
   \sum\limits_{i=0}^{n-1} x^{-(n+i)} + \sum\limits_{i=0}^{n-1} x^{-(2n-1+i)}              & \sum\limits_{i=1}^{n-1}x^{-(n+i)} + \sum\limits_{i=0}^{n-1} x^{-(2n+i)}\\[3ex]
   x^{-(n-1)} + \sum\limits_{i=0}^{n-2} x^{-(n+i)} + \sum\limits_{i=0}^{n-2} x^{-(2n-1+i)} & x^{-n} + \sum\limits_{i=1}^{n-2}x^{-(n+i)} + \sum\limits_{i=0}^{n-2} x^{-(2n+i)}
\end{pmatrix*}
\]}\[
\hspace*{4em}= \begin{pmatrix*}[r]
 x^{-(2n-1)}               + \alpha(x) & - x^{-n} + x^{-(3n-1)} + \alpha(x)\\
 x^{-(n-1)}  - x^{-(3n-2)} + \alpha(x) & - x^{-(2n-1)}          + \alpha(x)
\end{pmatrix*}
\]
with
\[
 \alpha(x) := \sum\limits_{i=0}^{2n-2} x^{-(n+i)}.
\]
Then, by Theorem~\ref{theoremrome}, the characteristic polynomial of
the Markov matrix of $F_n$ is
\begin{multline*}
 (-1)^{4n-4} x^{4n-2} \det\bigl(M_{\textsf{Rom}_n}(x) - \mathbf{I}_{2}\bigr) =\\[2ex]
 \frac{\bigl(x^{4n-2}-1\bigr)(x-1) -2x^n\bigl(x^{2n-1} - 1\bigr)}{x-1}.
\end{multline*}
Therefore, (b) holds.
\end{proof}

\begin{proof}[Proof of Theorem~\ref{theoremexampleBCNintroduction}]
In a similar way to the proof of Theorem~\ref{theoremfirstexamplebcncircle}
we see that $Q_n$ and $P_n$ are twist lifted periodic orbits of $F_n$
both of period $2n-1$ such that
$Q_n$ has rotation number $\tfrac{1}{2n-1}$ and
$P_n$ has rotation number $\tfrac{2}{2n-1}.$

The proof that $\Rot(F_n) = \left[\tfrac{1}{2n-1}, \tfrac{2}{2n-1}\right]$
also follows as in Theorem~\ref{theoremfirstexamplebcncircle}
except that in this example the upper and lower maps are as follows:
There exists a unique
$u^n_l \in \bigl(x_{2n-2},y_{2n-3}\bigr)$ such that
\[ F_n\bigl(u^n_l\bigr) = x_1 + 1 = F_n\bigl(x_0\bigr) + 1 = F_n(1) \]
(see (vii) from the proof of Proposition~\ref{propositionmarkovgraphexampleBCNintroduction}).
Then,
\begin{multline*}
 (F_n)_l(x) = \inf\set{F_n(y)}{y \ge x} = \\
 \begin{cases}
    F_n(x)& \text{for $x \in \bigl[0, u^n_l\bigr]$,}\\
    x_1+1 & \text{for $x \in \bigl[u^n_l,1\bigr]$,}\\
    (F_n)_l\bigl(x - \floor{x}\bigr) + \floor{x} & \text{if $x \notin [0,1].$}
 \end{cases}
\end{multline*}
Also, there exists a unique $u^n_u \in \bigl(x_{n-1},y_0)$ such that
$F_n\bigl(u^n_u\bigr) = y_{1} = F_n\bigl(y_{2n-2}) - 1$
(see (iii) and (ix) from the proof of Proposition~\ref{propositionmarkovgraphexampleBCNintroduction}).
Then,
\begin{multline*}
 (F_n)_u(x) = \sup\set{F_n(y)}{y \le x} = \\
 \begin{cases}
    y_1   & \text{for $x \in \bigl[0, u^n_u\bigr]$,}\\
    F_n(x)   & \text{for $x \in \bigl[u^n_u,y_{2n-2}\bigr]$,}\\
    y_1+1 & \text{for $x \in \bigl[y_{2n-2}, 1\bigr]$,}\\
    (F_n)_u\bigl(x - \floor{x}\bigr) + \floor{x} & \text{if $x \notin [0,1].$}
 \end{cases}
\end{multline*}
Then, as in the previous two examples, we have
$P_n \cap [0,1] \subset [u^n_u,y_{2n-2}],$
$(F_n)_u\evalat{P_n} = F_n\evalat{P_n},$
$\rho\bigl((F_n)_u\bigr) = \rho_{F_n}(P_n) = \tfrac{2}{2n-1},$
$Q_n \cap [0,1] \subset [0, u^n_l],$
$(F_n)_l\evalat{Q_n} = F_n\evalat{Q_n},$
and
$\rho\bigl((F_n)_l\bigr) = \rho_{F_n}(Q_n) = \tfrac{1}{2n-1}.$
Consequently, from Theorem~\ref{theoremrotationintwathermap},
$\Rot(F_n) = \left[\tfrac{1}{2n-1}, \tfrac{2}{2n-1}\right].$

Now we prove that $\Per(f_n) = \succs{n}.$
Observe that
\[
M\Bigl(\tfrac{1}{2n-1}, \tfrac{2}{2n-1}\Bigr) \supset \succs{2n}
\]
because $\len(\Rot(F_n))=\tfrac{1}{2n-1},$ and
\[
 \{2n-1\} \subset
 Q_{F_n}\Bigl(\tfrac{1}{2n-1}\Bigr) \cup  Q_{F_n}\Bigl(\tfrac{2}{2n-1}\Bigr)
 \subset (2n-1)\N \subset \succs{2n-1}.
\]
On the other hand, in view of Remark~\ref{Farey-setofperiods-exampleBCNintroduction},
\[
 M\Bigl(\tfrac{1}{2n-1}, \tfrac{2}{2n-1}\Bigr) \cap \{1,2,\dots, 2n-1\} =
 \{n, n+1,\dots , 2n-2\}.
\]
Consequently, by Theorem~\ref{theoremMisiurewicz},
\begin{multline*}
  \Per(f_{n}) =
  Q_{F_n}\Bigl(\tfrac{1}{2n-1}\Bigr) \cup M\Bigl(\tfrac{1}{2n-1}, \tfrac{2}{2n-1}\Bigr) \cup Q_{F_n}\Bigl(\tfrac{2}{2n-1}\Bigr) =\\
  \left(Q_{F_n}\Bigl(\tfrac{1}{2n-1}\Bigr) \cup Q_{F_n}\Bigl(\tfrac{2}{2n-1}\Bigr)\right)
     \cup \left( M\Bigl(\tfrac{1}{2n-1}, \tfrac{2}{2n-1}\Bigr) \cap \{1,2,\dots, 2n-1\}\right)
     \cup\\ \hspace*{17em}\left( M\Bigl(\tfrac{1}{2n-1}, \tfrac{2}{2n-1}\Bigr) \cap \succs{2n}\right) \supset\\
  \{2n-1\} \cup  \{n, n+1,\dots , 2n-2\} \cup \succs{2n}  = \succs{n} \supset \\
  Q_{F_n}\Bigl(\tfrac{1}{2n-1}\Bigr) \cup M\Bigl(\tfrac{1}{2n-1}, \tfrac{2}{2n-1}\Bigr) \cup Q_{F_n}\Bigl(\tfrac{2}{2n-1}\Bigr).
\end{multline*}
So, clearly,
$\bc(f_n) = \sbc(f_n) = n$ and $\lim_{n\to\infty} \bc(f_n) = \infty$.

Next we show that $f_n$ is totally transitive.
As in the previous examples,
$P_n \cup\; Q_n$ is a short Markov partition with respect to $F_n.$
Then, $f_n$ is an expansive Markov map with respect to the Markov
partition $\bigemap{P_n \cup Q_n},$
and the transition matrix of the Markov graph of $f_n$
with respect to $\bigemap{P_n \cup Q_n}$
coincides with the transition matrix of the Markov graph modulo 1
of $F_n$ with respect to $P_n \cup Q_n.$ Moreover,
this transition matrix is non-negative and irreducible,
and from
Proposition~\ref{propositionmarkovgraphexampleBCNintroduction}(a)
it follows that there exists a vertex in the
Markov graph modulo 1 of $F_n$ which is the beginning of more than
one arrow (for instance $\BIclass{x_{2n-2},y_{2n-3}}$).
That is, the transition matrix of the Markov graph of $f_n$
with respect to $\bigemap{P_n \cup Q_n}$ is not
a permutation matrix.
Then, $f_n$ is transitive by Theorem~\ref{theoremTransitivityexpansivemap}
and $\Per(f_n)$ is cofinite because $\Per(f_n) = \succs{n}.$
Hence, $f_n$ is totally transitive by Theorem~\ref{theoremtotallytransitivefromadrr}.

To end the proof of the theorem we need to show that
$\lim_{n\to\infty} h(f_n) = 0.$
With the notation of Proposition~\ref{propositionmarkovgraphexampleBCNintroduction}(b)
we have $\rho_n > 1$ and
\[
  0 = T_n(\rho_n)  = \bigl(\rho_n^{4n-2}-1\bigr)(\rho_n-1) -2\rho_n^n\bigl(\rho_n^{2n-1} - 1\bigr),
\]
which is equivalent to
\[
 \rho_n^{4n-2}(\rho_n-1) =
 2\rho_n^n\bigl(\rho_n^{2n-1} - 1\bigr) + (\rho_n-1).
\]
So, for $x > 1,$ we consider the equation
\[
 x^{4n-2}(x-1) = 2x^n\bigl(x^{2n-1} - 1\bigr) + (x-1)
 \Longleftrightarrow
 x^{n-1} = 2\frac{x^{2n-1} - 1}{x^{2n-1}(x-1)} + \frac{1}{x^{3n-1}}.
\]
Observe that
\[
 2\frac{x^{2n-1} - 1}{x^{2n-1}(x-1)} + \frac{1}{x^{3n-1}} = \frac{2\tfrac{x^{2n-1}-1}{x^{2n-1}} + \tfrac{x-1}{x^{3n-1}}}{x-1} < \frac{3}{x-1}.
\]
Now we proceed as in the proof of
Theorem~\ref{theoremfirstexamplebcncircle}
(see also the figure in page~\pageref{fig:arrels}):
\begin{enumerate}[(i)]\label{zerosarguments}
\item The map $x \mapsto \tfrac{3}{x-1}$ is strictly
      decreasing on the interval $(1,+\infty)$,
      $\lim_{x\to1^+} \tfrac{3}{x-1} = +\infty$ and
      $\lim_{x\to\infty} \tfrac{3}{x-1} = 0$
\item For every $n \ge 3$ and every $x \ge 1$,
      the map $x \mapsto x^{n-1}$ is strictly increasing and
      $x^{n-1}\evalat{x=1} = 1.$
\item For every $n,m \in \N,$  $3 \le n < m$ and $x > 1,$
      $x^{n-1} < x^{m-1}.$
\end{enumerate}
Then, for each $n \ge 3,$ there exists a unique real
number $\gamma_n > 1$ such that
$\gamma_n^{n-1} = \tfrac{3}{\gamma_n -1}$ and
$x^{n-1} > \tfrac{3}{x-1}$ for every $x > \gamma_n,$
the sequence $\{\gamma_n\}_n$ is strictly decreasing and
$\lim_{n\to\infty} \gamma_n = 1.$
Hence,
\[
 \frac{x^{4n-2}(x-1)}{2x^n\bigl(x^{2n-1} - 1\bigr) + (x-1)} =
 \frac{x^{n-1}}{2\tfrac{x^{2n-1} - 1}{x^{2n-1}(x-1)} + \tfrac{1}{x^{3n-1}}} >
 \frac{x^{n-1}}{\tfrac{3}{x-1}} > 1
\]
for every $x > \gamma_n.$
Consequently, $T_n(x) > 0$ for every $x > \gamma_n$ and, hence,
$\rho_n \le \gamma_n$ for every $n \ge 3,$ and
$
 \lim_{n\to\infty} \log\rho_n  \le \lim_{n\to\infty} \log\gamma_n = 0.
$
\end{proof}

\begin{proof}[Proof of Theorem~\ref{theoremexampleBCNintroductiongraph}]
As in the proof of proof of Theorem~\ref{theoremfirstexamplebcngraph}
we may assume that $G \ne \SI$ since otherwise
Theorem~\ref{theoremexamplemontevideucircle} already gives the desired
sequence of maps.

The proof in the case $G \ne \SI$ goes along the lines of the proof of
Theorems~\ref{theoremfirstexamplebcngraph}
and~\ref{theoremexamplemontevideugraph},
and most of the details will be omitted.

We fix a circuit $C$ of $G$ and an interval $I \subset C$ such that
$I \cap V(G) = \emptyset.$ Also, we choose a homeomorphism
$\map{\eta}{\SI}[C]$ such that
\begin{align*}
 & C\mkern-2mu\setminus\mkern-3mu\Int(I) = \BIgraph{y_5, y_6},\text{ and}\\
 & I \supset \bigeta{\bigemap{P_n \cup Q_n}} \cup \mspace{-83mu}
     \bigcup_{\mspace{120mu}\begin{subarray}{l}
                 [x,y] \in \SBI[P_n \cup Q_n]\\
                 [x,y] \notin \BIclass{y_5, y_6}
              \end{subarray}}                  \BIgraph{x,y}.
\end{align*}
where, as in Theorem~\ref{theoremexamplemontevideugraph},
$\BIgraph{x,y}$ denotes the convex hull (in $C$) of the set
$\{\etaemap{x}, \etaemap{y}\}.$
Observe that
$\BIgraph{y_5, y_6}$ plays the role of $\widetilde{I}_2$
in the proof of Theorem~\ref{theoremfirstexamplebcngraph}
(see Figure~\ref{figuregraphGandmapgnfirstexampleBCN}) and consequently,
by Proposition~\ref{propositionmarkovgraphexampleBCNintroduction}(a),
$\BIgraph{y_7, y_8}$ plays the role of the interval $\widetilde{I}_3$ while
$\BIgraph{y_{2j+1}, y_{2j+2}}$ play the role of $\widetilde{I}_j$ for $j = 0,1.$
Note that all the intervals are well defined since $n \ge 5$
and they are pairwise disjoint because of the ordering of points defined
in Theorem~\ref{theoremexampleBCNintroduction}.

We set
$X := G \setminus \Int(I) \supset \BIgraph{y_5, y_6},$
and
$V(X) = V(G) \cup \{a,b\}$
with
$a := \etaemap{y_5}$ and $b := \etaemap{y_6}.$
Then, as before, we use Lemma~\ref{lemmaPhiPsi} for the subgraph $X$
(see Figure~\ref{mapsfromarr}).
Let $m=m(X, a, b) \ge 5$ be odd,
consider the partition $0=s_0<s_1<\cdots <s_m=1,$ and
let the maps
$\map{\varphi_{a, b}}{[0,1]}[X]$ and
$\map{\psi_{a, b}}{X}[{[0,1]}]$
be as in Lemma~\ref{lemmaPhiPsi}.
Also, we define two arbitrary but fixed homeomorphisms
$\map{\zeta}{[0,1]}[{\BIgraph{y_7, y_8}}]$  and
$\map{\xi}{\BIgraph{y_3, y_4}}[{[0,1]}]$
such that
\[
  \zeta(0) = \etaemap{y_7},\ \zeta(1) = \etaemap{y_8},
  \ \xi\bigl(\etaemap{y_3}\bigr) = 0 \text{ and }
  \xi\bigl(\etaemap{y_4}\bigr) = 1
\]
(see Figure~\ref{figuregraphGandmapgnfirstexampleBCN} for an analogous situation).

With all these definitions, for $n \ge 5$ we set
\[
  g_n(x) := \begin{cases}
      \varphi_{a,b} (\xi(x)) & \text{if $x \in \BIgraph{y_3, y_4}$;}\\
      \zeta \bigl(\psi_{a,b}(x)\bigr)  & \text{if $x\in X$;}\\
      (\eta \circ f_n \circ \eta^{-1})(x) & \text{if $x \in I\setminus\Int\BIgraph{y_3, y_4},$}
\end{cases}
\]
and, as in the proof of Theorem~\ref{theoremfirstexamplebcngraph}
we can easily show that $g_n$ is a Markov map
with respect to the partition
\begin{multline*}
R_n = \bigeta{\bigemap{Q_n \cup P_n}} \cup
      \set{\xi^{-1}(s_i)}{i \in \{0,1,\dots,m\}} \cup \\
      \set{\varphi_{a,b}(s_i)}{i \in \{0,1,\dots,m\}},
\end{multline*}
whose $R_n$-basic intervals are:
\begin{align*}
&\Bigl\{\BIgraph{x,y} \,\colon \\
&\hspace*{1.5em}[x,y] \in \bigSBI{P_n \cup Q_n} \setminus \bigl( \BIclass{y_3, y_4} \cup \BIclass{y_5, y_6} \bigr)\\
&\hspace*{17em}\Bigr\} \subset I\setminus\Int\BIgraph{y_3, y_4},\\
& \set{L_i := \xi^{-1}([s_i,s_{i+1}])}{i\in\{0,1,\dots,m-1\}} \subset \BIgraph{y_3, y_4}, \text{ and}\\
&\ \{U_0, U_1,\dots,U_t\} = \set{\varphi_{a,b}\bigl([s_i,s_{i+1}]\bigr)}{i\in\{0,1,\dots,m-1\}} \subset X.
\end{align*}

Next we will derive the Markov graph of $g_n$ with respect to $R_n$
(recall that it can be obtained from the Markov graph modulo 1 of
$F_n$ with respect to $P_n \cup Q_n,$
which coincides with the Markov graph of $f_n$ with respect to
$\bigemap{P_n \cup Q_n}$ provided that we identify
$\BIgraph{x,y}$ with $\bigemap{\BIclass{x,y}} = \emap{[x,y]}$ and this
with $\BIclass{x,y}$ for every $[x,y] \in \bigSBI{P_n \cup Q_n}$ ---
see Proposition~\ref{propositionmarkovgraphexampleBCNintroduction}(a)):
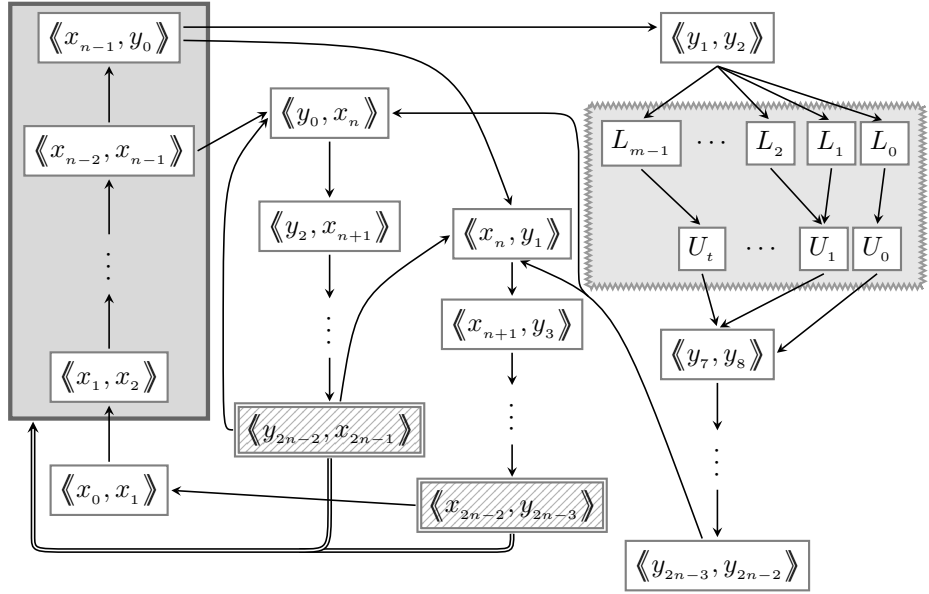
\begin{center}\small
\tikzstyle{place}=[rectangle,draw=black!50,fill=black!0,thick]
\tikzstyle{post}=[->,shorten >=1pt,>=stealth,semithick]
\medskip
\begin{tikzpicture}
\filldraw[thick, draw=black!40, fill=black!10, decorate, decoration={zigzag, amplitude=1pt, segment length=2pt}] (7.8,6.1) rectangle  (12.2,3.7);
\let\BIclass\BIgraph
\subfigurefngraphModOne
\node[place](l0)        at (11.7,5.6) {$L_0$};
\node[place](l1)        at (11,5.6)   {$L_1$};
\node[place](l2)        at (10.2,5.6) {$L_2$};
\node(dotsl)            at (9.45,5.6) {$\cdots$};
\node[place](lm1)       at (8.5,5.6)  {$L_{m-1}$};
\node[place](u0)        at (11.6, 4.2) {$U_0$};
\node[place](u1)        at (10.9,4.2)  {$U_1$};
\node[place](ur)        at (9.3,4.2)   {$U_t$};
\node(dotsu)            at (10.1,4.2)  {$\cdots$};
\path[post]
    (y1y2.south)+(0,-1pt) edge (lm1.north)
    (y1y2.south)+(0,-1pt) edge (l2.north)
    (y1y2.south)+(0,-1pt) edge (l1.north)
    (y1y2.south)+(0,-1pt) edge (l0.north)
    (lm1.south) +(0,-1pt) edge (ur.north)
    (l2.south)  +(0,-1pt) edge (u1.north)
    (l1.south)  +(0,-1pt) edge (u1.north)
    (l0.south)  +(0,-1pt) edge (u0.north)
    (u0.south)  +(0,-1pt) edge (y7y8.east)
    (u1.south)  +(0,-1pt) edge (y7y8.north)
    (ur.south)  +(0,-1pt) edge (y7y8.north);
\end{tikzpicture}
\end{center}
(the part of the Markov graph of $g_n$ with respect to $R_n$
which differs from the Markov graph modulo 1 of $F_n$ with respect to
$P_n \cup Q_n$ is shown inside a grey box with a zigzag border).

As before, by
Lemma~\ref{convertingtoexpansiveMM} and
Theorem~\ref{theoremTransitivityexpansivemap},
the map $g_n$ can be modified
without altering $g_n\evalat{R_n}$ and $g_n(K)$
for every $K \in \SBI[R_n]$ to become
$R_n$-expansive and transitive.
Moreover,
\[ \Per(f_n) = \succs{n} = \bigcup_{w\in\succs{n}} w\cdot\N = \Per(g_n) \]
(see the proof of Theorem~\ref{theoremfirstexamplebcngraph} but
here,  as in the previous example, the situation is much simpler
because $2 \notin \Per(f_n)$).
Consequently, $\Per(g_n)$ is cofinite and,
by Theorem~\ref{theoremtotallytransitivefromadrr},
$g_n$ is totally transitive.

Now we will estimate $h(g_n)$ with the same techniques as before
to show that $\lim_{n\to\infty} h(g_n) = 0.$
We use the rome which corresponds to the one used in the proof of
Proposition~\ref{propositionmarkovgraphexampleBCNintroduction}(b):
\[
\widetilde{\textsf{Rom}}_n = \{\tilde{\textsf{r}}_1 = \BIgraph{x_{2n-2},y_{2n-3}}, \tilde{\textsf{r}}_2 = \BIgraph{y_{2n-2},x_{2n-1}}\}.
\]
Then, we see by direct inspection that the matrix $M_{\widetilde{\textsf{Rom}}_n}(x)$ is:
{\small\begin{multline*}
 \hspace*{-1em}\begin{pmatrix*}[r]
   \sum\limits_{i=0}^{n-1} x^{-(n+i)} + m\mkern-5mu \sum\limits_{i=0}^{n-1} x^{-(2n-1+i)}             & \sum\limits_{i=1}^{n-1}x^{-(n+i)} + m\mkern-5mu\sum\limits_{i=0}^{n-1} x^{-(2n+i)}\\[3ex]
   x^{-(n-1)} + \sum\limits_{i=0}^{n-2} x^{-(n+i)} + m\mkern-5mu\sum\limits_{i=0}^{n-2} x^{-(2n-1+i)} & x^{-n} + \sum\limits_{i=1}^{n-2}x^{-(n+i)} + m\mkern-5mu\sum\limits_{i=0}^{n-2} x^{-(2n+i)}
\end{pmatrix*}\\
= \begin{pmatrix*}[r]
   \sum\limits_{i=0}^{n-1} x^{-(n+i)} + m \sum\limits_{i=-1}^{n-2} x^{-(2n+i)} & \sum\limits_{i=1}^{n-1}x^{-(n+i)} + m\sum\limits_{i=0}^{n-1} x^{-(2n+i)}\\[3ex]
   \sum\limits_{i=-1}^{n-2} x^{-(n+i)} + m\sum\limits_{i=-1}^{n-3} x^{-(2n+i)} & \sum\limits_{i=0}^{n-2}x^{-(n+i)} + m\sum\limits_{i=0}^{n-2} x^{-(2n+i)}
\end{pmatrix*}
\end{multline*}}
By Theorem~\ref{theoremrome}, the characteristic polynomial of
the Markov matrix of $g_n$ is
\[
 (-1)^{4n-5+m+t} x^{4n-3+m+t} \det\bigl(M_{\widetilde{\textsf{Rom}}_n}(x) - \mathbf{I}_{2}\bigr)
    = \pm x^{m+t-1}\frac{\widetilde{T}_n(x)}{x-1}
\]
with
\[
  \widetilde{T}_n(x)  = \bigl(x^{4n-2}-m\bigr)(x-1) - x^{2n-1}\bigl(2x^{n}-x-1\bigr) - m x^{n}\bigl(x^{n-1}(x+1)-2\bigr).
\]
To show that $\lim_{n\to\infty} h(f_n) = 0$, as we did before,
we consider the equation
\begin{multline*}
 x^{4n-2}(x-1) = x^{2n-1}\bigl(2x^{n}-x-1\bigr) + m x^{n}\bigl(x^{n-1}(x+1)-2\bigr) + m(x-1)\\
 \Longleftrightarrow\qquad
 x^{n-1} = \frac{2x^{n}-x-1}{x^n(x-1)} + m \frac{x^{n-1}(x+1)-2}{x^{2n-1}(x-1)} + \frac{m}{x^{3n-1}}
\end{multline*}
for $x > 1.$ Moreover,
\begin{multline*}
 \frac{2x^{n}-x-1}{x^n(x-1)} + m \frac{x^{n-1}(x+1)-2}{x^{2n-1}(x-1)} + \frac{m}{x^{3n-1}} <\\[1ex]
 \frac{\tfrac{2x^{n}-x-1}{x^{n}} + 2m \tfrac{x^{n}-1}{x^{2n-1}} + m \tfrac{x-1}{x^{3n-1}}}{x-1} <
 \frac{2 + 2m + m}{x-1}.
\end{multline*}
So, as in the proof of Theorems~\ref{theoremfirstexamplebcncircle}
and~\ref{theoremexampleBCNintroduction}
(see the figure and the arguments in pages~\pageref{fig:arrels}
and~\pageref{zerosarguments}),
for every $n \ge 5$ there exists a unique real
number $\gamma_n > 1$ such that
$\gamma_n^{n-1} = \tfrac{3m+2}{\gamma_n -1}$ and
$x^{n-1} > \tfrac{3m+2}{x-1}$ for every $x > \gamma_n,$
the sequence $\{\gamma_n\}_n$ is strictly decreasing and
$\lim_{n\to\infty} \gamma_n = 1.$
Thus, for every $x > \gamma_n,$
\begin{multline*}
 \frac{x^{4n-2}(x-1)}{x^{2n-1}\bigl(2x^{n}-x-1\bigr) + m x^{n}\bigl(x^{n-1}(x+1)-2\bigr) + m(x-1)} = \\
 \frac{x^{n-1}}{\tfrac{2x^{n}-x-1}{x^n(x-1)} + m \tfrac{x^{n-1}(x+1)-2}{x^{2n-1}(x-1)} + \tfrac{m}{x^{3n-1}}} >
 \frac{x^{n-1}}{\tfrac{3m+2}{x-1}} > 1,
\end{multline*}
which implies that $\widetilde{T}_n(x) > 0$ for every $x > \gamma_n.$
Hence,
if $\tilde{\rho}_n > 1$ denotes the largest root of $\widetilde{T}_n(x),$
it follows that $\tilde{\rho}_n \le \gamma_n$ and, consequently,
\[
 \lim_{n\to\infty} h(f_n) = \lim_{n\to\infty} \log\tilde{\rho}_n  \le \lim_{n\to\infty} \log\gamma_n = 0
\]
by Proposition~\ref{propositionmonotoneTopEnt}.
\end{proof}

\section{Proof of Theorem~\ref{MTS1}}\label{sectiontheorem}

\begin{proof}[Proof of Theorem~\ref{MTS1}]
Fix $L\in \N,$ $L > 8.$
Since $\lim_{n\to\infty} h(f_n) = 0,$
there exists $N\in \N$ such that
\[
   h(f_n) < \frac{3\log \sqrt{2}}{L}.
\]
for every $n \ge N.$
In the rest of the proof we consider a fixed but arbitrary $n \ge N$
and we denote $\Rot(F_n) = [c_n, d_n].$

We claim that
\begin{equation}\label{LaPartM}
  M(c_n,d_n) \subset \succs{L+1} = \set{k\in \N}{k \ge L+1}.
\end{equation}
To prove this note that for every $q \in M(c_n,d_n)$
there exists  $\tfrac{r}{s} \in (c_n,d_n)$ with $r\in\Z$ and $s\in \N$
coprime such that $q = \ell s$ with $\ell \in \N.$
In this situation,
$h(f_n) \ge \tfrac{\log 3}{s}$ by \cite[Corollary~4.7.7]{alm}.
Hence,
\[
 \frac{\log 3}{q} \le \frac{\log 3}{s} \le h(f_n) <
           \frac{3\log \sqrt{2}}{L} < \frac{\log 3}{L}.
\]
Consequently, $q > L$ and the claim holds.

From the claim we get that
$
\Int\left(\Rot(F_n)\right) \cap \set{k/L}{k \in \Z} = \emptyset.
$
This implies that
$\len\left(\Rot(F_n)\right) \le 1/L$ for every $n \ge N.$
So, it follows that $\lim_{n\to\infty} \len\left(\Rot(F_n)\right) = 0.$

By Theorem~\ref{theoremMisiurewicz} and the above claim,
\begin{equation}\label{SetOfperiods}
\begin{split}
 \Per(F_n) & = Q_{F_n}(c_n) \cup M(c_n,d_n) \cup Q_{F_n}(d_n)\\
           & \subset \succs{L+1} \cup Q_{F_n}(c_n) \cup Q_{F_n}(d_n).
\end{split}
\end{equation}

In view of the above inclusion for the set $\Per(f_n)$ we need to study
the intersections
\[
  \{1,2,\dots,L\} \cap Q_{F_n}(c_n)
  \andq
  \{1,2,\dots,L\} \cap Q_{F_n}(d_n).
\]
We will divide this study in three claims, according to
different situations for $c_n$ and $d_n.$

\begin{autocase}[Claim]{1} If $\alpha \notin \Q$ then  $\{1,2,\dots,L\} \cap Q_{F_n}(\alpha) = \emptyset.$\end{autocase}
This claim follows immediately from the definition of $Q_{F_n}(\alpha) = \emptyset.$

\begin{autocase}[Claim]{2}
Assume that $\alpha = \tfrac{r}{s}$ with $r\in\Z$ and $s\in \N$ coprime, and $s \ge L.$
Then, $\{1,2,\dots,L\} \cap Q_{F_n}(\alpha) \subset \{L\} \cap \{s\}.$
\end{autocase}
Again by the definition of $Q_{F_n}(\alpha),$ in this case we have
\[
  Q_{F_n}(\alpha) = \set{s k}{k\in\N \text{ and } k \leso{\Sho} s_{\alpha}} \subset s\N = \set{s k}{k\in\N}.
\]
Since $s \ge L,$ for every $k\in\N,\ k \ge 2$ we have
$
 sk \ge 2L > L.
$
Hence,
\[
  \{1,2,\dots,L\} \cap Q_{F_n}(\alpha) \subset \{1,2,\dots,L\} \cap \set{s k}{k\in\N} = \{L\} \cap \{s\}.
\]

\begin{autocase}[Claim]{3}
Assume that $\alpha = \tfrac{r}{s}$ with $r\in\Z$ and $s\in \N$ coprime, and $s < L.$
Then, $\Card\left(\{L-2, L-1, L\} \cap Q_{F_n}(\alpha)\right) \le 1$ and
\[
  \{1,2,\dots,L-1\} \cap Q_{F_n}(\alpha) \subset \set{s \cdot 2^\ell}{\ell\in\{0,1,2,\dots,m\}}
\]
where $m \ge 0$ is the integer part of $\log_2\left(\tfrac{L-1}{s}\right).$
\end{autocase}
To prove this claim assume first that $Q_{F_n}(\alpha)$ contains
an element of the form
$s \cdot t \cdot 2^\ell$ with $t \ge 3$ odd and $\ell \in \Z^+.$
From the definition of $Q_{F_n}(\alpha)$ it follows that then
the map $F_n^{s}-r$ has a periodic point of period $t \cdot 2^\ell$
(as a map of the real line).
Hence, by Lemmas~4.4.15 and 4.4.16 and Theorem~3.12.17 of \cite{alm}
(see also \cite[page~264]{alm}),
\[
  h(f_n) = \frac{1}{s} h(f_n^{s}) \ge \frac{1}{s} \frac{1}{2^\ell} \log \lambda_t
\]
where $\lambda_t$ is the largest root of the polynomial $x^t - 2x^{t-2} - 1.$
It is well known that $\lambda_t  > \sqrt{2}$ (see \cite[page~232]{alm}). So,
\[
 \frac{3\log \sqrt{2}}{L}> h(f_n) >\frac{\log \sqrt{2}}{s 2^\ell}
\]
which implies $s \cdot t \cdot 2^\ell \ge s\cdot 3\cdot2^\ell > L.$
So, for every set $\mathsf{A} \subset \{1,2,\dots,L\},$
\begin{equation}\label{inclusions}
\begin{split}
  \mathsf{A} \cap Q_{F_n}(\alpha)
     &\subset \mathsf{A} \cap s \N \\
     &= \mathsf{A} \cap \left(s \N\setminus \set{s \cdot t \cdot 2^\ell}{\ell\in\Z^+ \text{ and } t \ge 3 \text{ odd}}\right) \\
     &= \mathsf{A} \cap \set{s \cdot 2^\ell}{\ell\in\Z^+}.
\end{split}
\end{equation}

Since $s < L$ and $m$ is the integer part of
$\log_2\left(\tfrac{L-1}{s}\right)$ it follows that $m \ge 0$ and
\begin{equation}\label{counting}
  2^m \le \tfrac{L-1}{s} < 2^{m+1}.
\end{equation}
Then, from \eqref{inclusions} with $\mathsf{A} = \{1,2,\dots,L-1\}$ we obtain
\begin{multline*}
  \{1,2,\dots,L-1\} \cap Q_{F_n}(\alpha)
      \subset \{1,2,\dots,L-1\} \cap \set{ s \cdot 2^\ell}{\ell\in\Z^+} \\
      = \set{s \cdot 2^\ell}{\ell\in\{0,1,2,\dots,m\}},
\end{multline*}
which proves the second statement of the claim.

Now we will prove the first one.
We start by assuming that $s \cdot 2^m \le L-3.$
From \eqref{counting} we have
\[
 s \cdot 2^{m+2} = 2 \left(s \cdot 2^{m+1} \right) \ge 2L > L.
\]
Consequently, by \eqref{inclusions} with $\mathsf{A} = \{L-2, L-1, L\},$
\begin{multline*}
  \{L-2, L-1, L\} \cap Q_{F_n}(\alpha) \subset \\
      \{L-2, L-1, L\} \cap \set{s \cdot 2^\ell}{\ell\in\Z^+} \subset \\
      \{L-2, L-1, L\} \cap \{s \cdot 2^{m+1}\}.
\end{multline*}
Now we assume that $s \cdot 2^m \in \{L-2, L-1\}.$
Then,
\[
 s \cdot 2^{m+1} = 2 \left(s \cdot 2^m \right) \ge 2(L-2) = L + (L-4) > L + 4
\]
because $L > 8.$
Consequently, again by \eqref{inclusions} with $\mathsf{A} = \{L-2, L-1, L\},$
\begin{multline*}
  \{L-2, L-1, L\} \cap Q_{F_n}(\alpha) \subset \\
      \{L-2, L-1, L\} \cap \set{s \cdot 2^\ell}{\ell\in\Z^+} = \\
      \{L-2, L-1, L\} \cap \set{s \cdot 2^\ell}{\ell\in\{0,1,2,\dots,m\}} = \\
      \{L-2, L-1, L\} \cap \{s \cdot 2^m\}
\end{multline*}
whenever $m=0$ or $m > 0 $ and $s \cdot 2^{m-1} \le L-3.$
Now we have to show that we cannot simultaneously have
$m > 0$ and $s \cdot 2^{m-1} \ge L-2.$
Otherwise, as above,
\[
  L-1 \ge s \cdot 2^m = 2\left(s \cdot 2^{m-1}\right) \ge 2(L-2) > L + 4;
\]
a contradiction.
This ends the proof of Claim~3.

From the above three claims we obtain
\begin{equation}\label{toquote}
  \Card\left(\{L-2, L-1, L\} \cap \left( Q_{F_n}(c_n) \cup Q_{F_n}(d_n)\right)\right) \le 2
\end{equation}
and, consequently, $\{L-2, L-1, L\} \not\subset Q_{F_n}(c_n) \cup Q_{F_n}(d_n).$
Thus,
\[
   \{L-2, L-1, L\} \not\subset \Per(f_n)
\]
by \eqref{SetOfperiods}. So, for every $n \ge N$ we set
\begin{align*}
 \kappa_n & := \min \left( \{L-2, L-1, L\} \setminus \Per(f_n) \right),\text{ and}\\
 \nu_n & := \min \left( \Per(f_n) \cap \succs{\kappa_n + 1} \right).
\end{align*}

The inequality~\eqref{toquote} is crucial for this
proof. It allows us to define
$\kappa_n$ and, hence, $\nu_n$ and tells us that $\sbc(f_n) \ge \nu_n$
(because, as we will see, $\nu_n - 1 \notin \Per(f_n)$).
This is implicitly used in the rest of the proof of the theorem.

To end the proof of the theorem it is enough to show that
$\nu_n \in \sbcset(f_n)$ for every $n \ge N.$
Indeed, by Definition~\ref{SCB}, $\bc(f_n)$ exists and
\begin{equation}\label{oftenusefulbound}
 L-1 \le \kappa_n + 1 \le \nu_n \le \bc(f_n)
\end{equation}
for every $n \ge N.$
Consequently, $\lim_{n\to\infty} \bc(f_n) = \infty.$

Let us prove that $\nu_n \in \sbcset(f_n)$ for every $n \ge N.$
By Definition~\ref{SCB} we have to show that $\nu_n \in \Per(f),$
$\nu_n > 2,$ $\nu_n-1 \notin \Per (f)$ and
\begin{equation}\label{TheDensity}
\Card\bigl(\{1,\ldots, \nu_n-2\}\cap \Per(f)\bigr) \le 2 \log_2(\nu_n-2).
\end{equation}
Since $L > 8,$ from the definition of $\nu_n$ we get
\[
 7 < L-1 \le \nu_n \in \Per(f_n).
\]

The following claim will be useful in the rest of the proof.
It improves the knowledge of the set
$\{1,\ldots, \nu_n-2\}\cap \Per(f).$
\begin{autocase}[Claim]{4}
$\{\kappa_n, \kappa_n + 1,\dots, \nu_n-1\} \cap \Per(f_n) = \emptyset.$
\end{autocase}
When $\nu_n = \kappa_n + 1,$ the claim holds because
$\nu_n - 1 = \kappa_n \notin \Per(f_n)$
by the definition of $\kappa_n.$
Now we prove the claim in the case $\nu_n > \kappa_n + 1.$
We have
$\{\kappa_n + 1,\dots, \nu_n-1\} \subset \succs{\kappa_n + 1}$
and, hence,
\[ \{\kappa_n + 1,\dots, \nu_n-1\} \cap \Per(f_n) = \emptyset \]
by the minimality of $\nu_n.$
Moreover, $\kappa_n \notin \Per(f_n)$ by definition.
Hence,
\[
   \{\kappa_n, \kappa_n + 1,\dots, \nu_n-1\} \cap \Per(f_n) = \emptyset,
\]
which ends the proof of Claim~4.

Claim~4, in particular, tells us that $\nu_n-1 \notin \Per(f_n).$
Hence, to end the proof of $\nu_n \in \sbcset(f_n),$
we have to prove the inequality \eqref{TheDensity}.
By Claim~4, \eqref{SetOfperiods} and \eqref{LaPartM}
(notice that $\kappa_n - 1 \le L-1$ because, by definition, $\kappa_n \le L$),
\begin{multline*}
  \Card\left(\{1,\dots, \nu_n-2\} \cap \Per(f_n)\right) = \\
  \Card\left(\{1,\dots, \kappa_n-1\} \cap \Per(f_n)\right) = \\
  \Card\left(\{1,\dots, \kappa_n-1\} \cap \left(Q_{F_n}(c_n) \cup Q_{F_n}(d_n)\right)\right) \le \\
  \Card\left(\{1,\dots, \kappa_n-1\} \cap Q_{F_n}(c_n) \right) + \Card\left(\{1,\dots, \kappa_n-1\} \cap Q_{F_n}(d_n)\right).
\end{multline*}
So, to prove \eqref{TheDensity} it is enough to show that
\begin{equation}\label{finaldensitybound}
\begin{split}
\Card\left(\{1,\dots, \kappa_n-1\} \cap Q_{F_n}(c_n) \right) +
  & \Card\left(\{1,\dots, \kappa_n-1\} \cap Q_{F_n}(d_n)\right) \\
  & \le 2 \log_2(\nu_n-2).
\end{split}
\end{equation}
To this end, we have to compute appropriate upper bounds of
the two summands in the last expression.

Again, let $\alpha \in \{c_n, d_n\}$ denote
an arbitrary endpoint of $\Rot(F_n).$
In the assumptions of Claims~1 and 2 we have
\begin{equation}\label{densityboundCases12}
  \Card\left(\{1,\dots, \kappa_n-1\} \cap Q_{F_n}(\alpha)\right)
    \le \Card\left(\{1,\dots, L\} \cap Q_{F_n}(\alpha)\right) =\emptyset.
\end{equation}
Now suppose that the assumptions of Claim~3 hold.
We want to prove the following estimate:
\begin{equation}\label{densityboundCase3}
\begin{split}
  \Card\left(\{1,\dots, \kappa_n-1\} \cap Q_{F_n}(\alpha)\right)
    & \le \log_2\left(\tfrac{\nu_n - 2}{s}\right) + 1 \\
    & \le \log_2\left(\nu_n - 2\right) + 1.
\end{split}
\end{equation}
Assume first that $s \cdot 2^m \le \nu_n - 2.$
Then, by the second statement of Claim~3, we get
\begin{multline*}
  \Card\left(\{1,\dots, \kappa_n-1\} \cap Q_{F_n}(\alpha)\right) \le\\
  \Card\left(\{1,\dots, L-1\} \cap Q_{F_n}(\alpha)\right) \le m+1
     \le \log_2\left(\tfrac{\nu_n - 2}{s}\right) + 1.
\end{multline*}
Now assume that $\nu_n - 2 < s \cdot 2^m.$
By \eqref{oftenusefulbound} and \eqref{counting},
\[ L - 3 \le \kappa_n-1 \le \nu_n - 2 < s \cdot 2^m \le L-1. \]
Consequently, by \eqref{inclusions},
\begin{multline*}
  \Card\left(\{1,\dots, \kappa_n-1\} \cap Q_{F_n}(\alpha)\right) \le \\
  \Card\left(\{1,\dots, \kappa_n-1\} \cap \set{s \cdot 2^\ell}{\ell\in\Z^+}\right) \le \\
   \Card\set{s \cdot 2^\ell}{\ell\in\{0,1,2,\dots,m-1\}} = m.
\end{multline*}
Moreover, $s \cdot 2^{m-1} \le \nu_n - 2$ since, otherwise,
from the above inequalities and using again the fact that $L > 8$
we obtain
\[
 L + 2 < L + (L - 6) = 2(L-3) \le 2\left(\nu_n - 2\right) < 2\cdot s \cdot 2^{m-1} = s \cdot 2^m \le L-1;
\]
a contradiction.
So, $m-1 \le \log_2\left(\tfrac{\nu_n - 2}{s}\right).$
Putting all together we get,
\[
  \Card\left(\{1,\dots, \kappa_n-1\} \cap Q_{F_n}(\alpha)\right) \le m \le \log_2\left(\tfrac{\nu_n - 2}{s}\right) + 1.
\]
This ends the proof of \eqref{densityboundCase3}.

Now we are ready to prove \eqref{finaldensitybound}.
First assume that at most one of the endpoints of $\Rot(F_n)$
satisfies the assumptions of Claim~3.
By \eqref{densityboundCases12} and \eqref{densityboundCase3},
\begin{multline*}
  \Card\left(\{1,\dots, \kappa_n-1\} \cap Q_{F_n}(c_n) \right) + \Card\left(\{1,\dots, \kappa_n-1\} \cap Q_{F_n}(d_n)\right) \le \\
  \log_2\left(\nu_n - 2\right) + 1 < 2\log_2\left(\nu_n - 2\right)
\end{multline*}
because $7 < \nu_n$ implies $\log_2(\nu_n-2) > 1.$

It remains to consider the case when both endpoints of
$\Rot(F_n) = [c_n, d_n]$ satisfy the assumptions of Claim~3.
That is,
$c_n = \tfrac{r_n}{s_n}$ with $r_n\in\Z$ and $s_n\in \N$ coprime,
$d_n = \tfrac{q_n}{t_n}$ with $q_n\in\Z$ and $t_n\in \N$ coprime, and $s_n,t_n \le L-1.$
Observe that if $s_n,t_n \le 3,$ from above and the fact that $L > 8$ we get
\[
  \tfrac{1}{6} \le d_n - c_n = \len\left(\Rot(F_n)\right) \le \tfrac{1}{L} < \tfrac{1}{8};
\]
a contradiction. Hence, either $s_n$ or $t_n$ is larger than 3.
Assume for definiteness that $s_n \ge 4.$
Then, by \eqref{densityboundCase3},
\begin{multline*}
  \Card\left(\{1,\dots, \kappa_n-1\} \cap Q_{F_n}(c_n) \right) + \Card\left(\{1,\dots, \kappa_n-1\} \cap Q_{F_n}(d_n)\right) \le \\
  \log_2\left(\tfrac{\nu_n - 2}{s_n}\right) + \log_2(\nu_n - 2) + 2 \le \\
     \log_2\left(\tfrac{\nu_n - 2}{4}\right) + \log_2(\nu_n - 2) + 2 =  
     2\log_2(\nu_n - 2).
\end{multline*}
This ends the proof of \eqref{finaldensitybound} and, hence,
that $\nu_n \in \sbcset(f_n).$
\end{proof}

\bibliographystyle{plain}
\bibliography{CofinitenessBoundary-ABG}

\def\soft#1{\leavevmode\setbox0=\hbox{h}\dimen7=\ht0\advance \dimen7
  by-1ex\relax\if t#1\relax\rlap{\raise.6\dimen7
  \hbox{\kern.3ex\char'47}}#1\relax\else\if T#1\relax
  \rlap{\raise.5\dimen7\hbox{\kern1.3ex\char'47}}#1\relax \else\if
  d#1\relax\rlap{\raise.5\dimen7\hbox{\kern.9ex \char'47}}#1\relax\else\if
  D#1\relax\rlap{\raise.5\dimen7 \hbox{\kern1.4ex\char'47}}#1\relax\else\if
  l#1\relax \rlap{\raise.5\dimen7\hbox{\kern.4ex\char'47}}#1\relax \else\if
  L#1\relax\rlap{\raise.5\dimen7\hbox{\kern.7ex
  \char'47}}#1\relax\else\message{accent \string\soft \space #1 not
  defined!}#1\relax\fi\fi\fi\fi\fi\fi}
\begin{thebibliography}{10}

\bibitem{AKM}
R.~L. Adler, A.~G. Konheim, and M.~H. McAndrew.
\newblock Topological entropy.
\newblock {\em Trans. Amer. Math. Soc.}, 114:309--319, 1965.

\bibitem{arr}
Ll. Alsed{\`a}, M.~A. del R{\'{\i}}o, and J.~A. Rodr{\'{\i}}guez.
\newblock A splitting theorem for transitive maps.
\newblock {\em J. Math. Anal. Appl.}, 232(2):359--375, 1999.

\bibitem{adrr}
Ll. Alsed{\`a}, M.~A. del R{\'{\i}}o, and J.~A. Rodr{\'{\i}}guez.
\newblock A note on the totally transitive graph maps.
\newblock {\em Internat. J. Bifur. Chaos Appl. Sci. Engrg.}, 11(3):841--843,
  2001.

\bibitem{ARR2003-survey}
Ll. Alsed{\`a}, M.~A. Del~R{\'{\i}}o, and J.~A. Rodr{\'{\i}}guez.
\newblock A survey on the relation between transitivity and dense periodicity
  for graph maps.
\newblock {\em J. Difference Equ. Appl.}, 9(3-4):281--288, 2003.
\newblock Dedicated to Professor Alexander N. Sharkovsky on the occasion of his
  65th birthday.

\bibitem{adrr2}
Ll. Alsed{\`a}, M.~A. del R{\'{\i}}o, and J.~A. Rodr{\'{\i}}guez.
\newblock Transitivity and dense periodicity for graph maps.
\newblock {\em J. Difference Equ. Appl.}, 9(6):577--598, 2003.

\bibitem{ablm}
Lluis Alsed{\`a}, Stewart Baldwin, Jaume Llibre, and Micha{\l} Misiurewicz.
\newblock Entropy of transitive tree maps.
\newblock {\em Topology}, 36(2):519--532, 1997.

\bibitem{almm}
Llu{\'{\i}}s Alsed{\`a}, Jaume Llibre, Francesc Ma{\~n}osas, and Micha{\l}
  Misiurewicz.
\newblock Lower bounds of the topological entropy for continuous maps of the
  circle of degree one.
\newblock {\em Nonlinearity}, 1(3):463--479, 1988.

\bibitem{alm}
Llu{\'{\i}}s Alsed{\`a}, Jaume Llibre, and Micha{\l} Misiurewicz.
\newblock {\em Combinatorial dynamics and entropy in dimension one}, volume~5
  of {\em Advanced Series in Nonlinear Dynamics}.
\newblock World Scientific Publishing Co., Inc., River Edge, NJ, second
  edition, 2000.

\bibitem{BBCDS}
J.~Banks, J.~Brooks, G.~Cairns, G.~Davis, and P.~Stacey.
\newblock On {D}evaney's definition of chaos.
\newblock {\em Amer. Math. Monthly}, 99(4):332--334, 1992.

\bibitem{bc}
Louis Block and Ethan~M. Coven.
\newblock Topological conjugacy and transitivity for a class of piecewise
  monotone maps of the interval.
\newblock {\em Trans. Amer. Math. Soc.}, 300(1):297--306, 1987.

\bibitem{bgmy}
Louis Block, John Guckenheimer, Micha{\l} Misiurewicz, and Lai~Sang Young.
\newblock Periodic points and topological entropy of one-dimensional maps.
\newblock In {\em Global theory of dynamical systems ({P}roc. {I}nternat.
  {C}onf., {N}orthwestern {U}niv., {E}vanston, {I}ll., 1979)}, volume 819 of
  {\em Lecture Notes in Math.}, pages 18--34. Springer, Berlin, 1980.

\bibitem{Bkh83}
A.~M. Blokh.
\newblock On transitive mappings of one-dimensional branched manifolds.
\newblock In {\em Differential-difference equations and problems of
  mathematical physics ({R}ussian)}, pages 3--9, 131. Akad. Nauk Ukrain. SSR,
  Inst. Mat., Kiev, 1984.

\bibitem{Bkh87}
A.~M. Blokh.
\newblock The connection between entropy and transitivity for one-dimensional
  mappings.
\newblock {\em Uspekhi Mat. Nauk}, 42(5(257)):209--210, 1987.

\bibitem{g}
F.~R. Gantmacher.
\newblock {\em The theory of matrices. {V}ol. 1}.
\newblock AMS Chelsea Publishing, Providence, RI, 1998.
\newblock Translated from the Russian by K. A. Hirsch, Reprint of the 1959
  translation.

\bibitem{Ito}
R.~Ito.
\newblock Rotation sets are closed.
\newblock {\em Math. Proc. Cambridge Philos. Soc.}, 89(1):107--111, 1981.

\bibitem{KS}
Sergi{\u\i} Kolyada and {\soft{L}}ubom{\'{\i}}r Snoha.
\newblock Some aspects of topological transitivity---a survey.
\newblock In {\em Iteration theory ({ECIT} 94) ({O}pava)}, volume 334 of {\em
  Grazer Math. Ber.}, pages 3--35. Karl-Franzens-Univ. Graz, Graz, 1997.

\bibitem{m}
Micha{\l} Misiurewicz.
\newblock Periodic points of maps of degree one of a circle.
\newblock {\em Ergodic Theory Dynamical Systems}, 2(2):221--227 (1983), 1982.

\end{thebibliography}
\end{document}